\documentclass{amsart}
\usepackage[style=alphabetic]{biblatex} 
\usepackage{scalerel}

\calclayout

\usepackage{amscd,verbatim}
\usepackage{amssymb}
\usepackage[all]{xy}
\usepackage{xcolor}
\usepackage{preample}
\usepackage[nopatch=eqnum]{microtype}
\usepackage{hyperref}
\hypersetup{
 colorlinks,
 linkcolor={teal},
 citecolor={teal},
 urlcolor={teal}
}
\DeclareFieldFormat
  [article,book,inbook,incollection,inproceedings,patent,thesis,unpublished]
  {title}{\emph{#1\isdot}}
 
 \renewcommand{\bcube}{{\overline{\square}}}

\let\Bl=\relax

\let\PSh=\relax

\let\CI=\relax
\let\ch=\relax
\DeclareMathOperator{\CDiv}{CDiv}

\DeclareMathOperator{\Bl}{Bl}

\DeclareMathOperator{\PSh}{PSh}
\DeclareMathOperator{\Sh}{Sh}
\DeclareMathOperator{\CI}{CI}

\DeclareMathOperator{\Fil}{Fil}
\DeclareMathOperator{\ch}{ch}

\def\int{\mathrm{int}}

\def\ceil#1{\lceil #1 \rceil}

\def\inj{\mathrm{inj}}

\def\category#1{\operatorname{\mathrm{#1}}}
\def\Sm{\category{Sm}}

\def\Sch{\category{Sch}}
\def\Comp{\category{Comp}}

\def\Cor{\category{Cor}}

\def\MCor{\category{MCor}}

\def\CAlg{\category{CAlg}}
\def\RSC{\category{RSC}}
\def\Ab{\category{Ab}}
\def\Mod{\category{Mod}}

\def\Spc{\category{Spc}}

\def\Sp{\category{Sp}}

\def\ulMCor{\category{\underline{M}Cor}}
\def\ulMSm{\category{\underline{M}Sm}}

\def\Witt{\mathrm{W}}
\def\ulM{\underline{\mathrm{M}}}

\def\CI{\mathrm{CI}}
\def\BI{\mathrm{BI}}
\def\HI{\mathrm{HI}}
\def\PrL{\category{Pr}^{\mathrm{L}}}
\def\Conn{\mathrm{Conn}}

\def\DA{\category{DA}}
\def\mSm{\category{mSm}}

\def\mDA{\category{mDA}}

\def\mH{\category{mH}}
\def\mSH{\category{mSH}}
\def\tDA{\category{tDA}}

\def\H{\category{H}}
\def\SH{\category{SH}}
\def\logSH{\category{logSH}}
\def\logH{\category{logH}}
\DeclareMathOperator{\fib}{fib}
\DeclareMathOperator{\cofib}{cofib}
\def\colim{\mathop{\mathrm{colim}}}
\DeclareMathOperator{\motive}{M}
\DeclareMathOperator{\MTh}{MTh}
\def\normal{\mathrm{N}}

\def\mot{\mathrm{mot}}
\DeclareMathOperator{\map}{map}
\def\lim{\mathop{\mathrm{lim}}}

\def\tDA{\category{tDA}}
\def\logDA{\category{logDA}}

\def\SmlSm{\category{SmlSm}}

\def\sNis{\mathrm{sNis}}

\def\exc{\mathrm{exc}}

\def\modulus{\mathrm{mod}}
\def\pt{\mathrm{pt}}
\def\bA{\mathbb{A}}
\def\SmS{\Sm_S}
\def\SmSNis{\Sm_{S,\Nis}}
\def\mSmS{\mSm_S}
\def\mSmSNis{\mSm_{S,\Nis}}

\def\MS{\category{MS}}
\def\ur{\mathrm{ur}}

\makeatletter
\def\quotprojlim{%
  \mathop{``\mathpalette\varlim@{\leftarrowfill@\textstyle}"}\nmlimits@
}
\def\quotinjlim{%
  \mathop{``\mathpalette\varlim@{\rightarrowfill@\textstyle}"}\nmlimits@
}
\makeatother

\begin{document}
\title{Motivic homotopy theory with ramification filtrations}

\author[J. Koizumi]{Junnosuke Koizumi}
\address{RIKEN iTHEMS, Wako, Saitama 351-0198, Japan}
\email{junnosuke.koizumi@riken.jp}

\author[H. Miyazaki]{Hiroyasu Miyazaki}
\address{Institute for Fundamental Mathematics, Communication Science Laboratories, NTT, Inc., 3-9-11 Midori-cho,Musashino-shi, Tokyo 180-8585, Japan}
\email{hiroyasu.miyazaki@ntt.com}

\author[S. Saito]{Shuji Saito}
\address{Graduate School of Mathematical Sciences, the University of Tokyo, 3-8-1 Komaba Meguro-ku Tokyo 153-8914, Japan}
\email{sshuji.goo@gmail.com}

\date{\today}
\thanks{The first (resp. second, resp. third) author is supported by JSPS KAKENHI Grant 22J20698 (resp. 21K13783 and 24K06699, resp. 23K20203 and 25K00904).
}

\begin{abstract}
We construct a generalization of Morel--Voevodsky's motivic homotopy theory that captures non-$\mathbb{A}^1$-homotopy-invariant phenomena, such as wild ramification and irregular singularity.
In the first part, we develop our motivic homotopy theory over quasi-compact and quasi-separated schemes, which satisfies the fundamental properties such as the projective bundle formula, the blow-up sequence, the Gysin sequence, and the Thom isomorphism when the base is normal.
Moreover, we compare our theory with existing frameworks. In particular, we recover Morel--Voevodsky's motivic homotopy category and Binda--Park--Østvær's logarithmic motivic homotopy category as reflective localizations of our category over normal bases. Furthermore, we construct adjoint functors connecting Annala--Iwasa’s category of motivic spectra with ours.
In the second part, we equip several non-$\mathbb{A}^1$-homotopy invariant cohomology theories, such as Hodge cohomology, Hodge--Witt cohomology, rank $1$ integrable connections, and unramified cohomology, with canonical filtrations that encode arithmetic and geometric information such as irregular singularities and wild ramification,
and prove that these cohomology theories with filtrations are representable in our motivic homotopy category.
We also compute some of those filtrations explicitly, and show that they recover known constructions, including a ramification filtration on the Pontryagin dual of the abelian \'etale fundamental group, and an irregularity filtration on the sheaf of rank $1$ connections.
\end{abstract}

\maketitle
\setcounter{tocdepth}{1}

\vspace{-5mm}

\tableofcontents

\enlargethispage*{20pt}
\thispagestyle{empty}

\def\cO{\mathcal{O}}
\def\qz{\Q/\Z}
\def\cont{\mathrm{cont}}
\def\uMHet{\underline{\mathrm{M}} \mathrm{H}_{\et}}
\def\Het{\mathrm{H}_{\et}}
\def\mHet{\mathrm{mH}_{\et}}
\def\Hur{\mathrm{H}_{\mathrm{ur}}}
\def\uMHur{\underline{\mathrm{M}} \Hur}
\def\mHur{\mathrm{m} \Hur}
\def\mDAeff{\mDA^\eff}
\def\logDAeff{\category{logDA}^{\eff}}
\def\triv{\mathrm{tri}}
\def\Log{\mathcal{L}og}
\def\logH{\category{logH}}
\def\Mlog{M^{\log}}
\def\cM{\mathcal{M}}
\def\cC{\mathcal{C}}

\section{Introduction}

Morel--Voevodsky's celebrated motivic homotopy theory is a universal framework connecting many (co)homological invariants appearing in arithmetic and algebraic geometry \cite{Morel-Voevodsky}. 
A key input of Morel--Voevodsky's theory is the \emph{$\AA^1$-homotopy invariance}. 
Roughly speaking, it is an algebro-geometric analogue of the notion of homotopy invariance in topology, which uses, instead of the unit interval $[0,1]$, the affine line $\AA^1$ as the parameter space of homotopies. 
This implies that the affine line becomes contractible in Morel--Voevodsky's theory, and hence it encodes only $\AA^1$-homotopy invariant phenomena. 

However, there are important phenomena which are captured only by \emph{non-$\AA^1$-homotopy invariant} cohomology theories. 
Motivated by this fact, in recent years, there have been several attempts to construct non-$\AA^1$-homotopy invariant generalizations of Morel--Voevodsky's theory, e.g., the logarithmic motivic homotopy theory by Binda--Park--Østvær \cite{LogHom}, and the theory of motivic spectra by Annala--Iwasa \cite{Annala-Iwasa}. 
These two theories are known to be interrelated, and to capture many important invariants, e.g., (log) topological Hochschild homology, (log) topological cyclic homology, which are far from being $\AA^1$-homotopy invariant. 

In the present paper, we construct another generalization of Morel--Voevodsky's theory, which we call the \emph{motivic homotopy theory with ramification filtration}. 
The terminology is inspired by the fact that this theory is closely related to ramification theory. In fact, the theory in the present paper relies on many ideas that  already existed in the logarithmic motivic homotopy theory, and therefore we can regard the former as a ramification-sensitive refinement of the latter. We will formulate this perspective as a mathematical statement (see \eqref{eq;logHmHAdunction} and \S \ref{sec:modlog}) under the assumption that the base scheme is normal.

The central object of study of the ramification theory is \emph{ramification filtration} on the Galois groups of extensions of discretely valued fields. 
Abbes--Saito (e.g., \cite{AS}) globalized it to a theory which controls ramification in coverings of higher dimensional varieties along \emph{simple normal crossing divisors with rational coefficients}. 
In essence, our theory is the motivic homotopy theory for schemes equipped with  simple normal crossing divisors with rational coefficients. 
To construct a homotopy theory, instead of the affine line $\AA^1$, we use the pair $(\PP^1,\infty)$ of projective line $\PP^1$ and the infinity point $\infty$ regarded as a divisor. 
The idea of using such pairs as above as the basis of the motivic theory had already appeared in \cite{KMSY1,KMSY2,KMSY3} and \cite{BPO,LogHom}, with slight differences. Indeed, the former papers use divisors with \emph{integral} coefficients, while the latter uses \emph{reduced} simple normal crossing divisors which encode the compactifying log structures.

As one of the main results of the present paper, we will prove that the stable logarithmic motivic homotopy category is a reflective localization of our stable motivic homotopy category with ramification filtrations when the base scheme is normal. 
We will summarize this point of view in more detail in \S \ref{intro;comparison} below. 

At several places in this article, we assume that the base scheme $S$ is normal. The only reason for this is that we do not know if Lemma \ref{lem:gamma-eq} holds when $S$ is not normal. See also Remark \ref{rem;normality}.

\medbreak

Given a new motivic theory, it is natural to ask what kind of cohomology theories it controls. 
In the present paper, we will show that 
\begin{itemize}
    \item many interesting cohomology theories are equipped with \emph{canonical filtrations}, and
    \item these cohomology theories with canonical filtrations are \emph{representable} in our motivic category. 
\end{itemize}
The examples are the cohomology of:
\begin{enumerate}
    \item all $\AA^1$-homotopy invariant sheaves with transfers (introduced by Voevodsky),
    \item the structure sheaf $\OO$,
    \item the K\"ahler differential $\Omega$ and its exterior products $\Omega^i$,
    \item the Hodge--Witt sheaves $W_m \Omega^i$,
    \item the sheaf of rank $1$ connections,
    \item the unramified cohomology $H_{\ur,m}^{q+1} = H_{\et}^{q+1}(-,\Z/m(q))$ (see Definition \ref{def;unramcoh}).
\end{enumerate}
In fact, all of the above are examples of reciprocity sheaves introduced in \cite{KSY1,KSY2}. It is a generalization of the $\AA^1$-homotopy invariant sheaves with transfers which played a central role in Voevodsky's theory of mixed motives \cite{voetri}. 
The theory of reciprocity sheaves has a connection with the ramification theory as shown in \cite{RSram1,RSram2,RSram3}, and the existence of a motivic theory that represents the cohomology of reciprocity sheaves had been expected. The first success in this direction was established in \cite{Sai_log}, whose main theorem states that the cohomology of reciprocity sheaves are representable in the category of logarithmic motives. 

\medbreak

A key new input in the present paper is, as stated above, that our motivic category also represents cohomology of all reciprocity sheaves equipped with the canonical \emph{filtrations}.
The canonical filtrations, which are constructed in \S \ref{sec:Fmod} of the present paper, recovers some important filtrations known in the previous studies. 
Below, we explain a remarkable example, \emph{ramification filtrations on the unramified cohomology}. 
The unramified cohomology is, by definition, the \'etale motivic cohomology with finite coefficient $H_{\ur,m}^{q+1} = H_{\et}^{q+1}(-,\Z/m(q))$. It has been studied by many authors in various contexts. 
Specifically, Kato defined ramification filtrations on $H_{\ur,m}^{q+1}$ in \cite{Kato89}. 
For example, $H_{\ur,m}^{1}$ is canonically identified with the Pontryagin dual of the abelian \'etale fundamental group:
\begin{equation}\label{Het}
\Het^1: U \mapsto \Het^1(U,\qz) \cong \Hom_{\cont}(\pi_1^{\mathrm{ab}} (U),\qz).
\end{equation}
One can prove that this presheaf on the category of smooth schemes over a field is actually a sheaf with respect to the Zariski (and even Nisnevich) topology, and furthermore, a reciprocity sheaf (In \S \ref{subsec;repfil}, we make a brief recollection of the definition of reciprocity sheaves). 

Thanks to Abbes--Saito's ramification theory and its abelian case previously established by Kato \cite{Kato89},
the right hand side of \eqref{Het} is equipped with a canonical filtration called the \emph{ramification filtration}
\begin{equation}\label{uMHet}
\uMHet^1(X,D)\subset \Het^1(U,\qz)
\end{equation}
parametrized by pairs $(X,D)$ with a $k$-scheme $X$ and a $\Q$-Cartier divisor $D$ on $X$ such that $U=X-|D|$
(see Definition \ref{def:BK-filtration}).
One of our results in the present paper states that the unramified cohomology $H_{\ur,m}^{1}=\Het^1$ with the ramification filtration \eqref{uMHet}, regarded as a (pre)sheaf on a suitable category of pairs $(X,D)$, is recovered by our universal construction of canonical filtrations on reciprocity sheaves and that it is representable in our motivic category with ramification filtrations.
We conjecture that the same is true for $H_{\ur,m}^{q+1}$ with Kato's ramification filtration (see Conjecture \ref{conj:unramified}). A positive answer to this will bring us a new motivic point of view on the unramified cohomology (see \S \ref{sec:Hone_exc}).
We note that the idea to use such pairs $(X,D)$ is also inspired by the higher dimensional class field theory in \cite{Kerz-Saito}, where Chow groups $\CH_0(X,D)$ of zero cycles with modulus play an essential role. 

\medbreak

Aside from the unramified cohomology, we also discuss several interesting examples.
For instance, Hodge cohomology $\mathrm{R}\Gamma({-}, \Omega^q)$ is not $\mathbb{A}^1$-invariant and hence cannot be handled in the traditional motivic homotopy theory. However, the Hodge cohomology with ramification filtration $\mathrm{R}\Gamma({-}, \ulM\Omega^q)$, introduced by Kelly--Miyazaki, is representable in our category (see Example \ref{ex;mOmega} below). The same holds also for the Hodge--Witt cohomology (see Example \ref{ex;mWOmega} below), so our theory has potential applications to  $p$-adic cohomology theories. 

Moreover, we will show that there is a canonical way to upgrade the Nisnevich cohomology $\mathrm{R}\Gamma({-},F)$ of any reciprocity sheaf $F$ in the sense of \cite{KSY1}, \cite{KSY2} to a cohomology theory with ramification filtration which is representable in our category (see \eqref{functor-mod}), and also show that the cohomology theories $\mathrm{R}\Gamma({-}, \ulM\Omega^q)$ and $\mathrm{R}\Gamma({-}, \uMHet^1)$ are special cases of the general construction.
Also, we will show that the filtration on the sheaf of rank $1$ connections, which captures the information of irregularity, is also recovered by this universal construction.
We hope that this universal construction, which recovers both the ramification filtration and the irregularity filtration, will provide a new insight about the analogy between the wild ramifications and the irregular singularities pointed out in previous researches, e.g., \cite{Deligne70}, \cite{Kato94}.

\medbreak

\subsection{Overview of the construction}

Now, we explain the construction of our motivic homotopy category with ramification filtrations, denoted $\mSH(S,\Lambda)$. 
In a nutshell, it is defined as the localizations with respect to the $(\PP^1,\infty)$-invariance (also called the cube-invariance) and the SNC blow-up invariance of the categories of Nisnevich sheaves on the category $\mSmS$ of \emph{log-smooth $\mathbb{Q}$-modulus pairs over $S$}, followed by a certain stabilization process:
A log-smooth $\mathbb{Q}$-modulus pair over $S$ is a pair $\mathcal{X}=(X,D)$ where $X\in \SmS$ and $D$ is a relative $\mathbb{Q}$-SNCD over $S$ (see Definition \ref{relativeSNCD}).
    A morphism of log-smooth $\mathbb{Q}$-modulus pairs $f\colon (X,D)\to (Y,E)$ is a morphism of $S$-schemes $f\colon X\to Y$ such that $D\geq f^*E$.
  The category $\mSmS$ has a natural symmetric monoidal structure $\otimes$ given by $(X,D)\otimes (Y,E)=(X\times Y, \pr_1^*D+\pr_2^*E)$.
We can also define a modulus version of the Nisnevich topology on $\mSmS$ (see Definition \ref{def:topology}).
    We then consider the following classes of morphisms in $\mSmS$ (see Definition \ref{def:SNC_blowup} for SNC blow-ups):
 \[ \mathrm{CI}=\{\mathcal{X}\otimes \bcube\to\mathcal{X}\mid \mathcal{X}\in \mSmS\},\quad 
        \mathrm{BI}=\{\text{SNC blow-ups }\mathcal{Y}\to\mathcal{X}\}.\]
    Choosing a coefficient category 
    $\mathcal{C}\in \{\Spc,\Spc_*,\Sp,\Mod_\Lambda\}$, we define the \emph{$\mathcal{C}$-valued motivic homotopy category with modulus}
    $
        \mH(S,\mathcal{C})
    $
    to be the full $\infty$-subcategory of the category $\Sh_\Nis(\mSmS,\mathcal{C})$ of $\mathcal{C}$-valued Nisnevich sheaves on $\mSmS$ spanned by 
    the $(\mathrm{CI}\cup \mathrm{BI})$-local objects.
    It is a standard fact that $\mH(S,\mathcal{C})$ underlies a presentably symmetric monoidal $\infty$-category, where the monoidal structure is the Day convolution of the tensor product $\otimes$ on $\mSmS$.
    It is equipped with a functor
    \begin{equation}\label{eq;Motivicfunctor}
        \motive\colon \mSm_S\xrightarrow{y} \Sh_\Nis(\mSm_S,\mathcal{C})\xrightarrow{\mathrm{L}_\mot}\mH(S,\mathcal{C}),\end{equation}
    where $y$ is the Yoneda functor and $\mathrm{L}_\mot\colon \Sh_\Nis(\mSmS,\mathcal{C}) \to \mH(S,\mathcal{C})$ is the localization functor\footnote{i.e. left adjoint to the inclusion functor.}.
    We write $\mSH_{S^1}(S,\Lambda)$ for $\mH(S,\Mod_\Lambda)$.
    
    The stable category $\mSH(S,\Lambda)$,  equipped with a natural functor $\motive : \mSm_S \to \mSH(S,\Lambda)$, is then obtained from $\mSH_{S^1}(S,\Lambda)$ by stabilizing the endofunctor $S^1_t\otimes({-})$ on $\mSH(S,\Lambda)$, where $S^1_t$ is the \emph{Tate circle with modulus} defined as
    $$
        S^1_t=\motive(\mathbb{P}^1,[0]+[\infty])/\motive(\{1\}) \in \mSH_{S^1}(S,\Lambda).
    $$
When $\Lambda$ is the sphere spectrum, we simply write $\mSH(S)$ for $\mSH(S,\Lambda)$.

\medbreak

The following properties follow immediately from the definition:
    \begin{enumerate}
        \item (Nisnevich descent)
                For any elementary distinguished Nisnevich square in $\mSm_S$, its image under $\motive$ from \eqref{eq;Motivicfunctor} is a coCartesian square.
        \item (Cube-invariance)
                For any $\mathcal{X}=(X,D)\in \mSm_S$, the canonical morphism $\motive(\mathcal{X}\otimes\bcube)\to \motive(\mathcal{X})$ is an equivalence.
        \item (Blow-up invariance)
                For any SNC blow-up $\mathcal{Y}\to \mathcal{X}$, the induced morphism
                $\motive(\mathcal{Y})\to \motive(\mathcal{X})$ is an equivalence.
                \end{enumerate}
   
   Besides the above, we will prove the following (see Theorems \ref{SBU}, \ref{THA}, \ref{GS}).
                
\begin{theorem} \label{main-2}
    Let $\mathcal{C}\in \{\Sp,\Mod_\Lambda\}$\footnote{In fact, we prove similar statements in the unstable category, allowing the case $\mathcal{C}=\Spc_*$ too.} be as above and $\motive$ be as \eqref{eq;Motivicfunctor}.
    Let $\mathcal{X}=(X,D)\in \mSm_S$ and let $Z\subset X$ be a smooth closed subscheme which is transversal to $|D|$ (see Definition \ref{NCwithD}).
        \begin{enumerate}
                \item (Smooth blow-up excision)
                Let $E$ denote the exceptional divisor of the blow-up $\pi\colon \Bl_ZX\to X$.
                Then, the following square is coCartesian:
                $$
                \xymatrix{
                    \motive(E,\pi^*D|_E)\ar[r]\ar[d]       &\motive(\Bl_ZX,\pi^*D)\ar[d]\\
                    \motive(Z,D|_Z)\ar[r]             &\motive(\mathcal{X}).
                }
                $$
        \item (Tame Hasse-Arf theorem)
                Suppose that $Z\subset X$ has codimension $1$.
                Then for any $\varepsilon\in (0,1]\cap \mathbb{Q}$, the natural morphism 
                 $$
                    \motive(X,D+Z)\to \motive(X,D+\varepsilon Z)
                $$
                is an equivalence.
        \item (Gysin sequence)
                Assume that $S$ is normal.
                Let $\pi\colon \normal_ZX\to Z$ be the normal bundle of $Z$, and set $\mathcal{N}_ZX:=(\normal_ZX,\pi^*(D|_Z))$.
                Then there exists a canonical cofiber sequence
                $$
                    \motive(\Bl_ZX,q^*D+E)\to \motive(\mathcal{X})\to \MTh(\mathcal{N}_ZX),
                $$
                where $q\colon \Bl_ZX\to X$ is the blow-up along $Z$, $E$ is the exceptional divisor, and $\MTh(\mathcal{N}_ZX)$ is the Thom space of $\mathcal{N}_ZX$ (see Definition \ref{def:thomsp}).
        \end{enumerate}
\end{theorem}

We will also prove the projective bundle formula and the Thom isomorphism for oriented ring spectra in $\mSH(S)$ (see Theorem \ref{thm:pbf} and Corollary \ref{Thom}) under the condition that $S$ is normal. 
These are essentially reworkings of the corresponding material in the logarithmic motivic homotopy theory \cite[\S 7]{BPO}.
\medbreak

\subsection{Relation with other motivic homotopy categories} \label{intro;comparison}
Here, we mention the relation between $\mSH(S,\Lambda)$ and other motivic homotopy categories.

Firstly, we have a localization functor from $\mSH(S,\Lambda)$ to the $\A^1$-invariant motivic homotopy category
 (see Corollary \ref{cor;mH-H-adjunction}). 
Let $\SH_{S^1}(S,\Lambda)\subset \Sh_\Nis(\Sm_S,\Mod_\Lambda)$ be the full $\infty$-subcategory spanned by the $\A^1$-local objects, and let $\SH(S,\Lambda)$ be the stabilization of $\SH_{S^1}(S,\Lambda)$ with respect to the Tate circle $(\AA^1- \{0\})/\{1\}$.
Then, we will see that there exists an adjunction
\[ \omega_!\colon \mSH(S,\Lambda)\rightleftarrows \SH(S,\Lambda)\colon \omega^*,\]
with $\omega_!$ symmetric monoidal, $\omega^*$ fully faithful and $\omega_!\motive(\mathcal{X})\simeq \motive(\mathcal{X}^\circ)$. Moreover, we prove that the functor $\omega_!$ is obtained as the $\AA^1$-localization of $\mSH(S,\Lambda)$ in \S \ref{sec:A1-loc}. 

\medbreak
    
Annala--Iwasa \cite{Annala-Iwasa} constructed a non-$\mathbb{A}^1$-invariant motivic homotopy category $\MS_S$ which they call the category of motivic spectra.
The category $\MS_S$ is the localization of the category of Zariski sheaves of spectra on $\Sm_S$ with respect to the elementary blow-up excision (see Definition \ref{def:EBU}).
The smooth blow-up excision for $\mSH(S)$ (see Theorem \ref{main-2}) implies that there exists a pair of adjoint functors
\begin{equation}\label{MS-mSH}
    \begin{tikzcd}
            \MS_S
            \arrow[rr,shift left=0.75ex,"\lambda_!"]
            \arrow[rr,leftarrow,shift right=0.75ex,"\lambda^*" below]
             &&
            \mSH(S),
    \end{tikzcd}
\end{equation}
such that $\lambda_!(\Sigma^\infty_{\mathbb{P}^1}X_+)=\motive(X,\varnothing)$ (see Theorem \ref{thm:MS_comparison}).
\medbreak

Finally, we compare $\mSH(S,\Lambda)$ with 
the logarithmic motivic homotopy category $\logSH(S,\Lambda)$ defined by Binda--Park--Østvær (see Corollary \ref{mSH_tSH_adjunction}), assuming that the base scheme $S$ is \emph{normal}. We need the normality assumption to prove an auxiliary lemma (Lemma \ref{lem:gamma-eq}). See Remark \ref{rem;normality} for details.
We construct a string of adjoint functors
   \begin{equation}\label{eq;logHmHAdunction}
        \begin{tikzcd}
            \mSH(S,\Lambda)
            \arrow[rr,shift left=1.5ex,"t_!"]
            \arrow[rr,leftarrow,"t^*" description]
            \arrow[rr,shift right=1.5ex,"t_*"']
            &&
            \logSH(S,\Lambda),
        \end{tikzcd}   \end{equation}
where $t_!\dashv t^*$ is a symmetric monoidal adjunction, and $t^*\dashv t_*$ is an adjunction.
Moreover, $t^*$ is fully faithful and
\[  t_!\motive(X,D)\simeq\motive^{\log}(X,\cM_{|D|}),\quad t^*\motive^{\log} (X,\cM_D)\simeq\colim_{\varepsilon\to 0}\motive(X,\varepsilon D),
\]
where $\cM_D$ denotes the log structure on $X$ associated to 
an SNCD $D$ on $X$, and $\motive^{\log} $ is the logarithmic counterpart of \eqref{eq;Motivicfunctor}.
The essential image of $t^*$ is identified with the full subcategory generated by $\Q$-modulus pairs $(X,D)$ where all components of $D$ have multiplicity $\leq 1$.
Also, by construction, we have the following commutative diagram:
\[
\xymatrix{
&\mSH(S)\ar[dr]^-{\omega_!}\ar[dd]^-{t_!}\\
\MS_S\ar[ur]^-{\lambda_!}\ar[dr]_-{\lambda_{\#}}&&\SH(S)\\
&\logSH(S)\ar[ur]_-{\omega_{\#}}
}
\]
where the functors $\lambda_{\#}$ and $\omega_{\#}$ are defined in \cite[Constructions 4.0.8, 4.0.16]{LogHom}.

\medbreak

We would like to note that the idea of comparing the category of motives with modulus and the category of log motives was originally proposed by Shane Kelly in the private letter to the second author \cite{Letter2020} in 2020, and the unstable version of the comparison is treated in \cite{Kelly-loghomotopy-type}, where the notion of $\mathbb{Q}$-modulus pairs is not used.
A merit of the use of $\mathbb{Q}$-modulus pairs in the present paper is that $t^*$ in \eqref{eq;logHmHAdunction} has a description in terms of representables.

\subsection{Representability of cohomology theories with ramification filtrations} \label{subsec;repfil}

In the second part of this paper, we fix a perfect base field $k$ and work in the categories $\mDAeff(k):=\mSH_{S^1}(\Spec k,\mathbb{Z})$ and $\mDA(k):=\mSH(\Spec k,\Z)=\mDAeff(k)[(S^1_t)^{-1}]$.
Thanks to previous works, we already have some cohomology theories which are representable in these categories but not $\A^1$-homotopy invariant.

\begin{ex}\label{ex;mOmega}(The Hodge cohomology with modulus, \S\ref{sec;MOmega}. Cf. \cite{KellyMiyazaki_Hodge1}, \cite{KellyMiyazaki_Hodge2}, \cite{Koizumi-blowup})
For a modulus pair $\sX = (X,D)$, define 
\[ \ulM\Omega^q (\sX) := \Gamma (X,\Omega^q_X (\log |D|)(\lceil D \rceil -|D|)). \]
One can prove that this derives a cohomology functor on $\mSm$ which is representable in $\mDA^{\eff}(k)$ (see Theorem \ref{thm;MOmega_rep}), and moreover, in $\mDA(k)$ (Lemma \ref{lem:mOmega_deloop}).
\end{ex}

\begin{ex}\label{ex;mWOmega}(The Hodge--Witt cohomology with modulus, \S
\ref{sec;MWOmega}. Cf. \cite{Shiho})
    In positive characteristic, Shiho extended the Hodge--Witt cohomology to a cohomology theory on $\mSm$ which is representable in $\mDA^{\eff}(k)$ and in $\mDA(k)$. See \S
\ref{sec;MWOmega} for more details. 
\end{ex}

These examples motivate the following question: 

\begin{question}
Is there a canonical way to upgrade 
the Nisnevich cohomology
$\mathrm{R}\Gamma({-},F)$ for a given Nisnevich sheaf 
$F$ of abelian groups on $\Sm_k$ to a cohomology theory on $\mSm_k$ which is representable in
$\mDAeff(k)$ as in the above examples?
\end{question}

To answer this question, we first recall that $\Omega^q$ and $\Witt_n\Omega^q$ are objects of the category $\RSC_\Nis$ of reciprocity sheaves introduced in \cite{KSY1} \cite{KSY2}. 
For the reader's convenience, we recall basic notions related to reciprocity sheaves in \S \ref{subsec:reciprocity-presheaves} and \S \ref{subsec:reciprocity-sheaves}, with a slight refinement in view of $\mathbb{Q}$-divisors.
An answer to a similar question in the context of logarithmic homotopy theory was given in \cite{Sai_log} by constructing a functor
$$
   (-)^{\log} : \RSC_{\Nis} \to \logDAeff(k),
$$
such that there is a natural equivalence for $F\in \RSC_{\Nis}$ and $X\in \Sm_k$
$$
        \map_{\logDAeff(k)} (\motive^{\log}(X,\triv),F^{\log}) \simeq \mathrm{R}\Gamma(X,F),
            $$ 
where $\logDA^{\eff}(k)$ is defined in \cite[5.2.1]{BPO} and viewed as the logarithmic counterpart of $\mDAeff(k)$ and $\motive^{\log}(X,\triv)$ is the object of $\logDAeff(k)$ associated to the log scheme with trivial log structure via the logarithmic counterpart of  \eqref{eq;Motivicfunctor}. For $F\in \RSC_{\Nis}$, $F^{\log}$ is defined as the cohomology of a sheaf whose value on a log scheme $\mathcal{X}=(X,\mathcal{M}_D)$ associated to $X\in \Sm_k$ and $D$, a SNCD on $X$, is a subgroup of $F(X-D)$ consisting of those sections which have logarithmic poles along $D$.  
The existence of such a functor connecting the theory of reciprocity sheaves and logarithmic motives, has a fundamental importance.
In \S \ref{sec:Fmod}, Definition \ref{func_modulus}, we will refine $(-)^{\log}$ by constructing a functor 
\begin{equation}\label{functor-mod}
({-})^\modulus : \RSC_{\Nis} \to  \mDAeff(k)
\end{equation}
such that there is a natural equivalence for $F\in \RSC_{\Nis}$ and $X\in \Sm_k$
$$ \map_{\mDAeff(k)} (\motive(X,\varnothing),F^\modulus) \simeq \mathrm{R}\Gamma(X,F),
            $$ 
where $(X,\varnothing)\in \mSm_k$ is the modulus pair with the empty modulus.
The functor \eqref{functor-mod} gives a rich source of supply of cohomology theories with ramification filtrations representable in $\mDAeff(k)$.
For $F\in \RSC_{\Nis}$, $F^{\modulus}$ is defined as the cohomology of a sheaf whose value on $\mathcal{X}=(X,D)\in \mSm_k$ is a subgroup of $F(X-|D|)$ consisting of those sections whose ramification along $|D|$ is bounded by $D$ (see Definition \ref{def;FFil} for terminology).  
Moreover, we will show that there exists a natural equivalence of functors (see Theorem \ref{mod-log-diag}): 
\[t_* \circ ({-})^\modulus =(-)^{\log},\]
where $t_*$ is the functor from \eqref{eq;logHmHAdunction}.
In Theorem \ref{thm;MOmega_mod}, we will compute $(\Omega^q)^{\modulus}$ explicitly and prove that it coincides with $\mathrm{m}\Omega^q$ from  Examples \ref{ex;mOmega}.
We also show that $(\Het^1)^\modulus$ represents the cohomology theory
$\mathrm{R}\Gamma(-,\uMHet^1)$ (see \eqref{Het} and \eqref{uMHet}), which exhibits a connection of $\mDAeff(k)$ with ramification theory (see Theorem \ref{thm:MHone_exc}). We will also show that the filtrations capturing the irregular singularities of rank $1$ connections on smooth schemes over a field of characteristic $0$ are representable in our motivic homotopy category (see Theorem \ref{thm:ConnComp}).

In a forthcoming work, we hope to conduct a similar computation for $\mathrm{m} \Witt_m\Omega^q$ from Example \ref{ex;mWOmega}.
It is also interesting to ask whether we can $S^1_t$-deloop $F^{\modulus}$ for any $F\in \RSC_\Nis$ to produce an $S^1_t$-spectrum in $\mDA(k)$.
This would follow if we can extend the cancellation theorem (\cite[Corollary 3.6]{Merici-Saito}) to $\mathbb{Q}$-modulus pairs.
This perspective led the authors to conjecture that one could $S^1_t$-deloop $(\Het^1)^{\modulus}$ by using unramified cohomology (see Conjecture \ref{conj:unramified}).

\def\Zb{\overline{Z}}

\subsection*{Acknowledgements}
The authors would like to thank Marc Hoyois and Naruki Masuda for answering questions about $\infty$-categories.
We thank Shane Kelly for his interest on our work, and for sharing his ideas about the comparison between the modulus theory and log theory.
We thank Ryomei Iwasa for a helpful discussion on the existence of the pair of adjoint functors \eqref{MS-mSH},  and for many helpful comments on an earlier version of the paper. 
We thank Doosung Park for answering our many questions about logarithmic motivic homotopy theory and for his valuable comments. 
We also thank Federico Binda, Bruno Kahn,  and Kay R\"ulling for their interest on our work.
We thank the referee for carefully reading and for the helpful comments.

\subsection*{Notation and conventions}

We use the following notation.

\begin{table}[h]
\begin{tabular}{l|l}
    $\Sch_S$    &   category of separated finite type $S$-schemes\\
    $\Sm_S$     &   category of smooth separated finite type $S$-schemes\\
    $\Spc$      &   $\infty$-category of spaces\\
    $\Spc_*$    &   $\infty$-category of pointed spaces\\
    $\Sp$       &   $\infty$-category of spectra\\
    $\PrL$      &   $\infty$-category of presentable ∞-categories and colimit-preserving functors\\
    $\Map_\mathcal{C}({-},{-})$       &       mapping space in an $\infty$-category $\mathcal{C}$\\
    $\map_\mathcal{C}({-},{-})$       &       mapping spectrum in a stable $\infty$-category $\mathcal{C}$\\
    $\Hom_\mathcal{C}({-},{-})$       &       $\pi_0\Map_\mathcal{C}({-},{-})$\\
    $\PSh(\mathcal{C},\mathcal{D})$     &   $\infty$-category of $\mathcal{D}$-valued presheaves on $\mathcal{C}$\\
    $\Sh_\tau(\mathcal{C},\mathcal{D})$     &   $\infty$-category of $\mathcal{D}$-valued $\tau$-sheaves on $\mathcal{C}$
\end{tabular}
\end{table}

For an integral domain $A$, we write $A^N$ for its integral closure.
Similarly, for an integral scheme $X$, we write $X^N$ for its normalization.
For $X\in \Sm_S$, a \emph{coordinate} on $X$ means a family of functions $x_1,\dots,x_d\in \Gamma(X,\mathcal{O}_X)$ which defines an \'etale morphism $X\to \mathbb{A}_S^d$.
We say that an $S$-scheme $X$ is \emph{essentially smooth} over $S$ if there exists a cofiltered diagram $\{X_i\}_{i \in I}$ in $\Sm_S$, whose transition maps are \'etale and affine, such that $X=\lim_{i \in I}X_i$.
We often extend a presheaf $F$ on $\Sm_S$ to essentially smooth schemes by setting
\[
X = \lim_{i \in I}X_i \implies  F(X) := \colim_{i \in I} F(X_i),
\]
where the right hand side does not depend on the choice of the diagram. 

Unless otherwise specified, cohomologies are taken with respect to the Nisnevich topology.
For a ring $A$, we write $A\{t_1,\dots,t_d\}$ for the henselization of $A[t_1,\dots,t_d]$ along $(t_1,\dots,t_d)$.

\part{Constructions and basic properties}

\section{Construction of the motivic homotopy categories with modulus}

\def\bA{\mathbb{A}}
\def\SmS{\Sm_S}
\def\SmSNis{\Sm_{S,\Nis}}
\def\mSmS{\mSm_S}
\def\mSmSNis{\mSm_{S,\Nis}}
\def\mDAefffS{\mDA^\eff(S)}

Throughout this section, we fix a qcqs scheme $S$.

\subsection{Relative $\mathbb{Q}$-SNCD}\label{relativeSNCD}

For a smooth $S$-scheme $X$, we write $\CDiv^+(X/S)$ for the monoid of relative effective Cartier divisors on $X$ over $S$.

\begin{definition}
     Let $X$ be a smooth $S$-scheme.
     A \emph{relative SNCD} on $X$ over $S$ is an element $D\in \CDiv^+(X/S)$ with the following property:
     \begin{itemize}
         \item Zariski-locally on $X$, there is a coordinate $x_1,\cdots,x_n$ on $X$ over $S$ such that $D=\div(x_1x_2\cdots x_m)$ for some $m\leq n$.
     \end{itemize}
\end{definition}

\begin{definition}\label{NCwithD}
    Let $X\in \Sm_S$ and let $D$ be a relative SNCD on $X$ over $S$.
    Let $Z\subset X$ be a smooth closed subscheme.
    We say that $Z$ has \emph{normal crossings} to $D$ if Zariski locally on $X$, there is a coordinate $x_1,\dots,x_n$ on $X$ over $S$ such that $D=\{x_1x_2\cdots x_m=0\}$ and $Z=\{x_{i_1}=\dots=x_{i_k}=0\}$.
    If moreover we have $Z\not\subset D$, then we say that $Z$ is \emph{transversal} to $D$.
\end{definition}

We also define a ``$\mathbb{Q}$-divisor version'' of relative SNCD:

\begin{definition}
     Let $X$ be a smooth $S$-scheme.
     A \emph{relative $\mathbb{Q}$-SNCD} on $X$ over $S$ is an element $D\in \CDiv^+(X/S)\otimes_{\mathbb{N}}\mathbb{Q}_{\geq 0}$ with the following property:
     \begin{itemize}
         \item Zariski-locally on $X$, there is a coordinate $x_1,\cdots,x_n$ on $X$ over $S$ such that $D=\sum_{i=1}^m r_i\div(x_i)$ for some $m\leq n$ and $r_1,\cdots,r_m\in \mathbb{Q}_{>0}$.
     \end{itemize}
     We say that $D$ has \emph{multiplicity $\leq 1$} if $r_i\leq 1$ holds for all $i$.
\end{definition}

\begin{definition}
     Let $X$ be a smooth $S$-scheme and $D$ be a relative $\mathbb{Q}$-SNCD on $X$ over $S$.
     If there is a coordinate $x_1,\cdots,x_n$ on $X$ over $S$ such that $D=\sum_{i=1}^m r_i\div(x_i)$ for some $m\leq n$ and $r_1,\cdots,r_m\in \mathbb{Q}_{>0}$, then we define the \emph{support} of $D$ by
     $$
        |D| = \div(x_1x_2\cdots x_m).
     $$
     In general, we define $|D|$ by gluing the above construction.
     By definition, $|D|$ is a relative SNCD on $X$ over $S$.
\end{definition}

\begin{lemma}\label{transversal_structure}
    Let $X\in \Sm_S$ and let $D$ be a relative $\mathbb{Q}$-SNCD on $X$ over $S$.
    Let $Z\subset X$ be a smooth closed subscheme of codimension $n$ which is transversal to $|D|$.
    Then Zariski locally on $X$, there is a commutative diagram
    $$
    \xymatrix{
        Z\ar@{^(->}[d]&Z\ar[l]_-{\id}\ar[r]^-{\sim}\ar@{^(->}[d]&Z\times\{0\}\ar@{^(->}[d]\\
        X&X'\ar[l]_-p\ar[r]^-{q}&Z\times\mathbb{A}^n
    }
    $$
    where $p$ and $q$ \'etale, the squares are Cartesian, and $p^*D=q^*\pr_1^*(D|_Z)$.
\end{lemma}

\begin{proof}
    This is \cite[Lemma 8]{Kelly-Saito} in case $S=\Spec k$ for a field $k$.
    The same proof works over a general base. We include it for the sake of the readers (see also \cite[Lem.3.2.28]{Morel-Voevodsky} for an argument over a general base).
    We may assume $S=\Spec R$ is affine.
    By Definition \ref{NCwithD}, 
    we may assume that there exists an \'etale morphism $\rho: X\to \bA^{r+n} = \Spec R[T_1,\dots,T_{r+n}]$ such that 
    $Z = \rho^{-1}(\bA^r\times \{0\})$ and $|D| = \rho^{-1}(\{T_1\cdots T_s = 0\})$ with $s \leq r$. 
    Define $\Gamma = X \times_{\bA^{r+n}} (Z \times \bA^n)$, 
    where the right morphism comes from the composition $Z \to X \overset{\rho}{\to} \bA^{r+n} \to\bA^r$.
    Then we have
    \[\Gamma\times_{\bA^{r+n}}  (\bA^r\times \{0\}) = (Z\times\bA^n)\times_{\bA^{r+n}} X\times_{\bA^{r+n}} (\bA^r\times\{0\}) =(Z\times\bA^n)\times_{\bA^{r+n}}Z = Z \times_{\bA^r} Z.\]
    Since $Z \to \bA^r$ is \'etale, the last term is a disjoint union of the diagonal 
    $Z \to Z \times_{\bA^r} Z$ and a closed subscheme $\Sigma\subset Z \times_{\bA^r} Z$. Put $X' = \Gamma-\Sigma$ with projections 
    $p : X' \to X$ and $q : X' \to Z\times \bA^n$. 
    Since $\rho: X\to\bA^{r+n}$ and $Z \times \bA^n\to \bA^{r+n}$ are \'etale, $p$ and $q$ are \'etale.  By the construction, 
    \[ p^{-1}(Z) \simeq Z, \quad q^{-1}(Z\times\{0\})\simeq  Z\times\{0\},\quad p^*D = q^*(D|_Z\times\bA^n),\]
    which implies the lemma.
\end{proof}

\subsection{log-smooth $\mathbb{Q}$-modulus pairs}

\begin{definition}\label{def;mSm}
    A \emph{log-smooth $\mathbb{Q}$-modulus pair} over $S$ is a pair $\mathcal{X}=(X,D)$ where $X\in \SmS$ and $D$ is a relative $\mathbb{Q}$-SNCD over $S$.
    For a log-smooth $\mathbb{Q}$-modulus pair $\mathcal{X}=(X,D)$, we write $\mathcal{X}^\circ:=X-|D|$ and call it the \emph{interior} of $\mathcal{X}$.
    A morphism of log-smooth $\mathbb{Q}$-modulus pairs $f\colon (X,D)\to (Y,E)$ is a morphism of $S$-schemes $f\colon X\to Y$ such that $f(\mathcal{X}^\circ)\subset \mathcal{Y}^\circ$ and $D\geq f^*E$ hold.
    We write $\mSmS$ for the category of log-smooth $\mathbb{Q}$-modulus pairs.
\end{definition}

\begin{example}
    We write $\bcube=(\mathbb{P}^1,[\infty])\in \mSmS$ and call it the \emph{cube}.
\end{example}

For two log-smooth $\mathbb{Q}$-modulus pairs $\mathcal{X}=(X,D)$ and $\mathcal{Y}=(Y,E)$, we set $\mathcal{X}\otimes \mathcal{Y}=(X\times Y, \pr_1^*D+\pr_2^*E)$.
It is easy to see that this gives a symmetric monoidal structure on $\mSmS$.
By left Kan extension, this extends to a symmetric monoidal structure on the category $\PSh(\mSmS)$ of presheaves of abelian groups on $\mSmS$.

\begin{definition}\label{def:SNC_blowup}
    Let $\mathcal{X}=(X,D)\in \mSmS$.
    Let $Z$ be a smooth closed subscheme of $X$ which has normal crossings to $|D|$.
    We write
    $$
        \Bl_Z\mathcal{X}:=(\Bl_ZX, \pi^*D)
    $$ where $\pi\colon \Bl_ZX\to X$ is the blow-up along $Z$.
    A morphism $\mathcal{Y}\to \mathcal{X}$ in $\mSmS$ is called an \emph{SNC blow-up} if there is a smooth closed subscheme $Z\subset |D|$ which has normal crossings to $|D|$ such that $\Bl_Z\mathcal{X}\cong \mathcal{Y}$.
\end{definition}

\subsection{Nisnevich topology on $\mSmS$}

\begin{definition}\label{def:topology}
    An \emph{elementary distinguished square} in $\mSmS$ is a commutative diagram in $\mSmS$ which is isomorphic to a diagram of the form
    $$
    \xymatrix{
        (V,D|_V)\ar[r]\ar[d]&(Y,D|_Y)\ar[d]\\
        (U,D|_U)\ar[r]      &(X,D),
    }
    $$
    where the associated diagram of underlying $S$-schemes is an elementary distinguished square.
    This defines a cd-structure and hence a topology on $\mSmS$, which we call the Nisnevich topology on $\mSmS$.
    We also define the Zariski topology on $\mSmS$ in the same manner.
\end{definition}

Let $\mathcal{C}\in \{\Spc,\Spc_*,\Mod_\Lambda\}$, where $\Lambda$ is a connective commutative ring spectrum.
We write $\mSmSNis$ for the site defined by the Nisnevich topology, and $\Sh_\Nis(\mSmS, \mathcal{C})$ for the associated $\infty$-category of $\mathcal{C}$-valued sheaves.
The inclusion functor $\Sh_\Nis(\mSm_S,\mathcal{C})\to \PSh(\mSm_S,\mathcal{C})$ admits a left adjoint $a_\Nis$.
There is a sequence of functors
$$
    \mSm_S\to \Sh_\Nis(\mSm_S,\Spc)\xrightarrow{({-})_+}
    \Sh_\Nis(\mSm_S,\Spc_*)\xrightarrow{\Lambda\otimes\Sigma^\infty({-})}
    \Sh_\Nis(\mSm_S,\Mod_\Lambda).
$$
We write $y\colon \mSm_S\to \Sh_\Nis(\mSm_S,\mathcal{C})$ for the canonical functor appearing in the above sequence, and call it the Yoneda functor.

\subsection{Construction of the motivic homotopy categories}

\begin{definition}
    We define classes of morphisms $\mathrm{CI},\mathrm{BI}$ in $\mSmS$ as follows:
    \begin{align*}
        \mathrm{CI}=&\{\mathcal{X}\otimes \bcube\to\mathcal{X}\mid \mathcal{X}\in \mSmS\},\\
        \mathrm{BI}=&\{\text{SNC blow-ups }\mathcal{Y}\to\mathcal{X}\}.
    \end{align*}
    Let $\mathcal{C}\in \{\Spc,\Spc_*,\Sp,\Mod_\Lambda\}$, where $\Lambda$ is a connective commutative ring spectrum.
    We define the \emph{$\mathcal{C}$-valued motivic homotopy category with modulus}
    $$
        \mH(S,\mathcal{C})
    $$
    to be the full $\infty$-subcategory of $\Sh_\Nis(\mSmS,\mathcal{C})$ spanned by the objects which are local with respect to $\CI$ and $\BI$.
    We write $\mathrm{L}_\mot\colon \Sh_\Nis(\mSmS,\mathcal{C})\to \mH(S,\mathcal{C})$ for the left adjoint to the inclusion functor.
    We define $\motive\colon \mSm_S\to \mH(S,\mathcal{C})$ to be the composition
    $$
        \motive\colon \mSm_S\xrightarrow{y} \Sh_\Nis(\mSm_S,\mathcal{C})\xrightarrow{\mathrm{L}_\mot}\mH(S,\mathcal{C}),
    $$
    where $y$ is the Yoneda functor.
    The $\infty$-category $\mH(S,\mathcal{C})$ underlies a presentably symmetric monoidal $\infty$-category, where the monoidal structure is the Day convolution of the tensor product $\otimes$ on $\mSmS$ \cite[Proposition 2.2.1.9]{Lurie-HA}.
\end{definition}

\begin{remark}
    We use the following notation:
    $$
        \mH(S,\mathcal{C}) = 
        \begin{cases}
            \mH(S)&(\mathcal{C}=\Spc),\\
            \mH_*(S)&(\mathcal{C}=\Spc_*),\\
            \mSH_{S^1}(S)&(\mathcal{C}=\Sp),\\
            \mSH_{S^1}(S,\Lambda)&(\mathcal{C}=\Mod_\Lambda).
        \end{cases}
    $$
    When $\Lambda$ is a classical ring, we write $\mDA^\eff(S,\Lambda)$ for $\mSH_{S^1}(S,\Lambda)$.
\end{remark}

\begin{remark}
    By construction, the category $\mH(S,\mathcal{C})$ has the following properties:
    \begin{enumerate}
        \item (Nisnevich descent)
                For any elementary distinguished square in $\mSmS$, its image under $\motive$ is a coCartesian square.
        \item (Cube-invariance)
                For any $\mathcal{X}\in \mSmS$, the canonical morphism $\motive(\mathcal{X}\otimes\bcube)\to \motive(\mathcal{X})$ is an equivalence.
        \item (Blow-up invariance)
                For any SNC blow-up $\mathcal{Y}\to \mathcal{X}$ in $\mSmS$, the induced morphism $\motive(\mathcal{Y})\to \motive(\mathcal{X})$ is an equivalence.
    \end{enumerate}
\end{remark}

\begin{definition}
    For $\varepsilon\in (0,1]\cap \mathbb{Q}$, we write $\bcube^\varepsilon = (\mathbb{P}^1,\varepsilon[\infty])$.
\end{definition}

\begin{lemma}\label{cube_epsilon_invariance}
    For any $\mathcal{X}\in \mSmS$ and $\varepsilon\in (0,1]\cap \mathbb{Q}$, the morphism $\motive(\mathcal{X}\otimes\bcube^\varepsilon)\to \motive(\mathcal{X})$ in $\mH(S,\mathcal{C})$ is an equivalence.
\end{lemma}

\begin{proof}
    It suffices to show that $\motive(\bcube^\varepsilon)\to \motive(\pt)$ is an equivalence.
    Consider the multiplication map $\mu\colon \mathbb{A}^1\times\mathbb{A}^1\to \mathbb{A}^1$.
    It extends to a map $\mu\colon \Bl_{(0,\infty),(\infty,0)}(\mathbb{P}^1\times\mathbb{P}^1)\to \mathbb{P}^1$ and gives a morphism
    $$
        \mu\colon \Bl_{(0,\infty),(\infty,0)}(\bcube^\varepsilon\otimes\bcube)\to \bcube^\varepsilon
    $$
    in $\mSmS$.
    The canonical inclusions
    $$
        i_\nu\colon \bcube^\varepsilon\xrightarrow{\sim} \bcube^\varepsilon\otimes\{\nu\}\hookrightarrow \bcube^\varepsilon\otimes\bcube\quad(\nu=0,1)
    $$
    lift to $\iota_\nu\colon \bcube^\varepsilon\to \Bl_{(0,\infty),(\infty,0)}(\bcube^\varepsilon\otimes\bcube)$.
    Consider the following commutative diagram:
    $$
    \xymatrix{
        &\motive(\bcube^\varepsilon)\ar[dl]_-{i_1}\ar[d]^-{\iota_1}\ar[dr]^-{\id}\\
        \motive(\bcube^\varepsilon\otimes\bcube)&\motive(\Bl_{(0,\infty),(\infty,0)}(\bcube^\varepsilon\otimes\bcube))\ar[l]_-{\sim}\ar[r]^-{\mu}&\motive(\bcube^\varepsilon)\\
        &\motive(\bcube^\varepsilon)\ar[ul]^-{i_0}\ar[u]_-{\iota_0}\ar[r]&\motive(\pt).\ar[u]_-{0}
    }
    $$
    We have $i_0\simeq i_1$ since they are sections of the equivalence $\motive(\bcube^\varepsilon\otimes\bcube)\to \motive(\bcube^\varepsilon)$.
    Therefore we have $\iota_0 \simeq \iota_1$ and thus the composition $\motive(\bcube^\varepsilon)\to \motive(\pt)\to \motive(\bcube^\varepsilon)
    $ is homotopic to the identity.
\end{proof}

\subsection{Comparison with Morel--Voevodsky's categories}

In this subsection, we construct comparison functors between our categories and Morel--Voevodsky's motivic homotopy categories.
First we recall the definition of Morel--Voevodsky's categories:

\begin{definition}
    We define a class of morphisms $\mathrm{HI}$ in $\Sm_S$ by
    $$
        \mathrm{HI}=\{X\times\mathbb{A}^1\to X\mid X\in \Sm_S\}.
    $$
    Let $\mathcal{C}\in \{\Spc,\Spc_*,\Sp,\Mod_\Lambda\}$, where $\Lambda$ is a connective commutative ring spectrum.
    The \emph{$\mathcal{C}$-valued motivic homotopy category}
    $$
        \H(S,\mathcal{C})
    $$
    is defined to be the full $\infty$-subcategory of $\Sh_\Nis(\Sm_S,\mathcal{C})$ spanned by the objects which are local with respect to $\HI$.
    We write $\mathrm{L}_\mot\colon \Sh_\Nis(\Sm_S,\Spc)\to \H(S,\mathcal{C})$ for the left adjoint to the inclusion functor.
    We define $\motive\colon \Sm_S\to \H(S,\mathcal{C})$ to be the composition
    $$
        \motive\colon \Sm_S\xrightarrow{y} \Sh_\Nis(\Sm_S,\mathcal{C})\xrightarrow{\mathrm{L}_\mot}\H(S,\mathcal{C}),
    $$
    where $y$ is the Yoneda functor.
    We use the following notation:
    $$
        \H(S,\mathcal{C}) = 
        \begin{cases}
            \H(S)&(\mathcal{C}=\Spc),\\
            \H_*(S)&(\mathcal{C}=\Spc_*),\\
            \SH_{S^1}(S)&(\mathcal{C}=\Sp),\\
            \SH_{S^1}(S,\Lambda)&(\mathcal{C}=\Mod_\Lambda).
        \end{cases}
    $$
    When $\Lambda$ is a classical ring, we write $\DA^\eff(S,\Lambda)$ for $\SH_{S^1}(S,\Lambda)$.
\end{definition}

Let us compare $\H(S,\mathcal{C})$ with our category $\mH(S,\mathcal{C})$.
Consider the functor
$$
    \omega\colon \mSmS\to \SmS;\quad \mathcal{X}\mapsto \mathcal{X}^\circ,
$$
which admits a left adjoint $X\mapsto (X,\emptyset)$.
It is easy to see that $\omega$ gives a morphism of sites $\SmSNis\to \mSmSNis$.
Therefore the functor $\omega$ induces an adjunction between sheaf categories
\begin{align*}
\omega_!\colon \Sh_\Nis(\mSmS,\mathcal{C})\rightleftarrows\Sh_\Nis(\SmS,\mathcal{C})\colon \omega^*,
\end{align*}
where $\omega_!$ is symmetric monoidal and $\omega_!y(\mathcal{X})\simeq y(\mathcal{X}^\circ)$, $\omega^*F(\mathcal{X})\simeq F(\mathcal{X}^\circ)$.
Since the functor $\omega$ admits a left adjoint $X\mapsto (X,\emptyset)$, we also have $\omega_!F(X)\simeq F(X,\emptyset)$.
Moreover, the functor $\omega^*$ is fully faithful because $\omega_!\circ \omega^*\simeq \id$.

By the above formula for $\omega_!$, we see that the functor $\omega^*\colon \Sh_\Nis(\SmS,\mathcal{C})\to \Sh_\Nis(\mSmS,\mathcal{C})$ sends $\HI$-local objects to $(\CI\cup\BI)$-local objects.
This implies the following:

\begin{corollary}\label{cor;mH-H-adjunction}
    Let $\mathcal{C}\in \{\Spc,\Spc_*,\Sp,\Mod_\Lambda\}$, where $\Lambda$ is a connective commutative ring spectrum.
    Then there is an adjunction
    $$
        \omega_!\colon \mH(S,\mathcal{C})\rightleftarrows \H(S,\mathcal{C})\colon \omega^*
    $$
    where $\omega_!$ is symmetric monoidal, $\omega^*$ is fully faithful, and $\omega_!\motive(\mathcal{X})\simeq \motive(\mathcal{X}^\circ)$.
\end{corollary}

\section{Properties of motivic homotopy categories with modulus} \label{sec:propmot}

In this section, we fix $\mathcal{C}\in \{\Spc_*,\Sp,\Mod_\Lambda\}$ and study basic properties of the functor
$$
    \motive \colon \mSmS\to \mH(S, \mathcal{C}).
$$

\subsection{Motive of $(\mathbb{P}^n,\mathbb{P}^{n-1})$}

First we prepare a lemma about the motive of $(\mathbb{P}^n,\mathbb{P}^{n-1})$ and its blow-ups.
Here, we identify $\mathbb{P}^{n-1}$ with a hypersurface inside $\mathbb{P}^n$.
\begin{lemma}\label{P^n_relative_motive}
    The following statements hold true:
    \begin{enumerate}
        \item   $\motive(\mathbb{P}^n,\mathbb{P}^{n-1})\to \motive(\pt)$ is an equivalence.
        \item   If $x:\pt\to \mathbb{P}^{n-1}$ is a section, then $\motive(\Bl_x(\mathbb{P}^n,\mathbb{P}^{n-1}))\to \motive(\ast)$ is an equivalence.
        \item   If $x: \pt\to \mathbb{P}^{n}\backslash \mathbb{P}^{n-1}$ is a section, then $\motive(E,\varnothing)\to \motive(\Bl_x(\mathbb{P}^n,\mathbb{P}^{n-1}))$ is an equivalence, where $E$ is the exceptional divisor.
    \end{enumerate}
\end{lemma}

\begin{proof}
    We proceed by induction on $n$.
    If $n=1$, then we have $(\mathbb{P}^1,\mathbb{P}^0)\cong \Bl_x(\mathbb{P}^1,\mathbb{P}^0)\cong \bcube$, so all statements follow from the cube-invariance.
    Suppose $n\geq 2$.
    If $x\in \mathbb{P}^{n-1}$ (resp. $x\in \mathbb{P}^{n}- \mathbb{P}^{n-1}$), then $\Bl_x(\mathbb{P}^n,\mathbb{P}^{n-1})$ is a cube-bundle over $(\mathbb{P}^{n-1},\mathbb{P}^{n-2})$ (resp. $(E,\varnothing)$).
    Therefore the claims (2) and (3) follow from the Nisnevich descent and the cube-invariance.
    The claim (1) follows from (2) via the blow-up invariance; $\motive(\Bl_x(\mathbb{P}^n,\mathbb{P}^{n-1}))\xrightarrow{\sim}\motive(\mathbb{P}^n,\mathbb{P}^{n-1})$.
\end{proof}

\subsection{Smooth blow-up excision}

In this subsection we prove the smooth blow-up excision following Kelly-Saito's argument \cite{Kelly-Saito}.

\begin{definition} \label{def:SBU}
    Let $\mathcal{X}=(X,D)\in \mSmS$ and let $Z\subset X$ be a smooth closed subscheme which is transversal to $|D|$.
    Let $E$ denote the exceptional divisor of the blow-up $\Bl_ZX\to X$.
    Consider the following commutative diagram:
    $$
    \xymatrix{
        \motive(E,D|_E)\ar[r]\ar[d]       &\motive(\Bl_Z\mathcal{X})\ar[d]\\
        \motive(Z,D|_Z)\ar[r]             &\motive(\mathcal{X}).
    }
    $$
    We say that $(\mathrm{SBU})_{(\mathcal{X},Z)}$ holds if the total cofiber of the above square is zero.
\end{definition}

\begin{theorem}[Smooth blow-up excision]\label{SBU}
    Let $\mathcal{X}=(X,D)\in \mSmS$ and let $Z\subset X$ be a smooth closed subscheme which is transversal to $|D|$.
    Then $(\mathrm{SBU})_{(\mathcal{X},Z)}$ holds.
\end{theorem}

\begin{lemma}\label{SBU_tensor}
    Let $\mathcal{X}=(X,D)\in \mSmS$ and let $Z\subset X$ be a smooth closed subscheme which is transversal to $|D|$.
    If $(\mathrm{SBU})_{(\mathcal{X},Z)}$ holds, then $(\mathrm{SBU})_{(\mathcal{X}\otimes\mathcal{Y},Z\times Y)}$ holds for any $\mathcal{Y}=(Y,E)\in \mSmS$.
\end{lemma}

\begin{proof}
    This is a consequence of the fact that the functor $\motive(Y,E)\otimes({-})$ preserves colimits.
\end{proof}

\begin{lemma}\label{SBU_descent}
    Let $\mathcal{X}=(X,D)\in \mSmS$ and let $Z\subset X$ be a smooth closed subscheme which is transversal to $|D|$.
    Let $\{\mathcal{U}_i\}_{i\in I}$ be a Zariski covering of $\mathcal{X}$.
    For each finite non-empty subset $J\subset I$, we set $\mathcal{U}_J=\bigcap_{j\in J}\mathcal{U}_j$.
    If $(\mathrm{SBU})_{(\mathcal{U}_J,Z_J)}$ holds for every finite non-empty subset $J\subset I$, then $(\mathrm{SBU})_{(\mathcal{X},Z)}$ holds.
\end{lemma}

\begin{proof}
    For any $(i_0,\dots,i_n) \in I^{n+1}$, set 
    $\mathcal{U}_{i_0\dots i_n}:= \bigcap_{k=0}^n\mathcal{U}_{i_k}$, and 
    $Z_{i_0 \dots i_n}:=Z \times_{X} U_{i_0\dots i_n}$.
    Let $Q(\mathcal{X},Z)$ be the square in Definition \ref{def:SBU}.
    The Nisnevich descent (or the Zariski descent) expresses each of the four vertices of $Q(\mathcal{X},Z)$ as the colimit of its values on the corresponding \v{C}ech nerve.
    Here we use that blow-ups and their exceptional divisors commute with restriction to open subschemes.
    Hence, in the category of commutative squares in $\mH(S,\mathcal{C})$, we have
    $$
        Q(\mathcal{X},Z)\simeq
        \colim_{[n] \in \Delta^{\mathrm{op}}}
        \coprod_{(i_0, \dots ,i_n) \in I^{n+1}}
        Q(\mathcal{U}_{i_0\dots i_n}, Z_{i_0\dots i_n}).
    $$
    The total cofiber functor preserves colimits since it is constructed from pushouts and cofibers.
    This implies
    $$
        \operatorname{tcofib} Q(\mathcal{X},Z)\simeq
        \colim_{[n]\in\Delta^{\mathrm{op}}}
        \coprod_{(i_0,\dots,i_n)\in I^{n+1}}
        \operatorname{tcofib} Q(\mathcal{U}_{i_0\dots i_n}, Z_{i_0\dots i_n})=0,
    $$
    where the last equality holds since every intersection $\mathcal{U}_{i_0\dots i_n}$ is of the form $\mathcal{U}_J$ for a finite non-empty subset $J\subset I$.
\end{proof}

\begin{lemma}\label{SBU_P^n}
    Let $x=[0:\cdots:0:1]$.
    Then $(\mathrm{SBU})_{((\mathbb{P}^n,\mathbb{P}^{n-1}),x)}$ holds.
\end{lemma}

\begin{proof}
    Let $E$ denote the exceptional divisor of the blow-up $\Bl_x\mathbb{P}^n\to \mathbb{P}^n$.
    We have to show that the total cofiber of the following square is zero:
    $$
    \xymatrix{
        \motive(E,\varnothing)\ar[r]\ar[d]       &\motive(\Bl_x(\mathbb{P}^n,\mathbb{P}^{n-1}))\ar[d]\\
        \motive(\pt)\ar[r]^-{[0:\cdots:0:1]}             &\motive(\mathbb{P}^n,\mathbb{P}^{n-1}).
    }
    $$
    By Lemma \ref{P^n_relative_motive}, the horizontal morphisms are equivalences, so the diagram is coCartesian.
\end{proof}

\begin{lemma}\label{SBU_excision}
    Let $\mathcal{X}=(X,D)\in \mSmS$ and let $Z\subset X$ be a smooth closed subscheme which is transversal to $|D|$.
    Let $X'\to X$ be an \'etale morphism which induces an isomorphism $X'\times_XZ\cong Z$, and set $\mathcal{X}'=(X',D|_{X'})$.
    Then we have
    $$
        (\mathrm{SBU})_{(\mathcal{X}',Z)}\iff (\mathrm{SBU})_{(\mathcal{X},Z)}.
    $$
\end{lemma}

\begin{proof}
    Let $E$ denote the exceptional divisor of the blow-up $\Bl_ZX\to X$, which can be identified with the exceptional divisor of the blow-up $\Bl_ZX'\to X'$.
    Let $U=X- Z$ and $U'=X'- Z$.
    Consider the following commutative diagram:
    $$
    \xymatrix{
        \motive(U',D|_{U'})\ar[d]\ar[r]           &\motive(U,D|_U)\ar[d]\\
        \motive(\Bl_Z\mathcal{X}')\ar[d]\ar[r]   &\motive(\Bl_Z\mathcal{X})\ar[d]\\
        \motive(\mathcal{X}')\ar[r]     &\motive(\mathcal{X}).
    }
    $$
    The upper square and the total rectangle are coCartesian by the Nisnevich descent.
    By the pasting law, the lower square is also coCartesian.
    Next, we consider the following commutative diagram:
    $$
    \xymatrix{
        \motive(E,D|_E)\ar[r]\ar[d]       &\motive(\Bl_Z\mathcal{X}')\ar[d]\ar[r]   &\motive(\Bl_Z\mathcal{X})\ar[d]\\
        \motive(Z,D|_Z)\ar[r]             &\motive(\mathcal{X}')\ar[r]     &\motive(\mathcal{X}).
    }
    $$
    We have seen that the right square is coCartesian.
    This implies that the total cofiber of the left square is isomorphic to that of the total rectangle.
    Therefore $(\mathrm{SBU})_{(\mathcal{X}',Z)}$ holds if and only if $(\mathrm{SBU})_{(\mathcal{X},Z)}$ holds.
\end{proof}

\begin{proof}[Proof of Theorem \ref{SBU}]
    By Lemma \ref{transversal_structure} and Lemma \ref{SBU_descent}, we may assume that there is a commutative diagram
    \begin{equation} \label{diag:excision}
    \xymatrix{
        Z\ar@{^(->}[d]&Z\ar[l]_-{\id}\ar[rr]^-{\id}\ar@{^(->}[d]&&Z\times \{0\}\ar@{^(->}[d]\\
        X&X'\ar[l]_-p\ar[r]^-{q}&Z\times\mathbb{A}^n\ar@{^(->}[r]&Z\times\mathbb{P}^n,
    }
    \end{equation}
    where $p$ and $q$ \'etale, the squares are Cartesian, and $p^*D=q^*\pr_1^*(D|_Z)$ for the projection $\mathrm{pr}_1 : Z \times \AA^1 \to Z$.
    By Lemma \ref{SBU_P^n} and Lemma \ref{SBU_tensor}, we see that $(\mathrm{SBU})_{((Z\times\mathbb{P}^n,Z\times\mathbb{P}^{n-1}),Z)}$ holds.
    Using Lemma \ref{SBU_excision} twice, we conclude that $(\mathrm{SBU})_{(\mathcal{X},Z)}$ holds.
\end{proof}

\subsection{Tame Hasse-Arf theorem}

In this subsection we prove the following result:

\begin{definition}
    Let $\mathcal{X}=(X,D)\in \mSmS$ and let $Z\subset X$ be a smooth divisor which is transversal to $|D|$.
    We say that $(\mathrm{THA})_{(\mathcal{X},Z)}$ holds if for any $\varepsilon\in (0,1]\cap \mathbb{Q}$, the cofiber of the morphism
    $$
        \motive(X,D+Z)\to \motive(X,D+\varepsilon Z)
    $$
    is zero.
\end{definition}

\begin{theorem}[Tame Hasse-Arf theorem]\label{THA}
    Let $\mathcal{X}=(X,D)\in \mSmS$ and let $Z\subset X$ be a smooth divisor which is transversal to $|D|$.
    Then $(\mathrm{THA})_{(\mathcal{X},Z)}$ holds.
\end{theorem}

\begin{lemma}\label{THA_tensor}
    Let $\mathcal{X}=(X,D)\in \mSmS$ and let $Z\subset X$ be a smooth divisor which is transversal to $|D|$.
    If $(\mathrm{THA})_{(\mathcal{X},Z)}$ holds, then $(\mathrm{THA})_{(\mathcal{X}\otimes\mathcal{Y},Z\times Y)}$ holds for any $\mathcal{Y}=(Y,E)\in \mSmS$.
\end{lemma}

\begin{proof}
    This is a consequence of the fact that the functor $\motive(Y,E)\otimes({-})$ preserves colimits.
\end{proof}

\begin{lemma}\label{THA_descent}
    Let $\mathcal{X}=(X,D)\in \mSmS$ and let $Z\subset X$ be a smooth divisor which is transversal to $|D|$.
    Let $\{\mathcal{U}_i\}_{i\in I}$ be a Zariski covering of $\mathcal{X}$.
    For each finite non-empty subset $J\subset I$, we set $\mathcal{U}_J=\bigcap_{j\in J}\mathcal{U}_j$.
    If $(\mathrm{THA})_{(\mathcal{U}_J,Z_J)}$ holds for every finite non-empty subset $J\subset I$, then $(\mathrm{THA})_{(\mathcal{X},Z)}$ holds.
\end{lemma}

\begin{proof}
    This follows from the Nisnevich descent and a discussion similar to the proof of Lemma \ref{MBU_descent}.
\end{proof}

\begin{lemma}\label{THA_P^1}
    $(\mathrm{THA})_{((\mathbb{P}_S^1,\varnothing),\infty)}$ holds.
\end{lemma}

\begin{proof}
    What we have to show is that the cofiber of $\motive(\bcube)\to \motive(\bcube^\varepsilon)$ is zero.
    This morphism is an equivalence by Lemma \ref{cube_epsilon_invariance} and the cube-invariance.
\end{proof}

\begin{lemma}\label{THA_excision}
    Let $\mathcal{X}=(X,D)\in \mSmS$ and let $Z\subset X$ be a smooth divisor transversal to $|D|$.
    Let $X'\to X$ be an \'etale morphism which induces an isomorphism $X'\times_XZ\cong Z$ and set $\mathcal{X}'=(X',D|_{X'})$.
    Then we have
    $$
        (\mathrm{THA})_{(\mathcal{X}',Z)}
        \iff (\mathrm{THA})_{(\mathcal{X},Z)}
    $$
\end{lemma}

\begin{proof}
    Let $U=X- Z$ and $U'=X'- Z$.
    Consider the following commutative diagram:
    $$
    \xymatrix{
        \motive(U',D|_{U'})\ar[d]\ar[r]&\motive(X',D|_{X'}+Z)\ar[r]\ar[d]       &\motive(X',D|_{X'}+\varepsilon Z)\ar[d]\\
        \motive(U,D|_U)\ar[r]&\motive(X,D+Z)\ar[r]&\motive(X,D+\varepsilon Z).
    }
    $$
    The left square and the total rectangle are coCartesian by the Nisnevich descent.
    By the pasting law, the right square is also coCartesian, so we have an equivalence between the cofibers of two horizontal morphisms of the right square.
    Therefore $(\mathrm{THA})_{(\mathcal{X}',Z)}$ holds if and only if $(\mathrm{THA})_{(\mathcal{X},Z)}$ holds.
\end{proof}

\begin{proof}[Proof of Theorem \ref{THA}]
    As in the proof of Theorem \ref{SBU}, the result follows from Lemma \ref{transversal_structure}, Lemma \ref{THA_tensor}, Lemma \ref{THA_descent}, Lemma \ref{THA_P^1}, and Lemma \ref{THA_excision}.
\end{proof}

\section{Relation with logarithmic motivic stable homotopy theory} \label{sec:modlog}

Let $S$ be a qcqs scheme and $\Lambda$ be a connective commutative ring spectrum.
In this section, we construct a localization functor from our category $\mSH_{S^1}(S,\Lambda)$ to the category of logarithmic motives $\logSH_{S^1}(S,\Lambda)$ defined by Binda--Park--{\O}stv{\ae}r \cite{LogHom}.

We would like to mention here that Shane Kelly kindly shared his idea of this type of comparison (in a slightly different form since we didn't have the notion of $\mathbb{Q}$-modulus pairs) in 2020 in the private communication with the second (and afterwards with the third) authors \cite{Letter2020}. 
Afterwards, Kelly and the second author independently realized that the strategy doesn't work if we use the category of sheaves \emph{with transfers} (which are constructed in \cite{KMSY1,KMSY2,KMSY3} over a field and in \cite{Kelly-Miyazaki} over any noetherian scheme), but it perfectly works for sheaves \emph{without} transfers. 
An unstable version of the comparison result is in Kelly's note \cite{Kelly-loghomotopy-type}.

\subsection{Logarithmic motivic stable homotopy category}
First we recall the construction of the logarithmic motivic stable homotopy category due to Binda--Park--Østvær.
We write $\SmlSm_S$ for the category of log-smooth separated fs log schemes of finite type over $S$, whose underlying scheme is smooth over $S$.

\begin{definition}
    A \emph{log pair} over $S$ is a pair $\mathfrak{X}=(X,D)$ where $X\in \Sm_S$ and $D$ is a relative SNCD on $X$ over $S$.
    We write $\mathfrak{X}^\circ=X-D$.
\end{definition}

If $(X,D)$ is a log pair over $S$, then we can construct a log scheme $(X,\mathcal{M}_{(X,D)})\in \SmlSm_S$ where $\mathcal{M}_{(X,D)}$ is the compactifying log structure induced by $X-D\hookrightarrow X$.
When $S$ is normal (hence so is $X$ since $X$ is $S$-smooth), then this compactifying log structure coincides with the Deligne-Faltings log structure due to \cite[Proposition III.1.7.3(4))]{Ogus-logbook}. In particular, any object in $\SmlSm_S$ arises in this way.
Suppose that $\mathfrak{X}=(X,D)$ and $\mathfrak{Y}=(Y,E)$ be log pairs over $S$.
Then, a morphism of $S$-schemes $f:X \to Y$ induces a (unique) morphism of log schemes $(X,\mathcal{M}_{(X,D)}) \to (Y,\mathcal{M}_{(Y,E)})$ if and only if $f(\mathfrak{X}^\circ) \subset \mathfrak{Y}^\circ$.
Conversely, for any $X\in \SmlSm_S$, there exists a unique relative SNCD $D$ on $\underline{X}$ over $S$ such that $X\cong (\underline{X},\mathcal{M}_{(\underline{X},D)})$ \cite[Lemma A.5.10]{BPO}.
Therefore we get the following result:

\begin{lemma} \label{lem:gamma-eq}
    Assume that $S$ is normal.
    Then $\SmlSm_S$ is equivalent to the following category:
    \begin{itemize}
        \item The objects are log pairs over $S$.
        \item A morphism from $\mathfrak{X}=(X,D)$ to $\mathfrak{Y}=(Y,E)$ is a morphism of $S$-schemes $f\colon X\to Y$ such that $f(\mathfrak{X}^\circ)\subset \mathfrak{Y}^\circ$.
    \end{itemize}
\end{lemma}

\begin{remark} \label{rem;normality}
    Lemma \ref{lem:gamma-eq} is the only obstruction that prevents from proving the comparison in this section over any qcqs schemes. In other words, assuming that Lemma \ref{lem:gamma-eq} holds for any qcqs scheme, we can remove the normality assumption on $S$ from all statements in Part I of this article. Unfortunately, we do not know a proof or a counter-example of the lemma. 
\end{remark}

For the rest of subsection, we assume that the base scheme $S$ is normal, and identify $\SmlSm_S$ with the category of log pairs. We will keep this normality assumption on $S$ throughout \S \ref{sec:modlog}, except \S \ref{subsec:AI}.
The category $\SmlSm_S$ has Cartesian products: $(X,D)\times (Y,E)=(X\times Y, \pr_1^*D+\pr_2^*E)$.
We write $\bcube=(\mathbb{P}^1,[\infty])\in \SmlSm_S$.

\begin{definition}
    Let $\mathfrak{X}=(X,D)\in \SmlSm$ and let $Z\subset D$ be a smooth closed subscheme which has normal crossings to $D$.
    We write
    $$
        \Bl_Z\mathfrak{X}:=(\Bl_ZX, \pi^{-1}(D))
    $$ where $\pi\colon \Bl_ZX\to X$ is the blow-up along $Z$.
    A morphism $f\colon \mathfrak{Y}\to \mathfrak{X}$ in $\SmlSm_S$ is called an \emph{SNC blow-up} if there is a smooth closed subscheme $Z\subset D$ which has normal crossings to $D$ such that $f$ is isomorphic to $\pi\colon \Bl_Z\mathfrak{X}\to \mathfrak{X}$.
\end{definition}

\begin{definition}\label{def:log_topology}
    An \emph{elementary distinguished square} in $\SmlSm_S$ is a commutative diagram in $\SmlSm_S$ which is isomorphic to a diagram of the form
    $$
    \xymatrix{
        (V,D|_V)\ar[r]\ar[d]&(Y,D|_Y)\ar[d]\\
        (U,D|_U)\ar[r]      &(X,D),
    }
    $$
    where the associated diagram of underlying $S$-schemes is an elementary distinguished square.
    This defines a cd-structure and hence a topology on $\SmlSm_S$, which we call the strict Nisnevich topology on $\SmlSm_S$.
    We also define the Zariski topology on $\SmlSm_S$ in the same manner.
\end{definition}

Let $\Lambda$ be a connective commutative ring spectrum.
We write $\SmlSm_{S,\sNis}$ for the site defined by the strict Nisnevich topology, and $\Sh_\sNis(\SmlSm_S, \Mod_\Lambda)$ for the associated sheaf category.
The inclusion functor $\Sh_\sNis(\SmlSm_S,\Mod_\Lambda)\to \PSh(\SmlSm_S,\Mod_\Lambda)$ admits a left adjoint $a_\sNis$.
We write $y$ for the canonical functor
$$
    \SmlSm_S\to \Sh_\sNis(\SmlSm_S,\Spc)\xrightarrow{\Lambda\otimes\Sigma^\infty_+({-})}
    \Sh_\sNis(\SmlSm_S,\Mod_\Lambda).
$$

\begin{definition}\label{def;CIlogBIlog}
    We define classes of morphisms $\mathrm{CI}^{\log},\mathrm{BI}^{\log}$ in $\SmlSm_S$ as follows:
    \begin{align*}
        \mathrm{CI}^{\log}=&\{\mathfrak{X}\times \bcube\to\mathfrak{X}\mid \mathfrak{X}\in \SmlSm_S\},\\
        \mathrm{BI}^{\log}=&\{\text{SNC blow-ups }\mathfrak{Y}\to\mathfrak{X}\}.
    \end{align*}
    Let $\Lambda$ be a connective commutative ring spectrum.
    The \emph{$S^1$-stable logarithmic motivic homotopy category}
    $$
        \logSH_{S^1}(S,\Lambda)
    $$
    is defined to be the full $\infty$-subcategory of $\Sh_\sNis(\SmlSm_S,\Mod_\Lambda)$ spanned by the objects which are local with respect to $\CI^{\log}$ and $\BI^{\log}$.
    We write $\mathrm{L}_\mot\colon \Sh_\sNis(\SmlSm_S,\Mod_\Lambda)\to \logSH_{S^1}(S,\Lambda)$ for the left adjoint to the inclusion functor.
    We define $\motive^{\log}\colon \SmlSm_S\to \logSH_{S^1}(S,\Lambda)$ to be the composition
    $$
        \motive^{\log}\colon \SmlSm_S\xrightarrow{y} \Sh_\sNis(\SmlSm_S,\Mod_\Lambda)\xrightarrow{\mathrm{L}_\mot}\logSH_{S^1}(S,\Lambda),
    $$
    where $y$ is the Yoneda functor.
    The $\infty$-category $\logSH_{S^1}(S,\Lambda)$ underlies a presentably symmetric monoidal $\infty$-category, where the monoidal structure is the Day convolution of the Cartesian product in $\SmlSm_S$ \cite[Proposition 2.2.1.9]{Lurie-HA}.
\end{definition}

\begin{remark}
    Over normal bases, the category $\logSH_{S^1}(S,\Lambda)$ is equivalent to the category  $\logSH^{\eff}_S$ in \cite{LogHom}, thanks to \cite[Theorem 3.4.5, 3.4.6]{LogHom}.
\end{remark}

\begin{remark}
    In the original paper \cite{LogHom}, the object $\motive^{\log}(\mathfrak{X})\in \logSH_{S^1}(S,\Lambda)$ is denoted by $\Lambda\otimes \Sigma^\infty_{S^1,+}(\mathfrak{X})$.
    When $\Lambda$ is a classical ring, we write $\logDA^\eff(S,\Lambda)$ for $\logSH_{S^1}(S,\Lambda)$.
\end{remark}

\begin{remark}
    By construction, the category $\logSH_{S^1}(S,\Lambda)$ has the following properties:
    \begin{enumerate}
        \item (Nisnevich descent)
                For any elementary distinguished square in $\SmlSm_S$, its image under $\motive^{\log}$ is a coCartesian square.
        \item (Cube-invariance)
                For any $\mathfrak{X}\in \SmlSm_S$, the canonical morphism $\motive^{\log}(\mathfrak{X}\times\bcube)\to \motive^{\log}(\mathfrak{X})$ is an equivalence.
        \item (Blow-up invariance)
                For any SNC blow-up $\mathfrak{Y}\to \mathfrak{X}$ in $\SmlSm_S$, the induced morphism $\motive^{\log}(\mathfrak{Y})\to \motive^{\log}(\mathfrak{X})$ is an equivalence.
    \end{enumerate}
\end{remark}

\subsection{Construction of the localization functor}

Throughout this subsection, we keep the assumption that the base scheme $S$ is normal.

\begin{lemma}\label{SmlSm_vs_mSm}
    For $(X,D)\in \mSm_S$ and $(Y,E)\in \SmlSm_S$, we have
    $$
        \Hom_{\SmlSm_S}((X,|D|),(Y,E))\cong \colim_{\varepsilon\to 0} \Hom_{\mSm_S}((X,D),(Y,\varepsilon E)).
    $$
\end{lemma}

\begin{proof}
    Let $f\colon X\to Y$ be an $S$-morphism.
    Then $f$ gives a morphism of log pairs $(X,|D|)\to (Y,E)$ if and only if $|f^*E|\subset |D|$.
    Since $X$ is quasi-compact, this assumption is equivalent to saying that $n D\geq f^*E$ holds for $n\gg 0$, that is, $f\in \colim_{n\to \infty} \Hom_{\mSm_S}((X,nD),(Y,E))$.
    The claim follows from the canonical isomorphism
    $\Hom_{\mSm_S}((X,nD),(Y,E))\cong \Hom_{\mSm_S}((X,D),(Y,(1/n) E))$.
\end{proof}

\begin{definition}
    We define a symmetric monoidal functor $t\colon \mSm_S\to \SmlSm_S$ by $(X,D)\mapsto (X,|D|)$.
    By Lemma \ref{SmlSm_vs_mSm}, the functor $t$ admits an ind-right adjoint $(X,D)\mapsto \text{``}\colim_{\varepsilon\to 0}\text{''}(X,\varepsilon D)$.
\end{definition}

\begin{lemma}\label{t_morphism_of_sites}
    The functor $t$ gives a morphism of sites $\SmlSm_{S,\sNis}\to \mSm_{S,\Nis}$.
\end{lemma}

\begin{proof}
    It is easy to see that $t$ gives a continuous map of sites.
    To prove that $t$ gives a morphism of sites, it suffices to show that the comma category $\mathfrak{X}\downarrow t$ is cofiltered for any $\mathfrak{X}=(X,D)\in \SmlSm$.
    Let $y_i=((Y_i,E_i),f_i\colon \mathfrak{X}\to (Y_i,|E_i|))\;(i=1,2)$ be two objects in $\mathfrak{X}\downarrow t$.
    By Lemma \ref{SmlSm_vs_mSm}, there is some $n>0$ such that $f_i$ gives a morphism $(X,nD)\to (Y_i,E_i)$ for $i=1,2$.
    We define an object $x=((X,nD),\id\colon \mathfrak{X}\to (X,|nD|))$ of $\mathfrak{X}\downarrow t$.
    Then $f_i$ gives a morphism $x\to y_i$ for $i=1,2$, so we obtain a diagram $y_1\leftarrow x\rightarrow y_2$ in $\mathfrak{X}\downarrow t$.
    If there are two morphisms $g,h\colon y_1\to y_2$ in $\mathfrak{X}\downarrow t$, then we have $g\circ f_1=f_2=h\circ f_1$ by definition of $\mathfrak{X}\downarrow t$.
    Therefore $\mathfrak{X}\downarrow t$ is cofiltered.
\end{proof}

By Lemma \ref{t_morphism_of_sites}, the functor $t\colon \mSm_S\to \SmlSm_S$ induces an adjunction
\begin{align}
\label{eq: t_sheaf_adjunction}
    t_!\colon \Sh_\Nis(\mSm_S,\Mod_\Lambda)\rightleftarrows \Sh_\Nis(\SmlSm_S,\Mod_\Lambda)\colon t^*,
\end{align}
where $t_!y(X,D)\simeq y(X,|D|)$ and $t^*F(X,D)\simeq F(X,|D|)$.
The functor $t_!$ is symmetric monoidal.
Since the functor $t$ admits an ind-right adjoint $(X,D)\mapsto \text{``}\colim_{\varepsilon\to 0}\text{''}(X,\varepsilon D)$, we have
$$
    t^*y(X,D)\simeq \colim_{\varepsilon\to 0}y(X,\varepsilon D).
$$
Moreover, the functor $t^*$ is fully faithful because $t_!\circ t^*\simeq \id$.

\begin{lemma}\label{t^*_preserves_colimits}
    The functor $t^*$ preserves colimits.
\end{lemma}

\begin{proof}
    For any $F\in \Sh_\Nis(\SmlSm_S,\Mod_\Lambda)$ and $\mathcal{X}=(X,D)\in \mSm_S$, we have $(t^*F)_{(X,D)}\simeq F_{(X,|D|)}$ as a sheaf on $X_\Nis$.
    Since the colimits in $\Sh_\Nis(\mSm_S,\Mod_\Lambda)$ and $\Sh_\Nis(\SmlSm_S,\Mod_\Lambda)$ can be computed on each $X_\Nis$, it follows that $t^*$ preserves colimits.
\end{proof}

\begin{lemma}\label{t^*_monoidal}
    The functor $t^*$ is symmetric monoidal.
\end{lemma}

\begin{proof}
    It suffices to show that for $F,G\in \Sh_\Nis(\SmlSm_S,\Mod_\Lambda)$, the canonical morphism
    $$
        (t^*F)\otimes (t^*G)\to t^*t_!((t^*F)\otimes (t^*G))\xrightarrow{\simeq} t^*(F\otimes G)
    $$
    is an equivalence.
    Since $t^*$ preserves colimits by Lemma \ref{t^*_preserves_colimits}, we may assume that $F=y(\mathfrak{X})$ and $G=y(\mathfrak{Y})$ for $\mathfrak{X}=(X,D)$, $\mathfrak{Y}=(Y,E)\in \SmlSm_S$.
    In this case, the morphism in question is
    $$
        \colim_{\varepsilon\to 0} y(X,\varepsilon D)\otimes
        \colim_{\varepsilon\to 0} y(Y,\varepsilon E)\to
        \colim_{\varepsilon\to 0} y(X\times Y,\varepsilon(\pr_1^*D+\pr_2^*E)), 
    $$
    which is clearly an equivalence.
\end{proof}

\begin{definition}
    We define a functor
    $$
        t_*\colon \Sh_\Nis(\mSm_S,\Mod_\Lambda)\to \Sh_\sNis(\SmlSm_S,\Mod_\Lambda)
    $$
    to be the right adjoint of $t^*$.
\end{definition}

In summary, there is a string of adjoint functors
    $$
        \begin{tikzcd}
            \Sh_\Nis(\mSm_S,\Mod_\Lambda)
            \arrow[rr,shift left=1.5ex,"t_!"]
            \arrow[rr,leftarrow,"t^*" description]
            \arrow[rr,shift right=1.5ex,"t_*"']
            &&
            \Sh_\sNis(\SmlSm_S,\Mod_\Lambda),
        \end{tikzcd}
    $$
where $t_!\dashv t^*$ is a symmetric monoidal adjunction, and $t^*\dashv t_*$ is an adjunction.
Moreover, $t^*$ is fully faithful, and we have 
\begin{align}
    t_!y(X,D)\simeq y(X,|D|),& \quad t^*F(X,D)\simeq F(X,|D|),\label{t_formula1}\\
    t^*y(X,D)\simeq \colim_{\varepsilon\to 0}y(X,\varepsilon D),& \quad t_*F(X,D)\simeq \lim_{\varepsilon\to 0} F(X,\varepsilon D).\label{t_formula2}
\end{align}

Our next task is to upgrade these adjunctions to the level of motivic stable homotopy categories.
By the above formula for the functor $t_!$, we see that the functor $t^*$ sends $(\CI^{\log}\cup\BI^{\log})$-local objects to $(\CI\cup\BI)$-local objects.
This implies that there is an induced adjunction
$$
    t_!\colon \mSH_{S^1}(S,\Lambda) \rightleftarrows \logSH_{S^1}(S,\Lambda) \colon t^*.
$$
As for the other adjunction $t^*\dashv t_*$, we need the following lemma:

\begin{lemma}\label{lem:t_*_preserves_local_objects}
    The functor $t_*\colon \Sh_\Nis(\mSm_S,\Mod_\Lambda)\to \Sh_\sNis(\SmlSm_S,\Mod_\Lambda)$ sends $(\CI\cup\BI)$-local objects to $(\CI^{\log}\cup\BI^{\log})$-local objects.
\end{lemma}

\begin{proof}
    It suffices to show that for any morphism $f\colon \mathfrak{X}\to \mathfrak{Y}$ in $\CI^{\log}\cup \BI^{\log}$, the image of the morphism $t^*y(f)$ under $\mathrm{L}_\mot\colon \Sh_\Nis(\mSm_S,\Mod_\Lambda)\to \mSH_{S^1}(S,\Lambda)$ is an equivalence.
    First we take $f\in \CI^{\log}$.
    In other words, we consider the canonical morphism $f\colon \mathfrak{X}\otimes \bcube\to \mathfrak{X}$, where $\mathfrak{X}=(X,D)\in \SmlSm_S$.
    By the formula (\ref{t_formula2}), we get
    $$
        t^*y(f)\simeq[\colim_{\varepsilon\to 0} y((X,\varepsilon D)\otimes \bcube^\varepsilon)\to \colim_{\varepsilon\to 0} y(X,\varepsilon D)].
    $$
    Lemma \ref{cube_epsilon_invariance} implies that for any $\varepsilon\in (0,1]\cap \mathbb{Q}$, the image of the morphism $y((X,\varepsilon D)\otimes \bcube^\varepsilon)\to y(X,\varepsilon D)$ under $\mathrm{L}_\mot$ is an equivalence.
    This proves the claim for $f\in \CI^{\log}$.

    Next we take $f\in \BI^{\log}$.
    In other words, we consider an SNC blow-up $f\colon \mathfrak{Y}\to \mathfrak{X}$, where $\mathcal{X}=(X,D),\mathcal{Y}=(Y,E)\in \SmlSm_S$.
    Then we have $|f^*D|=E$, and $f\colon (Y,f^*D)\to (X,D)$ is an SNC blow-up in $\mSm_S$.
    By the formula (\ref{t_formula2}), we get
    \begin{align*}
        t^*y(f)&\simeq[\colim_{\varepsilon\to 0} y(Y,\varepsilon E)\to \colim_{\varepsilon\to 0}y(X,\varepsilon D)]\\
        &\simeq [\colim_{\varepsilon\to 0} y(Y,\varepsilon (f^*D))\to \colim_{\varepsilon\to 0} y(X,\varepsilon D)].
    \end{align*}
    The blow-up invariance implies that the image of the morphism $y(Y,\varepsilon (f^*D))\to y(X,\varepsilon D)$ under $\mathrm{L}_\mot$ is an equivalence.
    This proves the claim for $f\in \BI^{\log}$.
\end{proof}

By Lemma \ref{lem:t_*_preserves_local_objects}, the functor $t^*$ commutes with the localization functor $\mathrm{L}_{\mot}$.
In particular, $t^*$ is symmetric monoidal.
In summary, we have the following result:

\begin{corollary}\label{mH_tH_adjunction}
There is a string of adjoint functors
    $$
        \begin{tikzcd}
            \mSH_{S^1}(S,\Lambda)
            \arrow[rr,shift left=1.5ex,"t_!"]
            \arrow[rr,leftarrow,"t^*" description]
            \arrow[rr,shift right=1.5ex,"t_*"']
            &&
            \logSH_{S^1}(S,\Lambda),
        \end{tikzcd}
    $$
where $t_!\dashv t^*$ is a symmetric monoidal adjunction, and $t^*\dashv t_*$ is an adjunction.
Moreover, $t^*$ is fully faithful and
$$
    t_!\motive(X,D)\simeq\motive^{\log}(X,|D|),\quad t^*\motive^{\log} (X,D)\simeq\colim_{\varepsilon\to 0}\motive(X,\varepsilon D).
$$
\end{corollary}

\begin{theorem}\label{tame_motive_HA}
    Let $\mathcal{X}=(X,D)\in \mSm_S$ and suppose that $D$ has multiplicity $\leq 1$.
    Then $\motive(\mathcal{X})\to t^*t_!\motive(\mathcal{X})$ is an equivalence in $\mSH^\eff(S,\Lambda)$.
\end{theorem}

\begin{proof}
    We have $t^*t_!\motive(\mathcal{X})=\colim_{\varepsilon\to 0}\motive(X,\varepsilon D)$.
    Therefore the claim follows immediately from the tame Hasse-Arf theorem (Theorem \ref{THA}).
\end{proof}

\subsection{Stabilization with respect to the Tate circle}

Throughout this subsection, we still keep the assumption that the base scheme $S$ is normal.
In this subsection, we construct the motivic stable homotopy category with modulus by imitating the construction of the logarithmic motivic stable homotopy category in \cite{BPO}.

\begin{definition}\label{def;Tatecircle}
    The \emph{Tate circle with modulus} is defined by
    $$
        S^1_t=\motive(\mathbb{P}^1,[0]+[\infty])/\motive(\{1\}) \in \mSH_{S^1}(S,\Lambda).
    $$
    By Theorem \ref{tame_motive_HA}, we have $S^1_t\simeq t^*S^{1,\log}_t$, where $S^{1,\log}_t$ is the log Tate circle
    $$
        S^{1,\log}_t=\motive^{\log}(\mathbb{P}^1,[0]+[\infty])/\motive^{\log}(\{1\}) \in \logSH_{S^1}(S,\Lambda)
    $$
    defined in \cite[Definition 2.4.3]{LogHom}.
\end{definition}

\begin{lemma}\label{lem;S1S^1_t}
    There is a canonical equivalence in $\mSH_{S^1}(S,\Lambda)$
    $$
        \motive(\mathbb{P}^1,\varnothing)/\motive(\{1\})\simeq S^1\otimes S^1_t.
    $$
\end{lemma}

\begin{proof}
    This statement is proved for $\logSH_{S^1}(S,\Lambda)$ in \cite[Proposition 2.4.11]{LogHom}.
    Applying the functor $t^*$, we get the desired equivalence in $\mSH(S,\Lambda)$.
\end{proof}

\begin{lemma}\label{lem: G_m is symmetric}
    The object $S^1_t\in \mSH_{S^1}(S,\Lambda)$ is symmetric.
    That is, the cyclic permutation on $(S^1_t)^{\otimes 3}$ is homotopic to the identity.
\end{lemma}

\begin{proof}
    This statement is proved for $\logSH_{S^1}(S,\Lambda)$ in \cite[Proposition 3.2.16]{LogHom}.
    Applying the functor $t^*$, we get the desired result in $\mSH(S,\Lambda)$.
\end{proof}

\begin{definition}\label{def;mSH}
    Let $\Lambda$ be a connective commutative ring spectrum.
    We define the \emph{$\mathbb{P}^1$-stable $\Lambda$-linear motivic homotopy category with modulus}
    $$
        \mSH(S,\Lambda)
    $$
    to be the formal inversion of $S^1_t$ in $\mSH_{S^1}(S,\Lambda)$:
    $$
        \mSH(S,\Lambda):=\mSH_{S^1}(S,\Lambda)[(S^1_t)^{-1}]
    $$
    (see \cite[Definition 2.6]{Robalo_Ktheory}).
    We write $\mSH(S)$ for $\mSH(S,\mathbb{S})$.
    When $\Lambda$ is a classical ring, we write $\mDA(S,\Lambda)$ for $\mSH(S,\Lambda)$.
    By Lemma \ref{lem: G_m is symmetric} and \cite[Corollary 2.22]{Robalo_Ktheory}, we have an equivalence of $\infty$-categories
    \begin{align*}
        \mSH(S,\Lambda) \simeq {}
        &\colim_{\PrL}\bigl(\mSH_{S^1}(S,\Lambda)
        \xrightarrow{\Sigma_{S^1_t}}\mSH_{S^1}(S,\Lambda)
        \xrightarrow{\Sigma_{S^1_t}}\cdots\bigr)\\
        \simeq {} & \lim_{\mathrm{Cat}_\infty}\bigl(\cdots \xrightarrow{\Omega_{S^1_t}}\mSH_{S^1}(S,\Lambda)
        \xrightarrow{\Omega_{S^1_t}}\mSH_{S^1}(S,\Lambda)\bigr),
    \end{align*}
    where $\Sigma_{S^1_t} = S^1_t\otimes({-})$ and $\Omega_{S^1_t}$ is its right adjoint. Recall that $\PrL$ denotes the $\infty$-category whose objects are presentable ∞-categories and whose morphisms are colimit-preserving functors. 
\end{definition}

\begin{lemma}\label{lem:adjoint_inversion}
Let $\mathcal{C},\mathcal{D}$ be presentably symmetric monoidal $\infty$-categories.
Suppose that there is a symmetric monoidal adjunction
    $$
        \begin{tikzcd}
            \mathcal{C}
            \arrow[rr,shift left=0.75ex,"F"]
            \arrow[rr,leftarrow, shift right=0.75ex,"G"']
            &&
            \mathcal{D}
        \end{tikzcd}
    $$
such that $G$ is fully faithful and colimit-preserving.
Let $T$ be an object of $\mathcal{D}$.
Then, there is an induced symmetric monoidal adjunction
    $$
        \begin{tikzcd}
            \mathcal{C}[G(T)^{-1}]
            \arrow[rr,shift left=0.75ex,"F'"]
            \arrow[rr,leftarrow, shift right=0.75ex,"G'"']
            &&
            \mathcal{D}[T^{-1}]
        \end{tikzcd}
    $$
such that $G'$ is fully faithful and colimit-preserving.
\end{lemma}

\begin{proof}
    We can regard $\mathcal{C}$ as an object of $\CAlg(\PrL)_{\mathcal{D}/}$ via $G$.
    Since we have $F\circ G\simeq \id$, the functor $G$ can be regarded as a morphism in $\CAlg(\PrL)_{\mathcal{D}/}$.
    By \cite[Proposition 2.9]{Robalo_Ktheory}, the formal inversion of $T$ is given by the tensor product functor
    $$
        ({-})\otimes_{\mathcal{D}}\mathcal{D}[T^{-1}]\colon \CAlg(\PrL)_{\mathcal{D}/}\to \CAlg(\PrL)_{\mathcal{D}[T^{-1}]/}.
    $$
    Moreover, since this functor is $\CAlg(\PrL)$-linear, it actually gives a $2$-functor between the $(\infty,2)$-enhancement of these categories (see \cite[Theorem 1.1.2]{Stefanich}).
    Therefore, the symmetric monoidal adjunction $F\dashv G$ induces a symmetric monoidal adjunction $F'\dashv G'$ between $\mathcal{C}[G(T)^{-1}]$ and $\mathcal{D}[T^{-1}]$, where $G'$ is colimit-preserving.
    Since $F\circ G\simeq \id$, we have $F'\circ G'\simeq \id$, which means that $G'$ is fully faithful.
\end{proof}

\begin{remark}
    There is an alternative proof of Lemma \ref{lem:adjoint_inversion} based on \cite[Proposition 1.3.14, Lemma 1.5.4]{Annala-Iwasa}. We thank Ryomei Iwasa for pointing this out. 
    The nontrivial part is to see that the induced functors $F',G'$ are adjoint to each other. By {\it loc. cit.},  $F'$ is a localization with respect to the class of morphisms consisting of the ``shifts'' of $F$-equivalences, and hence the objects in the essential image of $G'$ are readily local for $F'$-equivalences. Since $F'G' \simeq \id$ by construction, this characterizes $G'$ as a right adjoint of $F'$. Note that we haven't used the fact that $G$ is colimit-preserving for the existence of the adjunction. If $G$ is colimit preserving, then so is $G'$ by construction.
\end{remark}

\begin{corollary}\label{mSH_tSH_adjunction}
There is a string of adjoint functors
    $$
        \begin{tikzcd}
            \mSH(S,\Lambda)
            \arrow[rr,shift left=1.5ex,"t_!"]
            \arrow[rr,leftarrow,"t^*" description]
            \arrow[rr,shift right=1.5ex,"t_*"']
            &&
            \logSH(S,\Lambda)
        \end{tikzcd}
    $$
where $t_!\dashv t^*$ is a symmetric monoidal adjunction, and $t^*\dashv t_*$ is an adjunction.
Moreover, $t^*$ is fully faithful and
$$
    t_!\motive(X,D)\simeq\motive^{\log}(X,|D|),\quad t^*\motive^{\log} (X,D)\simeq\colim_{\varepsilon\to 0}\motive(X,\varepsilon D).
$$
\end{corollary}

\begin{proof}
    This follows from Corollary \ref{mH_tH_adjunction} and Lemma \ref{lem:adjoint_inversion}.
\end{proof}

\subsection{Relation with Annala--Iwasa's category} \label{subsec:AI}

Annala--Iwasa \cite{Annala-Iwasa} defined the category $\MS_S$ (following the notation of \cite{Annala-Hoyois-Iwasa}) of motivic spectra as follows. We do not need to assume that $S$ is normal blow.

\begin{definition}\label{def:EBU}
     A closed immersion $i: Z \hookrightarrow X$ in $\Sm_S$ is called \emph{elementary} if, Zariski-locally on $X$, it is the zero section of $\mathbb{A}^n_Z \sqcup Y$ for some $n\geq 0$ and some $Y\in \Sm_S$.
     Let $F\in \PSh(\Sm_S, \Sp)$ be a presheaf of spectra on $\Sm_S$.
     We say that $F$ satisfies \emph{elementary blow-up excision} if $F(\varnothing)=0$ and for any elementary closed immersion $i\colon Z\hookrightarrow X$ in $\Sm_S$, $F$ sends the blow-up square
     \begin{align}\label{eq:EBU}
       \xymatrix{
       E\ar@{^(->}[r]\ar[d]&\Bl_ZX\ar[d]\\
       Z\ar@{^(->}[r]^-i&X
       } 
     \end{align}
     to a Cartesian square.
     Consider the full subcategory of $\PSh(\Sm_S, \Sp)$ spanned by presheaves satisfying Zariski descent and elementary blow-up excision.
     The category $\MS_S$ is defined to be the formal inversion of the pointed projective line $\mathbb{P}^1$ in this $\infty$-category.
\end{definition}

Let us compare $\MS_S$ with our category $\mSH(S)$.
Consider the symmetric monoidal functor
$$
    \lambda\colon \Sm_S\to \mSH(S);\quad X\mapsto \motive(X,\varnothing).
$$
By the Nisnevich descent in $\mSH(S)$, the functor $\lambda$ sends Zariski distinguished squares to coCartesian squares.
The smooth blow-up excision in $\mSH(S)$ (Theorem \ref{SBU}) implies that, for any elementary closed immersion $i\colon Z\to X$, the functor $\lambda$ sends the blow-up square \eqref{eq:EBU} to a coCartesian square.
Moreover, by Lemma \ref{lem;S1S^1_t}, the object $\motive(\mathbb{P}^1,\varnothing)/\motive(\{1\})$ is invertible in $\mSH(S)$.
Therefore, we get a symmetric monoidal, colimit-preserving functor
$$
    \lambda_!\colon \MS_S\to \mSH(S)
$$
which satisfies $\lambda_!(\Sigma^\infty_{\mathbb{P}^1}X_+)\simeq \motive(X)$.
Setting $\lambda^*$ to be the right adjoint of $\lambda_!$, we get the following:

\begin{theorem}\label{thm:MS_comparison}
    There is an adjunction
    $$
        \lambda_!\colon \MS_S\rightleftarrows\mSH(S)\colon \lambda^*
    $$
    where $\lambda_!$ is symmetric monoidal and $\lambda_!(\Sigma^\infty_{\mathbb{P}^1}X_+)\simeq \motive(X)$. 
\end{theorem}

In summary, we have the following sequence of symmetric monoidal, colimit-preserving functors:
$$
    \MS_S\xrightarrow{\lambda_!}\mSH(S)\xrightarrow{t_!}\logSH(S)\xrightarrow{\omega_{\sharp}}\SH(S).
$$

\subsection{$\mathbb{A}^1$-localization} \label{sec:A1-loc}

Again, throughout this subsection, we assume that the base scheme $S$ is normal.
Finally, we will see briefly that the natural functor $\mSH(S) \to \SH(S)$ is obtained by the $\AA^1$-localization. 
In other words, we prove the following result.

\begin{theorem} \label{thm:mSHSH}
    Let $\cC \in \{\Sp,\Mod_{\Lambda}\}$.
    The functor $\mSH(S,\cC) \to \SH(S,\cC)$ factors through the $\AA^1$-localization functor $L_{\AA^1} : \mSH(S,\cC) \to (\AA^1)^{-1}\mSH(S,\cC)$. Moreover, the induced functor $(\AA^1)^{-1}\mSH(S,\cC) \to \SH(S,\cC)$ is an equivalence, as a funny application of the tame Hasse-Arf theorem. Here, we identify $\AA^1$ with the $\mathbb{Q}$-modulus pair $(\AA^1,\varnothing)$.
\end{theorem}

In the following, we omit the coefficient category $\cC$ for the simplicity of notation. 
We know that the analogous functor 
\[
    (\AA^1)^{-1} \logSH(S) \to \SH(S)
\]
exists and is an equivalence, thanks to \cite[Remark 4.0.9]{LogHom}. Therefore, to prove Theorem \ref{thm:mSHSH}, it suffices to show that the functor $t^*\colon \logSH(S) \to \mSH(S)$ induces an equivalence 
\[
    (\AA^1)^{-1}\logSH(S) \xrightarrow{\sim} (\AA^1)^{-1}\mSH(S).
\]
Since $t_!$ sends $\motive(\mathbb{A}^1)$ to $\motive^{\log}(\mathbb{A}^1)$, it follows by adjunction that $t^*$ sends $\mathbb{A}^1$-local objects to $\mathbb{A}^1$-local objects.
Thus, we are reduced to showing that the fully faithful functor $(\AA^1)^{-1} \logSH(S) \to (\AA^1)^{-1} \mSH(S)$ is essentially surjective. 

\begin{lemma} \label{lem:ncube}
    Let $r$ be a positive rational number, and let $\CI^{(r)}$ denote the class of morphisms $\{\mathrm{pr}_1\colon \sX \otimes \bcube^{(r)} \to \sX\mid \sX \in \mSm_S\}$ in $\mSm_S$, where $\bcube^{(r)} := (\PP^1,r[\infty])$.
    Let $\mH^{(r)}(S)$ denote the full subcategory of $\Sh_{\Nis}(\mSm_S,\Spc)$ consisting of $(\CI^{(r)},\BI)$-local objects.
    Then, for any $\sX=(X,D) \in \mSm_S$ such that the multiplicity of each component of $D$ is less than or equal to $r$, the natural morphism 
    \[
         \motive (X,r|D|) \xrightarrow{\sim} \motive (X,D)
    \]
    is an isomorphism.
\end{lemma}

\begin{remark}
    This lemma says that inverting $(\PP^1,r[\infty])$ neglects the multiplicity $\leq r$.
\end{remark}

\begin{proof}
    Consider the automorphism $\mSm \to \mSm; \sX \to \sX^{(r)} := (X,rD)$, which obviously induces an equivalence $\mH(S) \to \mH^{(r)}(S)$. This sends $\motive (X,D)$ to $\motive (X,rD)$ by construction. Therefore, the statement immediately follows from the tame Hasse-Arf theorem (Theorem \ref{THA}) for $\mH(S)$.
\end{proof}

\begin{lemma} \label{lem:rHA}
    The functor $\mH(S) \to (\AA^1)^{-1} \mH(S)$ factors through $\mH^{(r)}(S)$ for any positive rational number $r$.
    In particular, for any $(X,D) \in \mSm$, there exists a canonical equivalence $\motive (X,D) \simeq \motive (X,|D|)$ in $(\AA^1)^{-1} \mH(S)$. 
    Consequently, the same assertion also holds in $(\AA^1)^{-1}\mSH(S)$.
\end{lemma}

\begin{proof}
    To prove the first assertion, it suffices to show that the morphism $\motive (\sX \otimes \bcube^{(r)}) \to \motive (\sX)$ is an equivalence in $(\AA^1)^{-1} \mH(S)$.
    The proof of Lemma \ref{cube_epsilon_invariance} applies after replacing all $\bcube$ with $\AA^1=(\AA^1,\varnothing)$, $\bcube^{\varepsilon}$ with $\bcube^{(r)}$, and the blow-up $\Bl_{(0,\infty), (\infty,0)} (\PP^1 \times \PP^1)$ with its open subscheme $\Bl_{(\infty,0)} (\PP^1 \times \AA^1)$.
    The rest of the assertion then follows from Lemma \ref{lem:ncube}.
\end{proof}

Now, we are ready to prove the main theorem in this subsection. 

\begin{proof}[Proof of Theorem \ref{thm:mSHSH}]
    Let $C$ be an $\AA^1$-local object in $\mSH(S)$. We want to prove that there exists an $\AA^1$-local object $D$ in $\logSH(S)$ with $C \simeq t^*D$.
    
    Since $t^*$ sends $\motive^{\log}(\mathbb{A}^1)$ to $\motive(\mathbb{A}^1)$, it follows by adjunction that the functor $t_*$ sends $\mathbb{A}^1$-local objects to $\mathbb{A}^1$-local objects.
    In particular, $t_*C$ is $\mathbb{A}^1$-local.
    Thus, it remains to prove $t^*t_*C\simeq C$.
    Note that $\mSH(S)$ is compactly generated by representables (see e.g. \cite[Corollary 1.5.3]{Annala-Iwasa}).
    For any $\sX \in \mSm$ and $i \in \Z$, we compute:
    \begin{align*}
        \map_{\mSH(S)}(\motive (\sX)[i],t^*t_*C)
        &{}\simeq \map_{\mSH(S)} (t^*t_!\motive (\sX)[i],C) \\
        &{}\simeq \map_{\mSH(S)} (\motive (\sX_{\red})[i],C) \\
        &{}\simeq \map_{\mSH(S)} (\motive (\sX)[i],C),
   \end{align*}
   where we used $t^*t_!\motive (X) \simeq \motive (\sX_{\red})$ with $\sX_{\red}:=(X,|D|)$ from Corollary \ref{mSH_tSH_adjunction} and Theorem \ref{THA}, and the last isomorphism follows from Lemma \ref{lem:rHA}.
   This proves the desired claim.
\end{proof}

\section{Gysin sequence}

In this section we construct the Gysin sequence in $\mSH_{S^1}(S)$ following \cite{Matsumoto}.
We assume that the base scheme $S$ is normal.

\subsection{Thom space}

\begin{definition} \label{def:vectorbundle}
    Let $\mathcal{X}=(X,D) \in \mSm_{S}$. A \emph{vector bundle of rank $n$ on $\mathcal{X}$} is a morphism $p:\mathcal{V}=(V,D_V) \to \mathcal{X}$ such that the underlying morphism $p\colon V \to X$ is a vector bundle of rank $n$ and $D_V = p^*D$.
    For a vector bundle $p\colon \mathcal{V}=(V,D_V)\to \mathcal{X}=(X,D)$, we set
    $$
        \mathbb{P}(\mathcal{V})=(\mathbb{P}(V), \pi^*D),
    $$
    where $\pi\colon \mathbb{P}(V)\to X$ is the canonical projection.
\end{definition}

\begin{remark}
    For any vector bundle $p:(V,D_V) \to (X,D)$ in the sense of Definition \ref{def:vectorbundle}, the zero section $X \to V$ induces a morphism $s : (X,D) \to (V,D_V)$ such that $ps=\id$. We call this morphism the zero section, too.
\end{remark}

\begin{definition} \label{def:thomsp}
    Let $\mathcal{X}=(X,D)\in \mSm_{S}$ and $p:\mathcal{V}=(V,D_V) \to \mathcal{X}$ be a vector bundle of rank $d$.
    Consider the blow up $q:\Bl_X(V) \to V$ along the zero section $X \hookrightarrow V$. Then the \emph{Thom space associated to $p$} is defined by
    \[
        \MTh(\mathcal{V}) := \cofib \bigl(\motive (\Bl_X(V),q^*D_V + E) \to \motive (V,D_V)\bigr),
    \]
    where $E$ denotes the exceptional divisor.
\end{definition}

\begin{proposition} \label{mthiso}
    Let $\sX=(X,D) \in \mSm_{S}$ and $\sV=(V,D_V) \to \sX$ be a vector bundle. 
    Regard $\mathbb{P}(\sV)$ as the closed subscheme $\mathbb{P}(\sV \oplus \OO)$ at infinity. Then there exists a canonical equivalence 
    \[
        \MTh(\sV) \xrightarrow{\sim} \cofib \bigl(\motive (\PP(\sV)) \to \motive (\PP (\sV \oplus \OO))\bigr).
    \]
\end{proposition}

\begin{proof}
    Let $\sY=(Y,D_Y) \to \sV$ and $\sY'=(Y',D_{Y'}) \to \PP(\sV \oplus \OO)$ be the blow-ups along the zero sections, and let $E \subset Y$ be the exceptional divisor. Note that $V=\PP(V \oplus \OO) - \PP(V) \subset \PP(V \oplus \OO)$ is an open subscheme, and hence we obtain a commutative diagram 
    \[\xymatrix{
        (Y,D_Y+E) \ar[r] \ar[d] & (Y',D_{Y'} + E) \ar[d] \\
        \sV \ar[r] & (\PP(V\oplus \OO),h^*D)
    }\]
    where $h:\PP(V\oplus \OO) \to X$ is the projection. 
    The underlying diagram of schemes is a pullback square. 
    After applying the motive functor $\motive$, the Nisnevich descent induces an equivalence between the cofibers of the vertical morphisms:
    \begin{equation} \label{eq0001}
        \MTh(\sV) \xrightarrow{\sim} \cofib \bigl(\motive (Y',D_{Y'}+E) \to \motive (\PP(V\oplus \OO),h^*D)\bigr).
    \end{equation}

    Consider the closed subscheme $Z:=\PP(V)$ at infinity.
    There is a canonical inclusion $\sZ:=(Z,D_Z) \hookrightarrow (Y',D_{Y'}+E)$, where $D_Z := D_{Y'}|_Z$.
    By the equivalence \eqref{eq0001}, it suffices to show that the induced morphism $\motive (\sZ) \to \motive (Y',D_{Y'}+E)$ is an equivalence.

    To see this, we may assume that the vector bundle $V$ is trivial since the problem is Zariski local.
    Let $\pi\colon \mathbb{P}(V)=\mathbb{P}^{n-1}_X\to X$ denote the canonical projection.
        Since the composite
        \[
            (Z,D_Z) \hookrightarrow (Y',D_{Y'}+E) \twoheadrightarrow \PP(\sV) = (\PP^{n-1}_X, \pi^*D)
        \]
        is an equivalence by construction, it suffices to show that the induced morphism 
        \[
            \motive (Y',D_{Y'}+E) \twoheadrightarrow \motive (\PP^{n-1}_X, \pi^*D)
        \]
        is an equivalence, but this follows from the fact that $(Y',D_{Y'}+E') \to (\PP^{n-1}_X, \pi^*D)$ is a $\bcube$-bundle.
\end{proof}

\begin{theorem}[Gysin sequence]\label{GS}
    Let $\mathcal{X}=(X,D)\in \mSm_S$ and let $Z\subset X$ be a smooth closed subscheme which is transversal to $|D|$.
    Let $\pi\colon \normal_ZX\to Z$ be the normal bundle of $Z$, and set $\mathcal{N}_ZX:=(\normal_ZX,\pi^*(D|_Z))$.
    Then there exists a canonical cofiber sequence
    $$
        \motive(\Bl_ZX,q^*D+E)\to \motive(\mathcal{X})\to \MTh(\mathcal{N}_ZX)
    $$
    in $\mSH_{S^1}(S)$, where $q\colon \Bl_ZX\to X$ is the blow-up along $Z$ and $E$ is the exceptional divisor.
\end{theorem}

\subsection{Reduction to the case of divisors}

First we reduce to the case where $Z$ has codimension $1$:

\begin{theorem}[Gysin sequence for divisors]\label{GS_divisor}
    Let $\mathcal{X}=(X,D)\in \mSm_S$ and let $Z\subset X$ be a smooth divisor which is transversal to $|D|$.
    Let $\pi\colon \normal_ZX\to Z$ be the normal bundle of $Z$, and set $\mathcal{N}_ZX:=(\normal_ZX,\pi^*(D|_Z))$.
    Then there exists a canonical cofiber sequence
    $$
        \motive(X,D+E)\to \motive(\mathcal{X})\to \MTh(\mathcal{N}_ZX)
    $$
    in $\mSH_{S^1}(S)$.
\end{theorem}

Let $\mathcal{X}=(X,D)\in \mSm_S$ and let $Z\subset X$ be a smooth closed subscheme which is transversal to $|D|$.
We write $q\colon \Bl_ZX\to X$ for the blow-up of $X$ along $Z$.
Let $\pi\colon \normal_ZX\to Z$ be the normal bundle of $Z$, and set $\mathcal{N}_ZX:=(\normal_ZX,\pi^*(D|_Z))$.
Similarly, we write $\tau\colon \mathrm{N}_E\Bl_ZX\to E$ for the normal bundle of $E$, and set $\mathcal{N}_E\Bl_ZX:=(\mathrm{N}_E\Bl_ZX,\tau^*(q^*D|_E))$.

\begin{lemma}\label{lem:blowup_normal_exchange}
    There is a canonical isomorphism of $S$-schemes $\mathrm{N}_E\Bl_ZX\cong \Bl_Z(\mathrm{N}_ZX)$.
    Moreover, the composition
    $$
        \mathrm{N}_E\Bl_ZX\cong \Bl_Z(\mathrm{N}_ZX)\to \mathrm{N}_ZX
    $$
    induces an isomorphism $\mathbb{P}(\mathrm{N}_E\Bl_ZX)\cong \mathbb{P}(\mathrm{N}_ZX)$.
\end{lemma}

\begin{proof}
    The first part can be easily checked by a local computation.
    As for the second part, observe that the induced morphism
    $$
        \mathrm{N}_E\Bl_ZX-E\to \mathrm{N}_ZX-Z
    $$
    is an isomorphism, since $E$ is the exceptional divisor of the blow-up $\Bl_ZX\to X$.
    Taking the quotient of both sides with respect to the canonical $\mathbb{G}_m$-action, we obtain the desired isomorphism.
\end{proof}

\begin{lemma}\label{lem:GS_reduction_lemma}
    The following diagram in $\mSH_{S^1}(S)$ is coCartesian:
    \[
    \xymatrix{
    \motive(E,q^*D|_E)\ar[r]\ar[d]&\MTh(\mathcal{N}_E\Bl_ZX)\ar[d]\\
    \motive(Z,D|_Z)\ar[r]&\MTh(\mathcal{N}_ZX).
    }
    \]
\end{lemma}

\begin{proof}
    We use the following notation for the canonical projections:
    \begin{align*}
        &\overline{\tau}\colon \mathbb{P}(\mathrm{N}_E\Bl_ZX\oplus \mathcal{O})\to E,\\
        &\overline{\pi}\colon \mathbb{P}(\mathrm{N}_ZX\oplus \mathcal{O})\to Z,\\
        &\tau_0\colon \mathbb{P}(\mathrm{N}_E\Bl_ZX)\to E,\\
        &\pi_0\colon \mathbb{P}(\mathrm{N}_ZX)\to Z.
    \end{align*}
    By Lemma \ref{lem:blowup_normal_exchange}, $\mathbb{P}(\mathrm{N}_E\Bl_ZX\oplus \mathcal{O})$ can be regarded as the blow-up of $\mathbb{P}(\mathrm{N}_ZX\oplus \mathcal{O})$ along the zero section.
    By the smooth blow-up excision in $\mSH_{S^1}(S)$ (Theorem \ref{SBU}), we have the following coCartesian diagram:
    \[
    \xymatrix{
        \motive(E,q^*D|_E)\ar[r]\ar[d]&\motive(\mathbb{P}(\mathrm{N}_E\Bl_ZX\oplus \mathcal{O}), \overline{\tau}^*(q^*D|_E))\ar[d]\\
        \motive(Z,D|_Z)\ar[r]&\motive(\mathbb{P}(\mathrm{N}_ZX\oplus \mathcal{O}), \overline{\pi}^*(D|_Z)).
    }
    \]
    Moreover, the induced morphism
    $$
        \motive(\mathbb{P}(\mathrm{N}_E\Bl_ZX),\tau_0^*(q^*D|_E))\to \motive(\mathbb{P}(\mathrm{N}_ZX),\pi_0^*(D|_Z))
    $$
    is an equivalence by Lemma \ref{lem:blowup_normal_exchange}.
    Therefore, the claim follows from Proposition \ref{mthiso}.
\end{proof}

\begin{proof}[Proof of Theorem \ref{GS} from Theorem \ref{GS_divisor}]
    Consider the following diagram:
    \[
    \xymatrix{
        \motive(E,q^*D|_E)\ar[r]\ar[d]&
        \motive(\Bl_ZX,q^*D)\ar[r]\ar[d]&
        \MTh(\mathcal{N}_E\Bl_ZX)\ar[d]\\
        \motive(Z,D|_Z)\ar[r]\ar@/_18pt/[rr]&
        \motive(X,D)\ar@{-->}[r]&
        \MTh(\mathcal{N}_ZX).
    }
    \]
    \vskip\baselineskip
    The left square is coCartesian by the smooth blow-up excision (Theorem \ref{SBU}), and the total rectangle is coCartesian by Lemma \ref{lem:GS_reduction_lemma}.
    Therefore, the dashed arrow is canonically induced by the universal property, and the right square is also coCartesian.
    In particular, we obtain a canonical equivalence
    \begin{align*}
        \fib(\motive(X,D)\to \MTh(\mathcal{N}_ZX))\simeq \fib(\motive(\Bl_ZX,q^*D)\to \MTh(\mathcal{N}_E\Bl_ZX)).
    \end{align*}
    The right hand side is equivalent to $\motive(\Bl_ZX,q^*D+E)$ by our assumption that Theorem \ref{GS_divisor} is true.
    Therefore we obtain the desired cofiber sequence.
\end{proof}

\subsection{Modified blow-up excision}

To prove Theorem \ref{GS_divisor}, we need the following auxiliary result:

\begin{proposition}[Modified blow-up excision]\label{MBU}
    Let $\mathcal{X}=(X,D) \in \mSm_S$ and let $Z\subset X$ be a smooth divisor which is transversal to $|D|$.
    For any scheme $Y$ over $X\times \mathbb{P}^1$, we set
    $$
        D_Y:=(\pr_1^*D+\pr_2^*[\infty])|_Y.
    $$
    Let $B = \Bl_{Z\times\{0\}}(X\times\mathbb{P}^1)$ and let $E$ denote the exceptional divisor.
    Let $Z_B$ be the strict transform of $Z\times \mathbb{P}^1$ and $Z_E=Z_B|_E$.
    Then the following square in $\mSH_{S^1}(S,\Lambda)$ is coCartesian:
    \begin{align}\label{MBU_square}
    \xymatrix{
        \motive(E,D_E+Z_E)\ar[r]\ar[d]  &\motive(B,D_B+Z_B)\ar[d]\\
        \motive(Z,D_Z)\ar[r]            &\motive(\mathcal{X}\otimes \bcube).
    }
    \end{align}
\end{proposition}

\begin{remark}
    In the situation of Proposition \ref{MBU}, we have the coCartesian square
    \begin{align}\label{MBU_square'}
    \xymatrix{
        \motive(E,D_E)\ar[r]\ar[d]  &\motive(B,D_B)\ar[d]\\
        \motive(Z,D_Z)\ar[r]            &\motive(\mathcal{X}\otimes \bcube)
    }
    \end{align}
    by the smooth blow-up excision (Theorem \ref{SBU}).
    Proposition \ref{MBU} can be regarded as a slight modification of this square.
\end{remark}

\begin{definition}
    Let $\mathcal{X}=(X,D)\in \mSm_S$ and let $Z\subset X$ be a smooth divisor which is transversal to $|D|$.
    We say that $(\mathrm{MBU})_{(\mathcal{X},Z)}$ holds if the conclusion of Proposition \ref{MBU} holds for $\mathcal{X}$ and $Z$.
\end{definition}

\begin{lemma}\label{MBU_tensor}
    Let $\mathcal{X}=(X,D)\in \mSm_S$ and let $Z\subset X$ be a smooth divisor which is transversal to $|D|$.
    If $(\mathrm{MBU})_{(\mathcal{X},Z)}$ holds, then $(\mathrm{MBU})_{(\mathcal{X}\otimes\mathcal{Y},Z\times Y)}$ holds for any $\mathcal{Y}=(Y,E)\in \mSm_S$.
\end{lemma}

\begin{proof}
    This is a consequence of the fact that the functor $\motive(Y,E)\otimes({-})$ preserves colimits.
\end{proof}

\begin{lemma}\label{MBU_descent}
    Let $\mathcal{X}=(X,D)\in \mSm_S$ and let $Z\subset X$ be a smooth divisor which is transversal to $|D|$.
    Let $\{\mathcal{U}_i\}_{i\in I}$ be a Zariski covering of $\mathcal{X}$.
    For each finite non-empty subset $J\subset I$, we set $\mathcal{U}_J=\bigcap_{j\in J}\mathcal{U}_j$.
    If $(\mathrm{MBU})_{(\mathcal{U}_J,Z_J)}$ holds for every finite non-empty subset $J\subset I$, then $(\mathrm{MBU})_{(\mathcal{X},Z)}$ holds.
\end{lemma}

\begin{proof}
    This follows immediately from the Nisnevich descent.
\end{proof}

\begin{lemma}\label{MBU_A^1}
    $(\mathrm{MBU})_{((\mathbb{A}^1,\varnothing),0)}$ holds.
\end{lemma}

\begin{proof}
    Let $E$ denote the exceptional divisor of the blow-up $\Bl_{(0,0)}(\mathbb{A}^1\times\mathbb{P}^1)\to \mathbb{A}^1\times\mathbb{P}^1$.
    We write $V\subset \Bl_{(0,0)}(\mathbb{A}^1\times\mathbb{P}^1)$ for the strict transform of $\{0\}\times\mathbb{P}^1$, and set $W=\mathbb{A}^1\times\{\infty\}$.
    We need to show that the following diagram in $\mSH_{S^1}(S,\Lambda)$ is coCartesian: 
    \begin{align}\label{MBU_square1}
    \xymatrix{
        \motive(E,V|_E)\ar[r]\ar[d]  &\motive(\Bl_{(0,0)}\motive(\mathbb{A}^1\times\mathbb{P}^1),V+W)\ar[d]\\
        \motive(\pt)\ar[r]^-{(0,0)}            &\motive(\mathbb{A}^1\times\mathbb{P}^1, W).
    }
    \end{align}
    All divisors appearing in the above diagram have multiplicity $\leq 1$.
    Therefore, by Theorem \ref{tame_motive_HA}, it suffices to show that the image of (\ref{MBU_square1}) under $t_!\colon \mSH_{S^1}(S,\Lambda)\to \logSH_{S^1}(S,\Lambda)$ is coCartesian.
    In other words, it suffices to show that the following diagram in $\logSH_{S^1}(S,\Lambda)$ is coCartesian:
    \begin{align}\label{MBU_square1log}
    \xymatrix{
        \motive^{\log}(E,V|_E)\ar[r]\ar[d]  &\motive^{\log}(\Bl_{(0,0)}(\mathbb{A}^1\times\mathbb{P}^1),V+W)\ar[d]\\
        \motive^{\log}(\pt)\ar[r]^-{(0,0)}            &\motive^{\log}(\mathbb{A}^1\times\mathbb{P}^1, W).
    }
    \end{align}
    
    We set $Q=\Bl_{[0:1:0]}\mathbb{P}^2$.
    The open immersion
    $$\mathbb{A}^1\times \mathbb{A}^1\hookrightarrow \mathbb{P}^2-\{[0:1:0]\};\quad (x,y)\mapsto [x:y:1]$$
    extends to an open immersion $\mathbb{A}^1\times\mathbb{P}^1\hookrightarrow Q$.
    We regard $\mathbb{A}^1\times\mathbb{P}^1$ as an open subscheme of $Q$ via this open immersion.
    The closure $\overline{W}$ of $W$ in $Q$ is the exceptional divisor of $Q\to \mathbb{P}^2$.
    Let $L\subset Q$ denote the strict transform of $\{[x:y:z]\mid z=0\}\subset \mathbb{P}^2$, which is equal to the complement $Q- \mathbb{A}^1\times\mathbb{P}^1$.
    Then (\ref{MBU_square1log}) is (the image under $\motive^{\log}$ of) the restriction of the following diagram to $\mathbb{A}^1\times\mathbb{P}^1\subset Q$:
    \begin{align}\label{MBU_square2}
    \xymatrix{
        (E,V|_E)\ar[r]\ar[d]  &(\Bl_{[0:0:1]}Q,V+\overline{W}+L)\ar[d]\\
        \pt \ar[r]^-{[0:0:1]}            &(Q, \overline{W}+L).
    }
    \end{align}
    Let $U=Q-(E\cup V)$.
    Then, $Q$ is covered by two open subschemes $\mathbb{A}^1\times\mathbb{P}^1$, $U$.
    The vertical morphisms of the above diagram are isomorphisms over $U\subset Q$.
    Therefore, by the Nisnevich descent, it suffices to show that the image under $\motive^{\log}$ of the diagram (\ref{MBU_square2}) is coCartesian.

    We make use of the canonical projections $\pi\colon Q\to \mathbb{P}^2$ and $\varpi\colon \Bl_{[0:0:1]}Q\to \Bl_{[0:0:1]}\mathbb{P}^2$.
    By the blow-up invariance of motives, the morphism $\motive^{\log}(\pt)\xrightarrow{[0:0:1]}\motive^{\log}(Q, \overline{W}+L)$ isomorphic to $\motive^{\log}(\pt)\xrightarrow{[0:0:1]}\motive^{\log}(\mathbb{P}^2, \pi(L))$, which is an isomorphism by Lemma \ref{P^n_relative_motive} (1).
    On the other hand, by the blow-up invariance, the morphism
    $$
        \motive^{\log}(E,V|_E)\to\motive^{\log}(\Bl_{[0:0:1]}Q,V+\overline{W}+L)
    $$
    is equivalent to
    $$
    \motive^{\log}(E,V|_E)\to
    \motive^{\log}(\Bl_{[0:0:1]}\mathbb{P}^2,\varpi(V)+\varpi(L)),
    $$
    which is an equivalence since $(\Bl_{[0:0:1]}\mathbb{P}^2,\varpi(V)+\varpi(L))$ is a cube-bundle over $(E,V|_E)$.
    This shows that the diagram (\ref{MBU_square2}) is coCartesian.
\end{proof}

\begin{lemma}\label{MBU_excision}
    Let $\mathcal{X}=(X,D)\in \mSm_S$ and let $Z\subset X$ be a smooth divisor which is transversal to $|D|$.
    Let $X'\to X$ be an \'etale morphism which induces an isomorphism $X'\times_XZ\cong Z$, and set $\mathcal{X}'=(X',D|_{X'})$.
    Then $(\mathrm{MBU})_{(\mathcal{X},Z)}$ holds if and only if $(\mathrm{MBU})_{(\mathcal{X}',Z)}$ holds.
\end{lemma}

\begin{proof}
    Let $U=X- Z$ and $U'=X'- Z$.
    Consider the following commutative diagram:
    $$
    \xymatrix{
        \motive(U',D|_{U'})\ar[d]\ar[r] &\motive(X',D|_{X'}+Z)\ar[r]\ar[d]       &\motive(X',D|_{X'}+\varepsilon Z)\ar[d]\\
        \motive(U,D|_U)\ar[r] &\motive(X,D+Z)\ar[r] &\motive(X,D+\varepsilon Z).
    }
    $$
    The left square and the total rectangle are coCartesian by the Nisnevich descent.
    By the pasting law, the right square is also coCartesian.
    This proves the claim.
\end{proof}

\begin{proof}[Proof of Proposition \ref{MBU}]
    As in the proof of Theorem \ref{SBU}, we may reduce to the case $\mathcal{X}=(Z,D|_Z)\otimes (\mathbb{A}^1,\varnothing)$ using Lemma \ref{transversal_structure}, Lemma \ref{MBU_descent}, and Lemma \ref{MBU_excision}.
    In this case the claim follows from Lemma \ref{MBU_tensor} and Lemma \ref{MBU_A^1}.
\end{proof}

\subsection{Gysin sequence}
We return to the proof of Theorem \ref{GS}.
First we construct the Gysin map.

\begin{definition}
    Let $\mathcal{X}=(X,D)\in \mSm_S$ and let $Z\subset X$ be a smooth divisor which is transversal to $|D|$.
    We define the \emph{motive of $\mathcal{X}$ supported on $Z$} to be the cofiber
    $$
        \motive_Z(\mathcal{X})=\cofib\bigl(\motive(X,D+Z)\to \motive(\mathcal{X})\bigr).
    $$
\end{definition}

\begin{lemma}\label{GS_deform_lemma}
    Let $\mathcal{X}=(X,D)\in \mSm_S$ and let $Z\subset X$ be a smooth divisor transversal to $|D|$.
    Define $B,E,Z_B,Z_E$ as in Theorem \ref{GS}.
    Then
    $$
        \motive_{Z_E}(E,D_E)\to \motive_{Z_B}(B,D_B)
    $$
    is an equivalence.
\end{lemma}

\begin{proof}
    Combining the coCartesian squares (\ref{MBU_square}) and (\ref{MBU_square'}), we get a coCartesian square
    $$
    \xymatrix{
        \motive(E,D_E+Z_E)\ar[r]\ar[d]  &\motive(B,D_B+Z_B)\ar[d]\\
        \motive(E,D_E)\ar[r]            &\motive(B,D_B),
    }
    $$
    which implies the claim.
\end{proof}

\begin{definition}
    Let $\mathcal{X}\in \mSm_S$ and let $Z\subset X$ be a smooth divisor which is transversal to $|D|$.
    Define $B,E,Z_B,Z_E$ as in Theorem \ref{GS}.
    Let $s_1$ denote the composition
    $$
           s_1\colon X\xrightarrow{\sim} X\times \{1\}\hookrightarrow B.
    $$
    We define the \emph{Gysin map} $\beta_{(\mathcal{X},Z)}$ to be the composition
    $$
        \motive_Z(\mathcal{X})\xrightarrow{s_1}\motive_{Z_B}(B,D_B)
        \xleftarrow{\sim}\motive_{Z_E}(E,D_E)
        \cong \MTh(\mathcal{N}_ZX),
    $$
    where the last equivalence is induced by the identification $E\cong \mathbb{P}(\normal_ZX\oplus \mathcal{O})$.
\end{definition}

\begin{definition}
    Let $\mathcal{X}=(X,D) \in \mSm_S$ and let $Z\subset X$ be a smooth divisor which is transversal to $|D|$.
     We say that $(\mathrm{GS})_{(\mathcal{X},Z)}$ holds if the Gysin map
     $$
     \beta_{(\mathcal{X},Z)}\colon \motive_Z(\mathcal{X}) \to \MTh(\mathcal{N}_ZX)
     $$
     is an equivalence.
     By definition, this is equivalent to saying that $s_1\colon \motive_Z(\mathcal{X})\to \motive_{Z_B}(B,D_B)$ is an equivalence.
\end{definition}

\begin{lemma}\label{GS_tensor}
    Let $\mathcal{X}=(X,D)\in \mSm_S$ and let $Z\subset X$ be a smooth divisor which is transversal to $|D|$.
    If $(\mathrm{GS})_{(\mathcal{X},Z)}$ holds, then $(\mathrm{GS})_{(\mathcal{X}\otimes\mathcal{Y},Z\times Y)}$ holds for any $\mathcal{Y}=(Y,E)\in \mSm$.
\end{lemma}

\begin{proof}
    This is a consequence of the fact that the functor $\motive(Y,E)\otimes({-})$ preserves colimits.
\end{proof}

\begin{lemma}\label{GS_descent}
    Let $\mathcal{X}=(X,D)\in \mSm_S$ and let $Z\subset X$ be a smooth divisor which is transversal to $|D|$.
    Let $\{\mathcal{U}_i\}_{i\in I}$ be a Zariski covering of $\mathcal{X}$.
    For each finite non-empty subset $J\subset I$, we set $\mathcal{U}_J=\bigcap_{j\in J}\mathcal{U}_j$.
    If $(\mathrm{GS})_{(\mathcal{U}_J,Z_J)}$ holds for every finite non-empty subset $J\subset I$, then $(\mathrm{GS})_{(\mathcal{X},Z)}$ holds.
\end{lemma}

\begin{proof}
    This follows immediately from the Nisnevich descent.
\end{proof}

\begin{lemma}\label{GS_lemma}
    The following square in $\mSH_{S^1}(S,\Lambda)$ is coCartesian:
    $$
    \xymatrix{
        \motive(\mathbb{P}^1,[0]+[\infty])\ar[r]\ar[d]      &\motive(\mathbb{P}^1,[0])\ar[d]\\
        \motive(\mathbb{P}^1,[\infty])\ar[r]      &\motive(\mathbb{P}^1,\varnothing).
    }
    $$
\end{lemma}

\begin{proof}
    By the Nisnevich descent, it suffices to prove that the restrictions of this square to $\mathbb{A}^1,\mathbb{P}^1-\{0\},\mathbb{A}^1-\{0\}\subset \mathbb{P}^1$ are coCartesian, which is obvious.
\end{proof}

\begin{lemma}\label{GS_cube}
    $(\mathrm{GS})_{(\bcube,0)}$ holds.
\end{lemma}

\begin{proof}
    Define $B,E,Z_B,Z_E$ as in Theorem \ref{GS}.
    Let $s_0$ denote the closed immersion
    $\mathbb{P}^1\cong E\hookrightarrow B$.
    This induces a morphism
    $$
        s_0\colon \motive_{\{0\}}(\bcube)\to
        \motive_{Z_B}(B,D_B).
    $$
    This is an equivalence since we have $\motive_{\{0\}}(\bcube)\cong\motive_{\{0\}}(\mathbb{P}^1)\cong \motive_{Z_E}(E,D_E)\cong \motive_{Z_B}(B,D_B)$, where the first equivalence follows from Lemma \ref{GS_lemma} and the last equivalence follows form Lemma \ref{GS_deform_lemma}.
    Therefore, to show that the Gysin map is an equivalence, it suffices to show that $s_1\colon \motive_{\{0\}}(\bcube)\to\motive_{Z_B}(B,D_B)$ is homotopic to $s_0$.
    Moreover, by Theorem \ref{tame_motive_HA}, it suffices to prove this for $\motive^{\log}({-})$ instead of $\motive({-})$.
    
    To this end, we consider the map
    $$
        h\colon \mathbb{A}^1\times \mathbb{A}^1\to \mathbb{A}^1\times \mathbb{A}^1;\quad (s,t)\mapsto (st,t).
    $$
    This lifts to a map $\widetilde{h}\colon \Bl_{(0,\infty),(\infty,0)}(\mathbb{P}^1\times\mathbb{P}^1)\to B=\Bl_{(0,0)}(\mathbb{P}^1\times\mathbb{P}^1)$ and gives a morphism of log pairs $\Bl_{(0,\infty),(\infty,0)}(\bcube\otimes\bcube)\to (B,D_B)$.
    Moreover, this induces
    $$
        \widetilde{h}\colon \motive^{\log}_{\pi^{-1}(\{0\}\times \mathbb{P}^1)}(\Bl_{(0,\infty),(\infty,0)}(\bcube\otimes\bcube))
        \to \motive^{\log}_{Z_B}(B,D_B),
    $$
    where $\pi\colon \Bl_{(0,\infty),(\infty,0)}(\mathbb{P}^1\times\mathbb{P}^1)\to \mathbb{P}^1\times\mathbb{P}^1$ is the projection.
    The canonical inclusions
    $$
        i_\nu\colon \mathbb{P}^1\xrightarrow{\sim}\mathbb{P}^1\times\{\nu\}\hookrightarrow
        \mathbb{P}^1\times\mathbb{P}^1\quad (\nu=0,1)
    $$
    can be lifted to
    $
        \iota_\nu\colon \mathbb{P}^1\to \Bl_{(0,\infty),(\infty,0)}(\mathbb{P}^1\times\mathbb{P}^1)
    $
    and induce morphisms of log pairs $\bcube\to \Bl_{(0,\infty),(\infty,0)}(\bcube\otimes\bcube)$.
    Moreover, they induce
    $$
        \iota_\nu\colon
        \motive^{\log}_{\{0\}}(\bcube)
        \to
        \motive^{\log}_{\pi^{-1}(\{0\}\times \mathbb{P}^1)}(\Bl_{(0,\infty),(\infty,0)}(\bcube\otimes\bcube))\quad (\nu=0,1).
    $$
    Consider the following commutative diagram in $\tDA^\eff$:
    $$
    \xymatrix{
        &
        \motive^{\log}_{\{0\}}(\bcube)
        \ar[dl]_-{i_0}\ar[d]^-{\iota_0}\ar[dr]^-{s_0}
        \\
        \motive^{\log}_{\{0\}\times\mathbb{P}^1}(\bcube\otimes\bcube)
        &
        \motive^{\log}_{\pi^{-1}(\{0\}\times \mathbb{P}^1)}(\Bl_{(0,\infty),(\infty,0)}(\bcube\otimes\bcube))
        \ar[r]^-{\widetilde{h}}
        \ar[l]_-{\sim}
        &
        \motive^{\log}_{Z_B}(B,D_B)
        \\
        &
        \motive^{\log}_{\{0\}}(\bcube).
        \ar[ul]^-{i_1}\ar[u]_-{\iota_1}\ar[ur]_-{s_1}
    }
    $$
    We have $i_0\simeq i_1$ since they are sections of the equivalence
    $$
        \pr_1\colon \motive^{\log}_{\{0\}\times\mathbb{P}^1}(\bcube\otimes\bcube)
        \xrightarrow{\sim}
        \motive^{\log}_{\{0\}}(\bcube).
    $$
    Therefore we get $\iota_0\simeq \iota_1$ and hence $s_0\simeq s_1$.
\end{proof}

\begin{lemma}\label{GS_excision}
    Let $\mathcal{X}=(X,D)\in \mSm_S$ and let $Z\subset X$ be a smooth divisor which is transversal to $|D|$.
    Let $X'\to X$ be an \'etale morphism which induces an isomorphism $X'\times_XZ\cong Z$ and set $\mathcal{X}'=(X',D|_{X'})$.
    Then $(\mathrm{GS})_{(\mathcal{X},Z)}$ holds if and only if $(\mathrm{GS})_{(\mathcal{X}',Z)}$ holds.
\end{lemma}

\begin{proof}
    For the pair $(\mathcal{X},Z)$ (resp. $(\mathcal{X}',Z)$), we define $B,Z_B$ (resp. $B',Z_{B'}$) as in Proposition \ref{MBU}.
    It suffices to show that the following diagram is coCartesian:
    $$
    \xymatrix{
         \motive_Z(X',D|_{X'})\ar[r]^-{s_1}\ar[d]^-\alpha
         &
         \motive_{Z_{B'}}(B',D_{B'})\ar[d]^-\beta\\
         \motive_Z(X,D)\ar[r]^-{s_1}
         &
         \motive_{Z_B}(B,D_B).
    }
    $$
    We prove that the morphisms $\alpha$ and $\beta$ are equivalences.
    As for $\alpha$, we consider the following square, where $U=X-Z$ and $U'=X'-Z$:
    $$
        \xymatrix{
        \motive(U',D|_{U'})
        \ar[r]
        \ar[d]
        &
        \motive(X',D|_{X'}+Z)
        \ar[r]
        \ar[d]
        &
        \motive(X',D|_{X'})
        \ar[d]
        \\
        \motive(U,D|_U)
        \ar[r]
        &
        \motive(X,D+Z)
        \ar[r]
        &
        \motive(X,D).
    }
    $$
    The left square and the total rectangle are coCartesian by the Nisnevich descent.
    By the pasting law, the right square is also coCartesian, which proves the claim.
    A similar argument shows that $\beta$ is also an equivalence.
\end{proof}

\begin{lemma}\label{GS_A^1}
    $(\mathrm{GS})_{((\mathbb{A}^1,\varnothing),0)}$ holds.
\end{lemma}

\begin{proof}
    This follows from Lemma \ref{GS_cube} and Lemma \ref{GS_excision}.
\end{proof}

\begin{proof}[Proof of Theorem \ref{GS}]
    As in the proof of Theorem \ref{SBU}, we may reduce to the case $\mathcal{X}=(Z,D|_Z)\otimes (\mathbb{A}^1,\varnothing)$ using Lemma \ref{transversal_structure} and Lemma \ref{GS_descent}.
    In this case the claim follows from Lemma \ref{GS_excision}, Lemma \ref{GS_tensor} and Lemma \ref{GS_A^1}.
\end{proof}

\section{Projective bundle formula and the Thom isomorphism}\label{PBF-Thom}

Assume that $S$ is normal.
Let $\Lambda$ be a connective commutative ring spectrum.
In this section we prove the projective bundle formula and the Thom isomorphism for oriented ring spectra in $\mSH(S,\Lambda)$ (cf. Definition \ref{def;mSH}).
The content dealt with in this section is essentially a reworking of the corresponding material in logarithmic motivic homotopy theory \cite[\S 7]{BPO}, so the proofs are kept concise.

\begin{definition}
    We write $\motive(\mathbb{P}^\infty,\varnothing)=\colim_{n\to \infty}\motive(\mathbb{P}^n,\varnothing)$.
\end{definition}

\begin{definition}
    Let $E$ be a homotopy commutative ring spectrum in $\mSH(S,\Lambda)$ and $F$ be an $E$-module.
    For $\mathcal{X}=(X,D)\in \mSm_S$, we write
    \begin{align*}
        E^{p,q}(\mathcal{X}):=\Hom_{\mSH(S,\Lambda)}(\motive(\mathcal{X}), \Sigma^{p,q}E),\\
        F^{p,q}(\mathcal{X}):=\Hom_{\mSH(S,\Lambda)}(\motive(\mathcal{X}), \Sigma^{p,q}F),
    \end{align*}
    where $\Sigma^{p,q}E=(S^1_t)^{\otimes q}\otimes \Sigma^{p-q}E$.
    We can define an action
     \begin{equation}\label{eq:actionEF}
        \cup\colon E^{p,q}(X,\varnothing)\otimes F^{p',q'}(\mathcal{X})\to F^{p+p',q+q'}(\mathcal{X})
    \end{equation}
    by the following composition:
    $$
        \alpha\cup \beta\colon \mathcal{X}\xrightarrow{\Delta} (X,\varnothing) \otimes \mathcal{X}
        \xrightarrow{\alpha\otimes \beta}\Sigma^{p,q}E\otimes \Sigma^{p',q'}F\xrightarrow{\cdot} \Sigma^{p+p',q+q'}F.
    $$
\end{definition}

We fix a homotopy commutative ring spectrum $E$ in $\mSH(S,\Lambda)$.

\begin{definition}
    An \emph{orientation} of $E$ is a cohomology class $c_\infty\in E^{2,1}(\motive(\mathbb{P}^\infty,\varnothing)/\motive(\pt))$ whose restriction to $\motive(\mathbb{P}^1,\varnothing)/\motive(\pt)$ is the class
    $$
        \Sigma^{2,1}(1)\colon \motive(\mathbb{P}^1,\varnothing)/\motive(\pt)\simeq \Sigma^{2,1}\motive(\pt)\to \Sigma^{2,1}E,
    $$
    where the equivalence comes from Lemma \ref{lem;S1S^1_t} and the last map is induced by the unit of $E$.
    We say that $E$ is \emph{oriented} if an orientation of $E$ is specified.
\end{definition}
\begin{definition}
    Suppose that $E$ is oriented.
    For $X\in \Sm_S$, we consider the ``classifying map''
    $$
        \Pic(X)\to \Hom_{\logSH_{S^1}(S,\Lambda)}(\motive^{\log}(X), \motive^{\log}(\mathbb{P}^\infty)/\motive^{\log}(\pt))
    $$
    constructed in \cite[(7.1.2)]{BPO}.
    Composing with the functor $t^*$ and the functor $\mSH_{S^1}(S,\Lambda)\to \mSH(S,\Lambda)$, we get a map
    $$
        \Pic(X)\to \Hom_{\mSH(S,\Lambda)}(\motive(X,\varnothing), \motive(\mathbb{P}^\infty,\varnothing)/\motive(\pt))
    $$
    Moreover, composing with the orientation class $c_\infty\colon \motive(\mathbb{P}^\infty,\varnothing)/\motive(\pt)\to \Sigma^{2,1}E$, we get the \emph{first Chern class}
    $$
        c_1\colon \Pic(X)\to E^{2,1}(X,\varnothing).
    $$
    By definition, we have 
      \begin{equation}\label{eq:c_1-formula}
      c_1(\mathcal{O}_{\mathbb{P}^1}(1))=\Sigma^{2,1}(1)\in E^{2,1}(\mathbb{P}^1).
      \end{equation}
    For any morphism $f\colon Y\to X$ in $\Sm_S$ and any line bundle $L$ on $X$, we have $f^*c_1(L)=c_1(f^*L)$.
\end{definition}

\begin{definition}
    Suppose that $E$ is oriented, and let $F$ be an $E$-module.
    Let $\mathcal{X}=(X,D)\in \mSm_S$ and let $\mathcal{V}\to \mathcal{X}$ be a vector bundle of rank $d+1$ on $\mathcal{X}$.
    Let $p\colon \mathbb{P}(\mathcal{V})\to \mathcal{X}$ be the projectivization.
    We define a map
    \begin{align}\label{eq:PBF_map}
    \rho_{\mathcal{V}}\colon \bigoplus_{i=0}^d F^{*-2i,*-i}(\mathcal{X})\to F^{*,*}(\mathbb{P}(\mathcal{V}))
    \end{align}
    by the following formula:
    $$
        \rho_{\mathcal{V}}(x_0,\dots,x_d) = \sum_{i=0}^d p^*(x_i)\cup c_1(L)^i.
    $$
    Here, $L$ denotes the dual of the tautological line bundle on $\mathbb{P}(V)$.
\end{definition}

\begin{theorem}[Projective bundle formula] \label{thm:pbf}
    Suppose that $E$ is oriented, and let $F$ be an $E$-module.
    Let $\mathcal{X}=(X,D)\in \mSm_S$ and let $\mathcal{V}\to \mathcal{X}$ be a vector bundle of rank $d+1$ on $\mathcal{X}$.
    Then, the map \eqref{eq:PBF_map} is an isomorphism.
\end{theorem}

\begin{remark}
    We thank Ryomei Iwasa for pointing out that Theorem \ref{thm:pbf} is a special case of the formalism of $\PP^1$-stable theories developed in \cite{Annala-Iwasa}.
    Indeed, in \cite[Lemma 3.3.5, 3.2.3]{Annala-Iwasa}, they discuss $\mathbb{P}^1$-stable theories for any reasonable modules $V$ over $\Sh_\Zar(\Sm_S)$; in particular, one can apply the results to $V=\mH_*$ to obtain the projective bundle formula in our case.
    Nevertheless, we keep the original argument below for the reader's convenience. 
\end{remark}

\begin{lemma}\label{lem:PBF_commutativity1}
    Suppose that $E$ is oriented, and let $F$ be an $E$-module.
    Define a map $\psi_d\colon F^{*-2,*-1}(\pt)\to F^{*,*}((\mathbb{P}^d,\varnothing)/\pt)$ by
    $$
        \psi_d(x)=p^*(x)\cup c_1(\mathcal{O}_{\mathbb{P}^d}(1)),
    $$
    where $p\colon \mathbb{P}^d\to \pt$ and $\cup$ is from \eqref{eq:actionEF}.
    Let $i\colon \mathbb{P}^1\to \mathbb{P}^d$ be the inclusion.
    Then, the following diagram is commutative:
    $$
    \xymatrix{
        F^{*-2,*-1}(\pt)\ar@{=}[r]\ar[d]^-{\simeq}_-{\Sigma^{2,1}}&
        F^{*-2,*-1}(\pt)\ar[d]_-{\psi_d}\\
        F^{*,*}((\mathbb{P}^1,\varnothing)/\pt)&
        F^{*,*}((\mathbb{P}^d,\varnothing)/\pt).\ar[l]_-{i^*}
    }
    $$
\end{lemma}

\begin{proof}
    Let $q\colon \mathbb{P}^1\to \pt$.
    We have (cf. \eqref{eq:c_1-formula})
    \begin{align*}
        i^*\psi_d(x)&=i^*(p^*(x)\cup c_1(\mathcal{O}_{\mathbb{P}^d}(1)))\\
        &=q^*(x)\cup c_1(\mathcal{O}_{\mathbb{P}^1}(1))\\
        &=q^*(x)\cup \Sigma^{2,1}(1)\\
        &=\Sigma^{2,1}(x).
    \end{align*}
    This shows that $i^*\circ \psi_d$ coincides with the canonical isomorphism induced by $\Sigma^{2,1}$.
\end{proof}

\begin{lemma}\label{lem:Pn_cell_motive}
    In $\mSH_{S^1}(S,\Lambda)$, there is an equivalence
    \begin{align}\label{eq:Pn_cell_motive}
        \motive(\mathbb{P}^d,\varnothing)/\motive(\mathbb{P}^{d-1},\varnothing)\xrightarrow{\simeq}
        (\motive(\mathbb{P}^1,\varnothing)/\motive(\pt))^{\otimes d}.
    \end{align}
\end{lemma}

\begin{proof}
    This statement is proved in \cite[Lemma 7.2.1]{BPO} for $\logSH_{S^1}(S,\Lambda)$.
    Applying the functor $t^*$, we get the desired equivalence.
\end{proof}

\begin{lemma}\label{lem:PBF_commutativity2}
    In $\mSH_{S^1}(S,\Lambda)$, there is a commutative diagram
    $$
    \xymatrix{
    (\motive(\mathbb{P}^1,\varnothing)/\motive(\pt))^{\otimes d}\ar[r]^-{\widetilde{i}}&
    (\motive(\mathbb{P}^d,\varnothing)/\motive(\pt))^{\otimes d}\\
    \motive(\mathbb{P}^d,\varnothing)/\motive(\mathbb{P}^{d-1},\varnothing)\ar[u]_-{\simeq}^-{\eqref{eq:Pn_cell_motive}}&
    \motive(\mathbb{P}^d,\varnothing)/\motive(\pt)\ar[l]\ar[u]^-{\Delta},
    }
    $$
    where $\widetilde{i}=i^{\otimes d}$ with the inclusion $i:\mathbb{P}^1\to \mathbb{P}^d$ and $\Delta$ is the diagonal.
\end{lemma}

\begin{proof}
    This statement is proved in \cite[Lemma 7.2.2]{BPO} for $\logSH_{S^1}(S,\Lambda)$.
    Applying the functor $t^*$, we get the desired commutative diagram.
\end{proof}

\begin{proof}[Proof of Theorem \ref{thm:pbf}]
    By the Nisnevich descent, we may assume that $\mathcal{V}$ is trivial.
    Replacing $F$ by $F^{\mathcal{X}}$, we may assume that $\mathcal{X}=\pt$.
    Therefore, have only to prove that
    $$
        \rho_d\colon \bigoplus_{i=0}^dF^{*-2i,*-i}(\pt)\to F^{*,*}(\mathbb{P}^d,\varnothing);\quad (x_0,\dots,x_d)\mapsto \sum_{i=0}^d p^*(x_i)\cup c_1(\mathcal{O}_{\mathbb{P}^d}(1))^i
    $$
    is an isomorphism, where $p\colon \mathbb{P}^d\to \pt$.
    It suffices to show that the induced map
    $$
        \widetilde{\rho}_d\colon \bigoplus_{i=1}^dF^{*-2i,*-i}(\pt)\to F^{*,*}((\mathbb{P}^d,\varnothing)/\pt)
    $$
    is an isomorphism.
    We proceed by induction on $d$.
    The claim is trivial if $d=0$.
    Let $d>0$ and Suppose that $\widetilde{\rho}_{d-1}$ is an isomorphism.
    Then, $i^*\colon F^{*,*}(\mathbb{P}^d,\varnothing)\to F^{*,*}(\mathbb{P}^{d-1},\varnothing)$ is surjective, where $i\colon \mathbb{P}^{d-1}\to \mathbb{P}^d$ is the inclusion.
    From the long exact sequence
    \begin{align*}
        \cdots \to F^{*-1,*}((\mathbb{P}^d,\varnothing)/\pt)\xrightarrow{i^*} F^{*-1,*}((\mathbb{P}^{d-1},\varnothing)/\pt)&\\
        F^{*,*}((\mathbb{P}^d,\varnothing)/(\mathbb{P}^{d-1},\varnothing))\to
        F^{*,*}((\mathbb{P}^d,\varnothing)/\pt)\xrightarrow{i^*}
        F^{*,*}((\mathbb{P}^{d-1},\varnothing)/\pt)&\to \cdots,
    \end{align*}
    we see that the sequence
    $$
        0\to 
        F^{*,*}((\mathbb{P}^d,\varnothing)/(\mathbb{P}^{d-1},\varnothing))\to
        F^{*,*}((\mathbb{P}^d,\varnothing)/\pt)\xrightarrow{i^*}
        F^{*,*}((\mathbb{P}^{d-1},\varnothing)/\pt)
        \to 0
    $$
    is exact.
    Consider the following diagram with exact rows:
    $$
    \xymatrix{
    0\ar[r]&
        F^{*-2d,*-d}(\pt)\ar[r]\ar[d]^-{\simeq}&
        \bigoplus_{i=1}^d F^{*-2i,*-i}(\pt)\ar[r]\ar[d]_-{\widetilde{\rho}_d}&
        \bigoplus_{i=1}^{d-1} F^{*-2i,*-i}(\pt)\ar[r]\ar[d]^-{\simeq}_-{\widetilde{\rho}_{d-1}}&0\\
    0\ar[r]& 
        F^{*,*}((\mathbb{P}^d,\varnothing)/(\mathbb{P}^{d-1},\varnothing))\ar[r]&
        F^{*,*}((\mathbb{P}^d,\varnothing)/\pt)\ar[r]^-{i^*}&
        F^{*,*}((\mathbb{P}^{d-1},\varnothing)/\pt)\ar[r]&0.
    }
    $$
    Here, the left vertical isomorphism is the composition
    $$
    \xymatrix{
        F^{*-2d,*-d}(\pt)\ar[r]^-{\Sigma^{2d,d}}_-{\simeq}&
        F^{*,*}(((\mathbb{P}^1,\varnothing)/\pt)^{\otimes d})\ar[r]^-{\eqref{eq:Pn_cell_motive}}_-{\simeq}&
        F^{*,*}((\mathbb{P}^d,\varnothing)/(\mathbb{P}^{d-1},\varnothing)),
    }
    $$
    and the right vertical morphism is an isomorphism by the induction hypothesis.
    Therefore, it suffices to show that the above diagram is commutative.
    This is equivalent to the commutativity of the following diagram:
    $$
    \xymatrix{
        F^{*-2d,*-d}(\pt)\ar@{=}[r]\ar[d]^-{\simeq}_-{\Sigma^{2d,d}}&
        F^{*-2d,*-d}(\pt)\ar[r]\ar[d]_-{\psi_d^{\otimes d}}&
        0\ar[dd]\\
        F^{*,*}(((\mathbb{P}^1,\varnothing)/\pt)^{\otimes d})\ar[d]^-{\simeq}_-{\eqref{eq:Pn_cell_motive}}&
        F^{*,*}(((\mathbb{P}^d,\varnothing)/\pt)^{\otimes d})\ar[l]\ar[d]_-{\Delta^*}&
        \\
        F^{*,*}((\mathbb{P}^d,\varnothing)/(\mathbb{P}^{d-1},\varnothing))\ar[r]&
        F^{*,*}((\mathbb{P}^d,\varnothing)/\pt)\ar[r]^-{i^*}&
        F^{*,*}((\mathbb{P}^{d-1},\varnothing)/\pt).
    }
    $$
    The left squares are commutative by Lemma \ref{lem:PBF_commutativity1} and Lemma \ref{lem:PBF_commutativity2} respectively.
    Since the composition of the bottom row is zero, it follows that the right square is also commutative.
    This finishes the proof of the projective bundle formula.
\end{proof}

\begin{corollary}[Thom isomorphism]\label{Thom}
    Suppose that $E$ is oriented, and let $F$ be an $E$-module.
    Let $\mathcal{X}=(X,D)\in \mSm_S$ and let $\mathcal{V}\to \mathcal{X}$ be a vector bundle of rank $d$ on $\mathcal{X}$.
    Then, there is a canonical isomorphism
    $$
        F^{*,*}(\mathcal{X})\xrightarrow{\simeq} F^{*+2d,*+d}(\MTh(\mathcal{V})).
    $$
\end{corollary}

\begin{proof}
    By the projective bundle formula, we have the following commutative diagram:
    $$
        \xymatrix{
        \bigoplus_{i=0}^{d}F^{*+2i,*+i}(\mathcal{X})\ar[r]\ar[d]^-{\simeq}&\bigoplus_{i=1}^dF^{*+2i,*+i}(\mathcal{X})\ar[d]^-{\simeq}\\
        F^{*+2d,*+d}(\mathbb{P}(\mathcal{V}\oplus \mathcal{O}))\ar[r]&F^{*+2d,*+d}(\mathbb{P}(\mathcal{V})).
        }
    $$
    This induces an isomorphism between kernels of the horizontal maps.
    By Proposition \ref{mthiso}, this gives the desired isomorphism.
\end{proof}

\part{Representability of cohomology theories}

\section{Hodge and Hodge--Witt cohomology}\label{sec:Hodge}

From here until the end of this paper, we fix a perfect base field $k$ and set $S=\Spec k$.
We write $\mSm$ for $\mSm_S$.
In this section, we prove that the Hodge cohomology with modulus \cite{KellyMiyazaki_Hodge1} \cite{KellyMiyazaki_Hodge2} and the Hodge--Witt cohomology with modulus \cite{Shiho} are representable in the categories $\mDA^\eff(k)=\mSH_{S^1}(\Spec k,\mathbb{Z})$ and $\mDA(k)=\mSH(\Spec k,\mathbb{Z})$.

We first recall from \cite{Koizumi-blowup} (inspired by \cite{RS21}) a useful construction of a sheaf of abelian groups on $\mSm$ from a collection of local data. 

\begin{definition}
A \emph{geometric henselian DVF} over $k$ is a discrete valuation field $(L,v_L)$ that is isomorphic to $\Frac \mathcal{O}_{X,x}^h$ for some $X \in \Sm$ and a point $x \in X$ of codimension $1$. Let $\Phi$ denote the collection of all geometric henselian DVFs over $k$. For each $L \in \Phi$, write $\mathcal{O}_L$ for the valuation ring of $L$.
\end{definition}

\begin{definition}\label{def;FFil}
Let $F$ be a Nisnevich sheaf of abelian groups on $\Sm$.
Suppose that we are given a collection of increasing filtrations $\mathrm{Fil}=\{\mathrm{Fil}_r F(L)\}_{r \in \mathbb{Q}_{\geq 0}}$ on $F(L)$ indexed by $L \in \Phi$.
Let $\mathcal{X}=(X,D)\in \mSm$.
We say a section $a \in F(\mathcal{X}^\circ)$ is \emph{bounded by $D$} if for any $L \in \Phi$ and any commutative diagram of the form
$$
    \xymatrix{
        \Spec L \ar[r]^{\rho} \ar@{^(->}[d] & \mathcal{X}^\circ \ar@{^(->}[d] \\
        \Spec \mathcal{O}_L \ar[r]^(0.58){\widetilde{\rho}} & X,
    }
$$
where the vertical arrows are the natural inclusions, we have 
$\rho^*a \in \mathrm{Fil}_{v_L(\widetilde{\rho}^*D)} F(L)$.
We set
$$
    F_{\mathrm{Fil}}(\mathcal{X}) := \{a \in F(\mathcal{X}^\circ)\ |\ \text{$a$ is bounded by $D$}\}.
$$
One can easily see that this defines a Nisnevich sheaf of abelian groups $F_{\mathrm{Fil}}$ on $\mSm$ (cf. \cite[Lemma 2.5]{Koizumi-blowup}).
\end{definition}

\begin{definition}
    For a Nisnevich sheaf of abelian groups $F$ on $\mSm$ and $\mathcal{X}=(X,D)\in \mSm$, we set
    $$
        \mathrm{R}\Gamma(\mathcal{X},F):=\mathrm{R}\Gamma(X,F_{(X,D)}).
    $$
\end{definition}

\subsection{Hodge cohomology with modulus}\label{sec;MOmega}
Kelly and the second author constructed an extension $\mathrm{R}\Gamma({-},\ulM\Omega^q)$ of the Hodge cohomology to modulus pairs \cite{KellyMiyazaki_Hodge1} \cite{KellyMiyazaki_Hodge2}.
The first author generalized their construction to $\mathbb{Q}$-modulus pairs \cite{Koizumi-blowup}.
In this section we show that $\mathrm{R}\Gamma({-},\ulM\Omega^q)$ is representable in $\mDA^\eff(k)$.

Fix a non-negative integer $q \geq 0$.
For $L\in \Phi$, we define a filtration $\{\Fil_r\Omega^q(L)\}_{r\in \mathbb{Q}_{\geq 0}}$ on $\Omega^q(L)$ by
\begin{equation} \label{eq:MOmega_filtration}
    \mathrm{Fil}_r \Omega^q (L)
    =\begin{cases}
        \Omega^q(\mathcal{O}_L) & (r=0), \\
        t^{-\lceil r \rceil + 1} \cdot \Omega^q (\mathcal{O}_L)(\log) & (r>0).
    \end{cases}
\end{equation}
The sheaf associated to this filtration is denoted by $\ulM\Omega^q \in \Sh_{\Nis}(\mSm,\Ab)$.
By \cite[Lemma 4.5]{Koizumi-blowup}, this sheaf has the following simple description:
\begin{equation} \label{eq:MOmega}
    \ulM\Omega^q (\sX) = \Gamma (X,\Omega^q_X (\log |D|)(\lceil D \rceil -|D|))
    \;\text{ for }\mathcal{X}=(X,D)\in \mSm.
\end{equation}
Note that we have $\ulM\Omega^q (X,\varnothing) = \Omega^q (X)$ and $\ulM\Omega^q (X,|D|)=\Gamma (X,\Omega^q_X(\log |D|))$.
In other words, the sheaf $\ulM\Omega^q$ generalizes the usual sheaf of (logarithmic) differential forms.

\begin{theorem}[Kelly-Miyazaki, Koizumi]\label{MOmega_local}
    For any $q \geq 0$, the sheaf of spectra $\mathrm{R}\Gamma({-},\ulM\Omega^q)$ on $\mSm$ is $(\CI\cup\BI)$-local.
\end{theorem}

\begin{proof}
    This is an direct consequence of \cite[Corollary 5.2]{KellyMiyazaki_Hodge2} and \cite[Corollary 4.7]{Koizumi-blowup}.
\end{proof}

Theorem \ref{MOmega_local} implies the following result stating that the Hodge cohomology with modulus is representable in the $S^1$-stable motivic homotopy category with $\mathbb{Q}$-modulus over a perfect field of arbitrary characteristic.

\begin{theorem}\label{thm;MOmega_rep}
    For any $q\geq 0$, there is an object $\mathrm{m}\Omega^q\in \mDA^\eff(k)$ such that there is a natural equivalence
    $$
        \map_{\mDA^\eff(k)} (\motive (\sX),\mathrm{m}\Omega^q) \simeq \mathrm{R}\Gamma(\mathcal{X},\ulM\Omega^q).
    $$
\end{theorem}

\begin{remark}
    In \cite{KellyMiyazaki_Hodge2} and \cite{Koizumi-blowup}, a similar representability result in the category of motives with modulus is proved, under the assumption that $k$ admits resolution of singularities.
    Here, we have avoided the use of resolution of singularities by modifying the construction of the motivic category.
\end{remark}

Next, we construct an oriented ring spectrum $\mathrm{m}\Omega\in \mDA(k)$ which represents the Hodge cohomology with modulus.

\begin{lemma}\label{lem:MOmega_key_computation}
    Let $\mathcal{X}=(X,D)\in \mSm$.
    Then, there is a canonical isomorphism of $\mathcal{O}_{X\times\mathbb{P}^1}$-modules
    $$
        \ulM\Omega^q_{\mathcal{X}\otimes (\mathbb{P}^1,[0]+[\infty])}\simeq
        \pr_1^*\ulM\Omega^q_{\mathcal{X}}\oplus \pr_1^*\ulM\Omega^{q-1}_{\mathcal{X}}.
    $$
\end{lemma}

\begin{proof}
    Let $\widehat{D}:=\ceil{D}-|D|$ and $E:=[0]+[\infty]$.
    Using \eqref{eq:MOmega}, we compute
    \begin{align*}
        \ulM\Omega^q_{\mathcal{X}\otimes (\mathbb{P}^1,[0]+[\infty])}
        &\simeq
        \Omega^q_{X\times\mathbb{P}^1}(\log |\pr_1^*D+\pr_2^*E|)(\pr_1^*\widehat{D})\\
        &\simeq
        \Omega^q_X(\log|D|)(\widehat{D})\boxtimes\Omega^0_{\mathbb{P}^1}\\
        &\qquad\oplus\Omega^{q-1}_X(\log|D|)(\widehat{D})\boxtimes\Omega^1_{\mathbb{P}^1}(\log E) \\
        &\simeq
        \pr_1^*\Omega^q_X(\log|D|)(\widehat{D})\oplus
        \pr_1^*\Omega^{q-1}_X(\log|D|)(\widehat{D})\\
        &\simeq
        \pr_1^*\ulM\Omega^q_{\mathcal{X}}\oplus\pr_1^*\ulM\Omega^{q-1}_{\mathcal{X}}.
    \end{align*}
    This finishes the proof.
\end{proof}

\begin{lemma}\label{lem:mOmega_deloop}
    In $\mDA^\eff(k)$, there is an equivalence
    $$
        \Omega_{S^1_t}(\mathrm{m}\Omega^q)\simeq \mathrm{m}\Omega^{q-1},
    $$
    where $\Omega_{S^1_t}$ is the right adjoint of $\Sigma_{S^1_t}=S^1_t\otimes({-})$.
\end{lemma}

\begin{proof}
    Let $\mathcal{X}=(X,D)\in \mSm$.
    In $\mDA^\eff(k)$, we have
    \begin{align*}
        &\map(\motive(\mathcal{X})\otimes \motive(\mathbb{P}^1,[0]+[\infty]),\mathrm{m}\Omega^q)\\
        &\simeq
        \mathrm{R}\Gamma(\mathcal{X}\otimes(\mathbb{P}^1,[0]+[\infty]),\ulM\Omega^q)\\
        &\simeq
        \mathrm{R}\Gamma(X\times \mathbb{P}^1,\pr_1^*\ulM\Omega^q_{\mathcal{X}})\oplus
        \mathrm{R}\Gamma(X\times \mathbb{P}^1,\pr_1^*\ulM\Omega^{q-1}_{\mathcal{X}})\\
        &\simeq
        \mathrm{R}\Gamma(\mathcal{X},\ulM\Omega^q)\oplus
        \mathrm{R}\Gamma(\mathcal{X},\ulM\Omega^{q-1}).
    \end{align*}
    Here, we used Lemma \ref{lem:MOmega_key_computation} for the second isomorphism.
    Recall that $S^1_t=\motive(\mathbb{P}^1,[0]+[\infty])/\motive(\{1\})$.
    Therefore, the above computation shows that
    $$
        \map(\motive(\mathcal{X}),\Omega_{S^1_t}(\mathrm{m}\Omega^q))
        \simeq
        \map(\motive(\mathcal{X})\otimes S^1_t,\mathrm{m}\Omega^q)
        \simeq
        \map(\motive(\mathcal{X}),\mathrm{m}\Omega^{q-1}).
    $$
    This finishes the proof.
\end{proof}

By Lemma \ref{lem:mOmega_deloop}, we can define the $S^1_t$-spectrum
$$
    \mathrm{m}\Omega := (\mathrm{m}\Omega^0,\mathrm{m}\Omega^1,\mathrm{m}\Omega^2,\dots)\in \mDA(k).
$$
For any $\mathcal{X}\in \mSm$, we have $\Sigma^{p,q}\mathrm{m}\Omega\simeq (\Sigma^{p-q}\mathrm{m}\Omega^q,\Sigma^{p-q}\mathrm{m}\Omega^{q+1},\dots)$ as an $S^1_t$-spectrum and hence
$$
    (\mathrm{m}\Omega)^{p,q}(\mathcal{X})
    :=\Hom_{\mDA(k)}(\motive(\mathcal{X}),\Sigma^{p,q}\mathrm{m}\Omega)\simeq
    \mathrm{H}^{p-q}(\mathcal{X},\ulM\Omega^q).
$$
Usual multiplication of differential forms
$$
    \ulM\Omega^q(\mathcal{X})\otimes
    \ulM\Omega^{q'}(\mathcal{Y})\to
    \ulM\Omega^{q+q'}(\mathcal{X}\otimes \mathcal{Y})
$$
defines a homotopy commutative ring structure on $\mathrm{m}\Omega$.
Moreover, the usual first Chern class $c_1(\mathcal{O}_{\mathbb{P}^d}(1))\in \mathrm{H}^1(\mathbb{P}^d,\Omega^1)\simeq (\mathrm{m}\Omega)^{2,1}(\mathbb{P}^d,\varnothing)$ determines an orientation of $\mathrm{m}\Omega$.
The projective bundle formula (Theorem \ref{thm:pbf}) and the Thom isomorphism (Theorem \ref{Thom}) imply the corresponding properties of the Hodge cohomology with modulus.
\begin{corollary}\label{cor:pbf_MOmega}
    Let $\mathcal{X}=(X,D)\in \mSm_k$ and let $\mathcal{V}\to \mathcal{X}$ be a vector bundle of rank $d+1$ on $\mathcal{X}$.
    Then, the map
    $$
        \bigoplus_{i=0}^d\mathrm{H}^{p-i}(\mathcal{X},\ulM\Omega^{q-i})\xrightarrow{\simeq}\mathrm{H}^{p}(\mathbb{P}(\mathcal{V}),\ulM\Omega^q);\quad (x_0,\dots,x_d)\mapsto \sum_{i=0}^d p^*(x)\cup c_1(L)^i
    $$
    is an isomorphism, where $L$ is the dual of the tautological bundle on $\mathbb{P}(\mathcal{V})$.
\end{corollary}

\begin{proof}
    This is a consequence of the projective bundle formula in $\mDA(k)$ (Theorem \ref{thm:pbf}).
\end{proof}

\begin{corollary}\label{cor:Thom_MOmega}
    Let $\mathcal{X}=(X,D)\in \mSm_k$ and let $\mathcal{V}\to \mathcal{X}$ be a vector bundle of rank $d$ on $\mathcal{X}$.
    Then, there is a canonical isomorphism
    $$
        \mathrm{H}^p(\mathcal{X},\ulM\Omega^{q})\xrightarrow{\simeq} \mathrm{H}^{p+d}(\MTh(\mathcal{V}),\ulM\Omega^{q+d}).
    $$
\end{corollary}

\begin{proof}
    This is a consequence of the Thom isomorphism in $\mDA(k)$ (Theorem \ref{Thom}).
\end{proof}

The properties of motives with modulus as in Theorem \ref{main-2} imply the corresponding properties of the Hodge cohomology with modulus.

\begin{corollary}
    Let $\mathcal{X}=(X,D)\in \mSm_k$ and let $Z\subset X$ be a smooth closed subscheme of codimension $d$ which is transversal to $|D|$.
    Then we have
    $$
        \mathrm{R}\Gamma((\Bl_ZX,\pi^*D), \ulM\Omega^q)\simeq
        \mathrm{R}\Gamma(\mathcal{X}, \ulM\Omega^q)\oplus
        \bigoplus_{i=1}^{d-1}\mathrm{R}\Gamma((Z,D|_Z),\ulM\Omega^{q-i}[-i]).
    $$
\end{corollary}

\begin{proof}
    This is a consequence of the smooth blow-up excision in $\mDA^\eff(k)$ (Theorem \ref{SBU}) and the projective bundle formula for $\ulM\Omega^q$ (Corollary \ref{cor:pbf_MOmega}).
\end{proof}

\begin{corollary}
    Let $\mathcal{X}=(X,D)\in \mSm_k$ and let $Z\subset X$ be a smooth closed subscheme which is transversal to $|D|$.
    Then there exists a canonical fiber sequence
    $$
        \mathrm{R}\Gamma((Z,D|_Z),\ulM\Omega^{q+d}[d])
        \to
        \mathrm{R}\Gamma(\mathcal{X},\ulM\Omega^q)
        \to
        \mathrm{R}\Gamma((\Bl_ZX,q^*D+E),\ulM\Omega^q),
    $$
    where $q\colon \Bl_ZX\to X$ is the blow-up along $Z$ and $E$ is the exceptional divisor.
\end{corollary}

\begin{proof}
    This is a consequence of the Gysin sequence in $\mDA^\eff(k)$ (Theorem \ref{GS}) and the Thom isomorphism for $\ulM\Omega^q$ (Corollary \ref{cor:Thom_MOmega}).
\end{proof}

\subsection{Hodge--Witt cohomology with modulus}\label{sec;MWOmega}
Assume that $\ch(k)=p>0$.
Shiho constructed in \cite{Shiho} an extension $\mathrm{R}\Gamma({-},\ulM\Witt_n\Omega^q)$ of the Hodge--Witt cohomology to $\mathbb{Q}$-modulus pairs.
In this section we show that $\mathrm{R}\Gamma({-},\ulM\Witt_n\Omega^q)$ is representable in $\mDA^\eff(k)$.

Note that F. Ren and K. R\"ulling study a different version of Hodge--Witt sheaf with modulus in \cite{Ren-Rulling}, by which a beautiful duality theory for Hodge--Witt cohomology is constructed. It is an interesting task to compare the definitions in {\it loc. cit.} and \cite{Shiho}. 

We first recall the filtration on $\Witt_n\Omega^q$ from \cite[Definition 2.6]{Shiho}.
\begin{definition}
    Let $(X,D)\in \SmlSm$ and write $D=D_1+\dots+D_m$.
    We define $\mathcal{I}_{(X,D),i}$ to be the ideal of the structure sheaf $\mathcal{O}_{(X,D)/\Witt_n}$ of the log crystalline site $((X,D)/\Witt_n(k))_{\mathrm{crys}}$ generated by the local equations of $D_i$ (see \cite[p. \!6]{Shiho}).
    For $\mathbf{b}=(b_1,\dots,b_m)\in \mathbb{Z}_{\geq 0}^m$, we define
    $$
        \mathcal{I}_{(X,D)}^{\otimes \mathbf{b}}:=\bigotimes_{i=1}^m\mathcal{I}_{(X,D),i}^{\otimes b_i},
    $$
    and $\mathcal{I}_{(X,D)}^{\otimes (-\mathbf{b})}:=\mathcal{H}\mathrm{om}(\mathcal{I}_{(X,D)}^{\otimes \mathbf{b}},\mathcal{O}_{(X,D)/\Witt_n(k)})$.
\end{definition}

\begin{lemma}\label{lem:MWOmega_lift}
    Let $(X,D)\in \SmlSm$ and write $D=D_1+\dots+D_m$.
    Suppose that $(X,D)$ can be lifted to $(\widetilde{X},\widetilde{D})\in \SmlSm_{\Witt_n(k)}$.
    Write $\widetilde{D}=\widetilde{D}_1+\dots+\widetilde{D}_m$ so that $\widetilde{D}_i$ is a lift of $D_i$.
    Let $u_X\colon ((X,|D|)/\Witt_n(k))_{\mathrm{crys}}\to X_\Zar$ denote the canonical morphism of sites.
    Then we have
    $$
        \mathrm{R}^qu_{X,*}\mathcal{I}_{(X,D)}^{\otimes (-\mathbf{b})}\simeq \mathcal{H}^q\bigl(\Omega^\bullet_{\widetilde{X}/\Witt_n(k)}(\log \widetilde{D})(\textstyle\sum_{i=1}^mb_i\widetilde{D}_i)\bigr).
    $$
\end{lemma}

\begin{proof}
    See \cite[\S 2]{Shiho}.
\end{proof}

For $L \in \Phi$, define a filtration $\{\mathrm{Fil}_r\Witt_n\Omega^q_L\}_{r\in \mathbb{Q}_{\geq 0}}$ on $\Witt_n\Omega^q_L$ by 
\begin{equation}\label{eq:MWOmega_filtration}
    \mathrm{Fil}_r \Witt_n\Omega^q_L
    :=\begin{cases}
        \Witt_n\Omega^q_{\mathcal{O}_L} & (r=0) \\
        \colim_{(X,D) \in \mathcal{P}_L} \mathrm{H}^q_{\mathrm{crys}} (\mathcal{I}_{(X,D)}^{\otimes (- p\lceil r \rceil + 1)}) & (r>0),
    \end{cases}
\end{equation}
where $\mathcal{P}_L$ is the partially ordered set of all pairs of an affine $X\in \Sm$ and a smooth principal divisor $D$ on $X$ approximating $\mathcal{O}_L$ (see the beginning of \S 2 of \cite{Shiho} for details).
The sheaf associated to this filtration is denoted by $\ulM \Witt_n\Omega^q\in \Sh_\Nis(\mSm,\Ab)$.
This sheaf has the following simple description:

\begin{lemma}\label{lem:MWOmega_global}
    Let $\mathcal{X}=(X,D)\in \mSm$, and write $D=\sum_{i=1}^mr_iD_i$ with $D_i$ smooth.
    Set $\ceil{\mathbf{r}}:=(\ceil{r_1},\dots,\ceil{r_m})$ and $\mathbf{1}=(1,\dots,1)$.
    Then we have
    $$
    \ulM\Witt_n\Omega^q (\sX) = \Gamma (X,\mathrm{R}^qu_{X,*}\mathcal{I}_{(X,|D|)}^{\otimes -p\ceil{\mathbf{r}}+\mathbf{1}}) =  \Gamma (X,\mathrm{R}^qu_{X,*}\mathcal{I}_{(X,|D|)}^{\otimes -p(\ceil{\mathbf{r}}-\mathbf{1})}).
    $$
\end{lemma}
\begin{proof}
    The first equality is proved in \cite[Proposition 2.12]{Shiho}.
    The second equality follows from \cite[Proposition 2.4 (3)]{Shiho}.
\end{proof}

\begin{remark}
    The sheaf $\ulM \Witt_n\Omega^q$ coincides with $\ulM \Omega^q$ from the above example when $n=1$, and with $\ulM \Witt_n$ from \cite{Koizumi-blowup} when $q=0$ (see \cite[Theorem 1.4]{Shiho}).
\end{remark}

\begin{theorem}[Shiho]
    For any $n \geq 1$ and $q \geq 0$, the sheaf of spectra $\mathrm{R}\Gamma({-},\ulM \Witt_n\Omega^q)$ on $\mSm$ is $(\CI\cup\BI)$-local.
\end{theorem}

\begin{proof}
    This is a direct consequence of  \cite[Theorem 1.4]{Shiho}.
\end{proof}

This immediately implies the following representability result. 

\begin{theorem}\label{thm;mWOmega}
    For any $n \geq 1$ and $q \geq 0$, there is an object $\mathrm{m}\Witt_n\Omega^q\in \mDA^\eff(k)$ such that there is a natural equivalence
    $$
        \map_{\mDA^\eff(k)} (\motive (\sX),\mathrm{m}\Witt_n\Omega^q) \simeq \mathrm{R}\Gamma(\mathcal{X},\ulM\Witt_n\Omega^q).
    $$
\end{theorem}

Next, we construct an oriented ring spectrum $\mathrm{m}\Witt_n\Omega\in \mDA(k)$ which represents the Hodge--Witt cohomology with modulus.

\begin{lemma}\label{lem:MWOmega_key_computation}
    Let $\mathcal{X}=(X,D)\in \mSm$.
    Then there is a canonical isomorphism of $\Witt_n\mathcal{O}_{X\times\mathbb{P}^1}$-modules
    $$
        \ulM\Witt_n\Omega^q_{\mathcal{X}\otimes (\mathbb{P}^1,[0]+[\infty])}\simeq
        \pr_1^*\ulM\Witt_n\Omega^q_{\mathcal{X}}\oplus \pr_1^*\ulM\Witt_n\Omega^{q-1}_{\mathcal{X}}.
    $$
\end{lemma}

\begin{proof}
    Since the problem is local, we may assume that $(X,D)$ can be lifted to $(\widetilde{X},\widetilde{D})\in \mSm_{\Witt_n(k)}$.
    Let $\widehat{D}:=\ceil{\widetilde{D}}-|\widetilde{D}|$ and $E:=[0]+[\infty]$.
    Using Lemmas \ref{lem:MWOmega_lift} and \ref{lem:MWOmega_global}, we compute
    \begin{align*}
        \ulM\Witt_n\Omega^q_{\mathcal{X}\otimes (\mathbb{P}^1,[0]+[\infty])}
        &\simeq
        \mathcal{H}^q\bigl(\Omega^\bullet_{\widetilde{X}\times\mathbb{P}^1/\Witt_n(k)}(\log |\pr_1^*\widetilde{D}+\pr_2^*E|)(p\cdot\pr_1^*\widehat{D})\bigr)\\
        &\simeq
        \mathcal{H}^q\bigl(\Omega^\bullet_{\widetilde{X}}(\log|\widetilde{D}|)(p\widehat{D})\boxtimes_{\Witt_n\mathcal{O}}\Omega^0_{\mathbb{P}^1/\Witt_n(k)}\\
        &\qquad\oplus\Omega^{\bullet-1}_{\widetilde{X}}(\log|\widetilde{D}|)(p\widehat{D})\boxtimes_{\Witt_n\mathcal{O}}\Omega^1_{\mathbb{P}^1/\Witt_n(k)}(\log E)\bigr) \\
        &\simeq
        \pr_1^*\mathcal{H}^q(\Omega^\bullet_{\widetilde{X}}(\log|\widetilde{D}|)(p\widehat{D}))\oplus
        \pr_1^*\mathcal{H}^{q-1}(\Omega^\bullet_{\widetilde{X}}(\log|\widetilde{D}|)(p\widehat{D}))\\
        &\simeq
        \pr_1^*\ulM\Witt_n\Omega^q_{\mathcal{X}}\oplus\pr_1^*\ulM\Witt_n\Omega^{q-1}_{\mathcal{X}}.
    \end{align*}
    This finishes the proof.
\end{proof}

\begin{lemma}\label{lem:mWOmega_deloop}
    In $\mDA^\eff(k)$, there is an equivalence
    $$
        \Omega_{S^1_t}(\mathrm{m}\Witt_n\Omega^q)\simeq \mathrm{m}\Witt_n\Omega^{q-1},
    $$
    where $\Omega_{S^1_t}$ is the right adjoint of $\Sigma_{S^1_t}=S^1_t\otimes({-})$.
\end{lemma}

\begin{proof}
    Let $\mathcal{X}=(X,D)\in \mSm$.
    In $\mDA^\eff(k)$, we have
    \begin{align*}
        &\map(\motive(\mathcal{X})\otimes \motive(\mathbb{P}^1,[0]+[\infty]),\mathrm{m}\Witt_n\Omega^q)\\
        &\simeq
        \mathrm{R}\Gamma(\mathcal{X}\otimes(\mathbb{P}^1,[0]+[\infty]),\ulM\Witt_n\Omega^q)\\
        &\simeq
        \mathrm{R}\Gamma(X\times \mathbb{P}^1,\pr_1^*\ulM\Witt_n\Omega^q_{\mathcal{X}})\oplus
        \mathrm{R}\Gamma(X\times \mathbb{P}^1,\pr_1^*\ulM\Witt_n\Omega^{q-1}_{\mathcal{X}})\\
        &\simeq
        \mathrm{R}\Gamma(\mathcal{X},\ulM\Witt_n\Omega^q)\oplus
        \mathrm{R}\Gamma(\mathcal{X},\ulM\Witt_n\Omega^{q-1}).
    \end{align*}
    Here, we used Lemma \ref{lem:MWOmega_key_computation} for the second isomorphism.
    Recall that $S^1_t=\motive(\mathbb{P}^1,[0]+[\infty])/\motive(\{1\})$.
    Therefore, the above computation shows that
    $$
        \map(\motive(\mathcal{X}),\Omega_{S^1_t}(\mathrm{m}\Witt_n\Omega^q))
        \simeq
        \map(\motive(\mathcal{X})\otimes S^1_t,\mathrm{m}\Witt_n\Omega^q)
        \simeq
        \map(\motive(\mathcal{X}),\mathrm{m}\Witt_n\Omega^{q-1}).
    $$
    This finishes the proof.
\end{proof}

By Lemma \ref{lem:mWOmega_deloop}, we can define the $S^1_t$-spectrum
$$
    \mathrm{m}\Witt_n\Omega := (\mathrm{m}\Witt_n\Omega^0,\mathrm{m}\Witt_n\Omega^1,\mathrm{m}\Witt_n\Omega^2,\dots)\in \mDA(k).
$$
For any $\mathcal{X}\in \mSm$, we have $\Sigma^{p,q}\mathrm{m}\Witt_n\Omega\simeq (\Sigma^{p-q}\mathrm{m}\Witt_n\Omega^q,\Sigma^{p-q}\mathrm{m}\Witt_n\Omega^{q+1},\dots)$ as an $S^1_t$-spectrum and hence
$$
    (\mathrm{m}\Witt_n\Omega)^{p,q}(\mathcal{X})
    :=\Hom_{\mDA(k)}(\motive(\mathcal{X}),\Sigma^{p,q}\mathrm{m}\Witt_n\Omega)\simeq
    \mathrm{H}^{p-q}(\mathcal{X},\ulM\Witt_n\Omega^q).
$$
Usual multiplication on the de Rham-Witt complex
$$
    \ulM\Witt_n\Omega^q(\mathcal{X})\otimes
    \ulM\Witt_n\Omega^{q'}(\mathcal{Y})\to
    \ulM\Witt_n\Omega^{q+q'}(\mathcal{X}\otimes \mathcal{Y})
$$
defines a homotopy commutative ring structure on $\mathrm{m}\Witt_n\Omega$.
Moreover, the usual first Chern class $c_1(\mathcal{O}_{\mathbb{P}^d}(1))\in \mathrm{H}^1(\mathbb{P}^d,\Witt_n\Omega^1)\simeq (\mathrm{m}\Witt_n\Omega)^{2,1}(\mathbb{P}^d,\varnothing)$ determines an orientation of $\mathrm{m}\Witt_n\Omega$.
The projective bundle formula (Theorem \ref{thm:pbf}) and the Thom isomorphism (Theorem \ref{Thom}) imply the corresponding properties of the Hodge--Witt cohomology with modulus.

\begin{corollary}\label{cor:pbf_MWOmega}
    Let $\mathcal{X}=(X,D)\in \mSm_k$ and let $\mathcal{V}\to \mathcal{X}$ be a vector bundle of rank $d+1$ on $\mathcal{X}$.
    Then, the map
    $$
        \bigoplus_{i=0}^d\mathrm{H}^{p-i}(\mathcal{X},\ulM\Witt_n\Omega^{q-i})\xrightarrow{\simeq}\mathrm{H}^{p}(\mathbb{P}(\mathcal{V}),\ulM\Witt_n\Omega^q);\quad (x_0,\dots,x_d)\mapsto \sum_{i=0}^d p^*(x)\cup c_1(L)^i
    $$
    is an isomorphism, where $L$ is the dual of the tautological bundle on $\mathbb{P}(\mathcal{V})$.
\end{corollary}

\begin{proof}
    This is a consequence of the projective bundle formula in $\mDA(k)$ (Theorem \ref{thm:pbf}).
\end{proof}

\begin{corollary}\label{cor:Thom_MWOmega}
    Let $\mathcal{X}=(X,D)\in \mSm_k$ and let $\mathcal{V}\to \mathcal{X}$ be a vector bundle of rank $d$ on $\mathcal{X}$.
    Then, there is a canonical isomorphism
    $$
        \mathrm{H}^p(\mathcal{X},\ulM\Witt_n\Omega^{q})\xrightarrow{\simeq} \mathrm{H}^{p+d}(\MTh(\mathcal{V}),\ulM\Witt_n\Omega^{q+d}).
    $$
\end{corollary}

\begin{proof}
    This is a consequence of the Thom isomorphism in $\mDA(k)$ (Theorem \ref{Thom}).
\end{proof}

The properties of motives with modulus as in Theorem \ref{main-2} imply the corresponding properties of the Hodge--Witt cohomology with modulus.

\begin{corollary}
    Let $\mathcal{X}=(X,D)\in \mSm_k$ and let $Z\subset X$ be a smooth closed subscheme of codimension $d$ which is transversal to $|D|$.
    Then we have
    $$
        \mathrm{R}\Gamma((\Bl_ZX,\pi^*D), \ulM\Witt_n\Omega^q)\simeq
        \mathrm{R}\Gamma(\mathcal{X}, \ulM\Witt_n\Omega^q)\oplus
        \bigoplus_{i=1}^{d-1}\mathrm{R}\Gamma((Z,D|_Z),\ulM\Witt_n\Omega^{q-i}[-i]).
    $$
\end{corollary}

\begin{proof}
    This is a consequence of the smooth blow-up excision in $\mDA^\eff(k)$ (Theorem \ref{SBU}) and the projective bundle formula for $\ulM\Witt_n\Omega^q$ (Corollary \ref{cor:pbf_MWOmega}).
\end{proof}

\begin{corollary}
    Let $\mathcal{X}=(X,D)\in \mSm_k$ and let $Z\subset X$ be a smooth closed subscheme which is transversal to $|D|$.
    Then there exists a canonical fiber sequence
    $$
        \mathrm{R}\Gamma((Z,D|_Z),\ulM\Witt_n\Omega^{q+d}[d])
        \to
        \mathrm{R}\Gamma(\mathcal{X},\ulM\Witt_n\Omega^q)
        \to
        \mathrm{R}\Gamma((\Bl_ZX,q^*D+E),\ulM\Witt_n\Omega^q),
    $$
    where $q\colon \Bl_ZX\to X$ is the blow-up along $Z$ and $E$ is the exceptional divisor.
\end{corollary}

\begin{proof}
    This is a consequence of the Gysin sequence in $\mDA^\eff(k)$ (Theorem \ref{GS}) and the Thom isomorphism for $\ulM\Witt_n\Omega^q$ (Corollary \ref{cor:Thom_MWOmega}).
\end{proof}

\section{Modulus sheaves with transfers}

In order to prove the representability of cohomology for more general sheaves, we recall the theory of \emph{modulus sheaves with transfers} developed in \cite{KMSY1}, \cite{KMSY2}, \cite{KSY2}, \cite{Sai20}, \cite{BRS}.
In particular, we upgrade the \emph{strict cube-invariance theorem} due to the third author \cite[Theorem 9.3]{Sai20} to $\mathbb{Q}$-modulus pairs.

We continue to fix a perfect base field $k$.
First we recall the definition of general $\mathbb{Q}$-modulus pairs and finite correspondences between them (see \cite{Koizumi-Miyazaki}).

\begin{definition}\label{def;ulMCorQ}
    A \emph{$\mathbb{Q}$-modulus pair} over $k$ is a pair $\mathcal{X}=(X,D)$ where $X\in \Sch_k$ and $D$ is an effective $\mathbb{Q}$-Cartier divisor on $X$ such that $\mathcal{X}^\circ:=X-|D|$ is smooth over $k$.
    If $D$ is represented by a usual Cartier divisor, then we say that $(X,D)$ is a \emph{$\mathbb{Z}$-modulus pair}.
    
    Let $\mathcal{X}=(X,D)$, $\mathcal{Y}=(Y,E)$ be two $\mathbb{Q}$-modulus pairs.
    We define a subgroup $\ulMCor(\mathcal{X},\mathcal{Y})$ of $\Cor(\mathcal{X}^\circ,\mathcal{Y}^\circ)$ by
    $$
        \ulMCor(\mathcal{X},\mathcal{Y})=
        \biggl\{\sum_i n_i[V_i]\in \Cor(\mathcal{X}^\circ,\mathcal{Y}^\circ)\biggm|
        \begin{array}{l}
            \overline{V_i}\text{ is proper over }X,\\
            (\pr_1^*D)|_{\overline{V_i}^N}\geq (\pr_2^*E)|_{\overline{V_i}^N}
        \end{array}
        \biggr\},
    $$
    where $\overline{V_i}$ is the closure of $V_i$ in $X\times Y$ and $\overline{V_i}^N$ is its normalization.
    An element of $\ulMCor(\mathcal{X},\mathcal{Y})$ is called a \emph{finite correspondence} from $\mathcal{X}$ to $\mathcal{Y}$.
\end{definition}

There is a category $\ulMCor$ whose objects are $\mathbb{Q}$-modulus pairs and whose morphisms are given by $\ulMCor(\mathcal{X},\mathcal{Y})$ \cite[Lemma 2.2]{Koizumi-Miyazaki}.
The category $\ulMCor$ has a symmetric monoidal structure $\otimes$ which is defined by
$$
	(X,D)\otimes (Y,E)=(X\times Y, \pr_1^*D+\pr_2^*E).
$$

\begin{remark}\label{rem:graph}
    \begin{enumerate}
        \item Let $\mathcal{X}=(X,D)$ and $\mathcal{Y}=(Y,E)$ be two $\mathbb{Q}$-modulus pairs.
        If $f\colon X\to Y$ is a morphism of $k$-schemes satisfying $D\geq f^*E$, then the graph of $f|_{\mathcal{X}^\circ}$ defines an element of $\ulMCor(\mathcal{X},\mathcal{Y})$.
        A finite correspondence of this form is called an \emph{ambient morphism}.
        There is a faithful functor $\mSm\to \ulMCor$ which sends a morphism to the associated ambient morphism.
        \item Let $\mathcal{X}=(X,D)$ be a $\mathbb{Q}$-modulus pair.
        If $f\colon Y\to X$ is a proper morphism which is an isomorphism over $X-|D|$, then the ambient morphism $f\colon (Y,f^*D)\to (X,D)$ is an isomorphism in $\ulMCor$ because the transpose of the graph of $f|_{\mathcal{X}^\circ}$ gives the inverse of $f$.
    \end{enumerate}
\end{remark}

\begin{remark}
    In the previous literature, e.g., \cite{KMSY1,KMSY2, KMSY3}, $\Z$-modulus pairs are simply called modulus pairs, and the category of $\Z$-modulus pairs was denoted $\ulMCor$. 
For simplicity of notation, we use the same notation for the category of $\Q$-modulus pairs. 
By definition, the category of $\Z$-modulus pairs is a full subcategory of the category of $\Q$-modulus pairs, hence there is little risk of confusion caused by this abuse of notation. 
\end{remark}

\begin{definition}
    A presheaf of abelian groups $F\in \PSh(\ulMCor,\Ab)$ is called a \emph{Nisnevich sheaf} if for every $\mathbb{Q}$-modulus pair $\mathcal{X}=(X,D)$, the presheaf
    $$
        F_\mathcal{X}\colon (X_\Nis)^\op\to \Ab;\quad (U\to X)\mapsto F(U,D|_U)
    $$
    is a Nisnevich sheaf.
    In this situation, the restriction of $F$ to $\mSm$ is a Nisnevich sheaf of abelian groups on $\mSm$.
    We write $\mathrm{H}^i(\mathcal{X},F):=\mathrm{H}^i(X,F_\mathcal{X})$ for $\mathcal{X}\in \mSm$.
    We write $\Sh_\Nis(\ulMCor,\Ab)$ for the category of Nisnevich sheaves of abelian groups on $\ulMCor$.
\end{definition}

\begin{definition}
    Let $\omega\colon \ulMCor\to \Cor$ denote the functor $\mathcal{X}\mapsto \mathcal{X}^\circ$, which admits a left adjoint $X\mapsto (X,\varnothing)$.
    These functors induce a pair of adjoint functors
    $$
        \begin{tikzcd}
            \mathcal{\Sh_\Nis(\ulMCor,\Ab)}
            \arrow[rr,shift left=0.75ex,"\omega_!"]
            \arrow[rr,leftarrow, shift right=0.75ex,"\omega^*"']
            &&
            \mathcal{\Sh_\Nis(\Cor,\Ab)},
        \end{tikzcd}
    $$
    where $(\omega_!F)(X)=F(X,\varnothing)$ and $(\omega^*G)(\mathcal{X})=G(\mathcal{X}^\circ)$.
\end{definition}

\begin{definition}
    Let $H$ be a presheaf of abelian groups on $\mSm$.
    \begin{enumerate}
        \item We say that $H$ is \emph{cube-invariant} if for any $\mathcal{X}\in \mSm$, the map
        $H(\mathcal{X})\to H(\mathcal{X}\otimes \bcube)$
        is an isomorphism.
        \item We say that $H$ is \emph{blow-up invariant} if for any SNC blow-up $\mathcal{Y}\to \mathcal{X}$ in $\mSm$, the map
        $H(\mathcal{X})\to H(\mathcal{Y})$
        is an isomorphism.
    \end{enumerate}
\end{definition}

\subsection{LS-Cube-invariance}

\begin{definition}
Let $F$ be a presheaf of abelian groups on $\ulMSm^\mathbb{Q}$.
We say that $F$ is \emph{LS-cube-invariant} if $F|_{\mSm}$ is cube-invariant\footnote{The abbreviation LS stands for log-smooth.}.

\end{definition}

\begin{definition}\label{def;cubehomotopic}
Let $\mathcal{X}=(X,D)$, $\mathcal{Y}=(Y,E)$ be $\mathbb{Q}$-modulus pairs, and let $\alpha,\beta\colon \mathcal{X}\to \mathcal{Y}$ be two finite correspondences.
We say that $\alpha$ and $\beta$ are \emph{cube-homotopic} if there is a finite correspondence $\gamma\colon \mathcal{X}\otimes\bcube\to \mathcal{Y}$ such that
$$
\gamma\circ i_0=\alpha,\quad \gamma\circ i_1=\beta,
$$
where $i_0, i_1\colon \mathcal{X}\to \mathcal{X}\otimes\bcube$ are the ambient morphisms induced by
$$
    i_\nu\colon X\xrightarrow{\sim} X\times \{\nu\}\hookrightarrow X\times \mathbb{P}^1\quad (\nu=0,1).
$$
\end{definition}

\begin{lemma}\label{LS_CI_characterization}
    Let $F$ be a presheaf of abelian groups on $\ulMCor$.
    The following conditions are equivalent:
    \begin{enumerate}
        \item $F$ is LS-cube-invariant.
        \item If $\mathcal{X},\mathcal{Y}$ are $\mathbb{Q}$-modulus pairs with $\mathcal{X}$ log-smooth, and $\alpha,\beta\colon \mathcal{X}\to \mathcal{Y}$ are two finite correspondences which are cube-homotopic, then we have
        $$
        \alpha^*=\beta^*\colon F(\mathcal{Y})\to F(\mathcal{X}).
        $$
    \end{enumerate}
\end{lemma}

\begin{proof}
    See \cite[Lemma 3.5]{Koizumi-Miyazaki}.
\end{proof}

\begin{lemma}\label{cube_epsilon_invariance_sheaf}
    Let $F\in \PSh(\ulMCor,\Ab)$ be an LS-cube-invariant presheaf.
    For any $\mathcal{X}\in \mSm$ and $\varepsilon\in (0,1]\cap \mathbb{Q}$, the map $F(\mathcal{X})\to F(\mathcal{X}\otimes\bcube^\varepsilon)$ is an isomorphism.
\end{lemma}

\begin{proof}
    Consider the multiplication map $\mu\colon \mathbb{A}^1\times\mathbb{A}^1\to \mathbb{A}^1$.
    By \cite[Lemma 5.1.1]{KMSY3}, this gives a finite correspondence $\mu\colon \bcube^\varepsilon\otimes\bcube\to \bcube^\varepsilon$.
    Consider the following commutative diagram:
    $$
    \xymatrix{
        &F(\mathcal{X}\otimes\bcube^\varepsilon)\\
        F(\mathcal{X}\otimes\bcube^\varepsilon)\ar[r]^-{(\id\otimes\mu)^*}\ar[d]^-{i_0^*}\ar[ur]^-{\id}
        &F(\mathcal{X}\otimes\bcube^\varepsilon\otimes\bcube)\ar[d]^-{i_0^*}\ar[u]_-{i_1^*}\\
        F(\mathcal{X})\ar[r]&F(\mathcal{X}\otimes\bcube^\varepsilon).
    }
    $$
    We have $i_0^*=i_1^*$ since they are retractions of the isomorphism $F(\mathcal{X}\otimes\bcube^\varepsilon)\to F(\mathcal{X}\otimes\bcube^\varepsilon\otimes\bcube)$.
    Therefore the composition $F(\mathcal{X}\otimes\bcube^\varepsilon)\to F(\mathcal{X})\to F(\mathcal{X}\otimes\bcube^\varepsilon)
    $ coincides with the identity.
\end{proof}

\subsection{Good and excellent presheaves}

\begin{definition}\label{def;compactification}
    Let $\mathcal{X}=(X,D)$ be a $\mathbb{Q}$-modulus pair.
    A \emph{compactification} of $\mathcal{X}$ is a triple $(\overline{X},\overline{D},\Sigma)$ where $\overline{X}$ is a proper $k$-scheme and $\overline{D},\Sigma$ are effective $\mathbb{Q}$-Cartier divisors on $\overline{X}$, equipped with an identification $\overline{X}-|\Sigma|\cong X$ such that $\overline{D}|_X=D$.
\end{definition}

\begin{definition} \label{def:goodexc}
    Let $F\in \PSh(\ulMCor,\Ab)$.
    We say that $F$ is \emph{good} if the following conditions are satisfied:
    \begin{enumerate}
        \item (M-reciprocity) For any $\mathbb{Q}$-modulus pair $\mathcal{X}$, the canonical map
        $$
        \colim_{(\overline{X},\overline{D},\Sigma)}F(\overline{X},\overline{D}+\Sigma)\to F(\mathcal{X})
        $$
        is an isomorphism, where the colimit is taken over all compactifications of $\mathcal{X}$.
        \item (semi-purity) For any $\mathbb{Q}$-modulus pair $\mathcal{X}=(X,D)$, the canonical map $F(X,D)\to F(\mathcal{X}^\circ,\varnothing)$ is injective.
    \end{enumerate}
    We say that $F$ is \emph{excellent} if $F$ satisfies (1), (2), and the following condition:
    \begin{enumerate}\setcounter{enumi}{2}
        \item (left continuity) For any $\mathbb{Q}$-modulus pair $\mathcal{X}=(X,D)$, the canonical map $$\colim_{\varepsilon\to 0} F(X,(1-\varepsilon)D)\to F(X,D)$$ is an isomorphism.
    \end{enumerate}
\end{definition}

\begin{remark}
Let $F\in \PSh(\ulMCor,\Ab)$.
Then $F$ has semi-purity if and only if the unit morphism $F\to \omega^*\omega_!F$ is a monomorphism.
\end{remark}

\begin{remark}\label{rmk;M-reciprocity}
    Let $\mathcal{X}$ be a $\mathbb{Q}$-modulus pair and let $(\overline{X},\overline{D},\Sigma)$ be a compactification of $\mathcal{X}$.
    Then $\{(\overline{X},\overline{D},n\Sigma)\}_{n>0}$ is cofinal in the category of compactifications of $\mathcal{X}$ (see \cite[Lemma 2.4 (2)]{Koizumi-Miyazaki}).
    In particular, a presheaf of abelian groups $F$ on $\ulMCor$ has M-reciprocity if and only if for any $\mathbb{Q}$-modulus pair $\mathcal{X}$ and any compactification $(\overline{X},\overline{D},\Sigma)$ of $\mathcal{X}$, the canonical map
    $$
        \colim_{n\to \infty}F(\overline{X},\overline{D}+n\Sigma)\to F(\mathcal{X})
    $$
    is an isomorphism.
\end{remark}

\begin{remark}
    Let $G$ be a presheaf of abelian groups on the full subcategory $\MCor^\mathbb{Q}\subset \ulMCor$ spanned by proper modulus pairs, i.e., $(X,D)$ with $X$ proper.
    Then the left Kan extension of $G$ along $\MCor^\mathbb{Q}\hookrightarrow\ulMCor$ is given by
    $$
        (X,D)\mapsto
        \colim_{(\overline{X},\overline{D},\Sigma)}G(\overline{X},\overline{D}+\Sigma)
    $$
    (see \cite[Lemma 2.4]{Koizumi-Miyazaki}).
    Therefore a presheaf of abelian groups $F$ on $\ulMCor$ has M-reciprocity if and only if $F$ is left Kan-extended from $\MCor^\mathbb{Q}$.
\end{remark}

\begin{definition}
    Let $F\in \PSh(\ulMCor,\Ab)$ be a good presheaf.
    We define $F^\exc\in \PSh(\ulMCor,\Ab)$ by
    $$
        F^\exc (X,D)=\colim_{\varepsilon\to 0} F(X,(1-\varepsilon)D).
    $$
\end{definition}

\begin{lemma}\label{exc_preserves_everything}
    Let $F\in \PSh(\ulMCor,\Ab)$ be a good presheaf.
    \begin{enumerate}
        \item $F^\exc$ is excellent.
        \item If $F$ is a Nisnevich sheaf, then so is $F^\exc$.
        \item If $F$ is LS-cube-invariant, then so is $F^\exc$.
    \end{enumerate}
\end{lemma}

\begin{proof}
    The statements (1) and (2) are clear from the definition.
    As for (3), we have to show that
    $$
        \colim_{\varepsilon\to 0} F(X,(1-\varepsilon) D)\to \colim_{\varepsilon\to 0} F((X,(1-\varepsilon) D)\otimes(\mathbb{P}^1,(1-\varepsilon)[\infty]))
    $$
    is an isomorphism for $(X,D)\in \mSm$.
    This follows from Lemma \ref{cube_epsilon_invariance_sheaf}.
\end{proof}

\subsection{Dilation}

\begin{definition}
    For $F\in \PSh(\ulMCor,\Ab)$ and $N>0$, we define the \emph{$N$-th dilation} $F^{[N]}$ of $F$ by
    $$
    F^{[N]}(X,D)=F(X,(1/N)D)\;\text{ for } (X,D)\in \ulMCor.
    $$
\end{definition}

\begin{lemma}\label{dilation_preserves_everything}
    Let $F\in \PSh(\ulMCor,\Ab)$ and $N>0$.
    \begin{enumerate}
        \item If $F$ is a Nisnevich sheaf, then so is $F^{[N]}$.
        \item If $F$ is good (resp. excellent), then so is $F^{[N]}$.
        \item If $F$ is LS-cube-invariant, then so is $F^{[N]}$.
    \end{enumerate}
\end{lemma}

\begin{proof}
    The statements (1) and (2) are clear from the definition.
    As for (3), we have to show that $F(X,(1/N)D)\to F((X,(1/N)D)\otimes (\mathbb{P}^1,(1/N)[\infty]))$ is an isomorphism for $(X,D)\in \mSm$.
    This follows from Lemma \ref{cube_epsilon_invariance_sheaf}.
\end{proof}

\subsection{Nisnevich sheafification}

\begin{lemma}
    The inclusion functor $\Sh_\Nis(\ulMCor,\Ab)\hookrightarrow \PSh(\ulMCor,\Ab)$ admits an exact left adjoint $F\mapsto F_\Nis$, which is given by
    $$
    F_\Nis(X,D) = \colim_{f\colon Y\to X} \Gamma(Y,a_\Nis( F_{(Y,f^*D))})).
    $$
    Here, the colimit is taken over proper morphisms $f\colon Y\to X$ which is an isomorphism over $X-|D|$.
\end{lemma}

\begin{remark} \label {rem;sheafification}\ 
    \begin{enumerate}
        \item We have to take the colimit along blow-ups in the above construction since, otherwise, the resulting sections will not form a presheaf on the category $\ulMCor$. Indeed, by definition, the morphisms $Y \to X$ appearing in the above colimits induce isomorphisms in $\ulMCor$.
        \item \label{rem;Z-Q-agree} In the previous literature \cite{KMSY1}, the same result is proven for (pre)sheaves on the category of $\Z$-modulus pairs and modulus correspondences, which we denote by $\ulMCor^{\Z}$ to distinguish it from our category $\ulMCor$. Note that the natural functor $\ulMCor^{\Z} \to \ulMCor$ is fully faithful. 
        Given $F \in \PSh(\ulMCor,\Ab)$, we can consider its sheafification $(F|_{\ulMCor^{\Z}})_{\Nis}$ (which was denoted $(F|_{\ulMCor^{\Z}})_{\ulMNis}$ in \cite{KMSY1}) given by the same formula as above.
        By construction, for any $\Z$-modulus pair $(X,D)$, there exists a canonical identification $F_{\Nis}(X,D)=(F|_{\ulMCor^{\Z}})_{\Nis} (X,D)$. 
        In brief, this means that the two notions of sheafification agree on $\Z$-modulus pairs. 
    \end{enumerate}
    
\end{remark}

\begin{proof}
    This statement is proved in \cite[4.5.5]{KMSY1} for presheaves on the category of $\mathbb{Z}$-modulus pairs $\ulMCor^{\Z}$ (cf. Remark \ref{rem;sheafification} \eqref{rem;Z-Q-agree}), which can be applied to the restiction $F|_{\ulMCor^{\Z}}$ for any presheaf $F$ on $\ulMCor$.
    We extend this result to $\mathbb{Q}$-modulus pairs by using dilation.
    Let us define $F_\Nis(X,D)$ by the above formula.
    The non-trivial part is the construction of the correspondence action on $F_\Nis$.
    Given any finite correspondence $\alpha\colon (X,D)\to (Y,E)$ between $\mathbb{Z}$-modulus pairs, we have a well-defined pullback map
    $$
        \alpha^*\colon F_\Nis(Y,E)\to F_\Nis(X,D)
    $$
    by \cite[4.5.5]{KMSY1}.
    Given any finite correspondence $\alpha\colon (X,D)\to (Y,E)$ between $\mathbb{Q}$-modulus pairs, we can always take a positive integer $N$ such that $(X,ND)$ and $(Y,NE)$ are $\mathbb{Z}$-modulus pairs.
    Then, we can define the correspondence action $\alpha^*\colon F_\Nis(Y,E)\to F_\Nis(X,D)$ by
    $$
        F_\Nis(Y,E)=F^{[N]}_\Nis(Y,NE)\xrightarrow{\alpha^*} F^{[N]}_\Nis(X,ND)=F_\Nis(X,D),
    $$
    where $F^{[N]}_\Nis := (F^{[N]})_\Nis = (F_\Nis)^{[N]}$ (the second equality is easily checked by construction). 
    
    It is easy to see that this gives a well-defined presheaf structure on $F$.
    Let us prove that this gives an exact left adjoint to the inclusion functor.
    Let $F\in \PSh(\ulMCor,\Ab)$, $G\in \Sh_\Nis(\ulMCor,\Ab)$ and let $\varphi\colon F\to G$ be a morphism of presheaves.
    For any $\mathbb{Q}$-modulus pair $(X,D)$ and a proper morphism $f\colon Y\to X$ which is an isomorphism over $X-|D|$, we have an induced map
    $$
        \Gamma(Y,a_\Nis(F_{(Y,f^*D)}))\to \Gamma(Y,G_{(Y,f^*D)})\cong \Gamma(X,G_{(X,D)}),
    $$
    where the first map is induced by an adjunction since $G_{(Y,f^*D)}$ is a sheaf on $Y_{\Nis}$.
    This induces a map $F_\Nis(X,D)\to G(X,D)$.
    By construction, these maps are compatible with finite correspondences between $\mathbb{Z}$-modulus pairs.
    By replacing $F$ by $F^{[N]}$, it follows that these maps are compatible with finite correspondences between arbitrary $\mathbb{Q}$-modulus pairs.
\end{proof}

\begin{remark}
    In general, the sheaf $F_\Nis|_{\mSm}$ is not isomorphic to $a_\Nis(F|_{\mSm})$ since the former involves the colimit along the blow-ups while the latter does not. 
\end{remark}

\begin{lemma}\label{good_Nis}
    Let $F\in \PSh(\ulMCor,\Ab)$.
    If $F$ is good, then so is $F_\Nis$.
\end{lemma}

\begin{proof}
    This statement is proved in \cite[Lemma 1.27, Lemma 1.29 (4)]{Sai20} for $\mathbb{Z}$-modulus pairs.
    We extend this result to $\mathbb{Q}$-modulus pairs by using dilation.
    First we assume that $F\in \PSh(\ulMCor,\Ab)$ has semi-purity, and show that $F_\Nis$ also has semi-purity.
    By \cite[Lemma 1.29 (4)]{Sai20}, the map
    $$
        F_\Nis(\mathcal{X})\to F_\Nis(\mathcal{X}^\circ,\varnothing)
    $$
    is injective if $\mathcal{X}$ is a $\mathbb{Z}$-modulus pair.
    For an arbitrary $\mathbb{Q}$-modulus pair $\mathcal{X}=(X,D)$, we take a positive integer $N$ such that $(X,ND)$ is a $\mathbb{Z}$-modulus pair.
    By Lemma \ref{dilation_preserves_everything}, the dilation $F^{[N]}$ also has semi-purity.
    Therefore, it follows that the map
    $$
    F_\Nis(X,D)=F_\Nis^{[N]}(X,ND)\to F_\Nis^{[N]}(\mathcal{X}^\circ,\varnothing)=F_\Nis(\mathcal{X}^\circ,\varnothing)
    $$
    is injective.
    This means that $F_\Nis$ has semi-purity.
    By a similar argument using \cite[Lemma 1.27]{Sai20}, we can show that if $F$ has M-reciprocity, then so does $F_\Nis$.
\end{proof}

We refer to an argument as in the proof of Lemma \ref{good_Nis} as a \emph{dilation argument}.

\subsection{Contraction}

Recall that in Voevodsky's theory, the \emph{contraction} of an $\mathbb{A}^1$-invariant sheaf $F\in \Sh_\Nis(\Cor,\Ab)$ is defined by $F_{-1}(X)=F(X\times (\mathbb{A}^1-\{0\}))/F(X\times \mathbb{A}^1)$.
There is an analogous construction in our theory:

\begin{definition}
    Let $F\in \Sh_\Nis(\ulMCor,\Ab)$ be a good LS-cube-invariant sheaf.
    For $a\in \mathbb{Q}_{>0}$, we define $F^{(a)}_{-1}\in \PSh(\ulMCor,\Ab)$ by
    $$
    F^{(a)}_{-1} (\mathcal{X}) := \dfrac{F(\mathcal{X} \otimes (\mathbb{P}^1,a[0]+[\infty]))}{F(\mathcal{X}\otimes(\mathbb{P}^1,[\infty]))}.
    $$
    By definition, $F^{(a)}_{-1}$ is again a good LS-cube-invariant sheaf.
\end{definition}

\begin{theorem} \label{thm7.1}
Let $S$ be the henselian localization of an object of $\Sm$, and $D$ be an effective $\mathbb{Q}$-Cartier divisor on $S$.
For an $S$-scheme $T$, we write $D_T$ for the pullback of $D$ to $T$.
Let $Z:=\{t=0\}\subset \mathbb{A}^1_S$, $X:=(\AA^1_S)^h_{|Z}$ and write $i\colon Z\hookrightarrow X$ for the inclusion.
Let $F\in \Sh_\Nis(\ulMCor,\Ab)$ be a good LS-cube-invariant sheaf and $a\in \mathbb{Q}_{>0}$.
\begin{enumerate}
    \item   There exists an exact sequence of sheaves
    \[
    0 \to F_{(X,D_X)} \to F_{(X,aZ+D_X)} \to i_* (F^{(a)}_{-1})_{(Z,D_Z)} \to 0.
    \]
    \item   For any log-smooth $\mathbb{Q}$-modulus pair $\mathcal{Y}$, there exists an exact sequence of sheaves
            \[
            0 \to F_{(X,D_X)\otimes \mathcal{Y}} \to F_{(X,aZ+D_X)\otimes\mathcal{Y}} \to (i\times \id)_* (F^{(a)}_{-1})_{(Z,D_Z)\otimes \mathcal{Y}} \to 0.
            \]
\end{enumerate}
\end{theorem}

\begin{proof}
    This statement is proved in \cite[Theorem 7.1]{Sai20} for $\mathbb{Z}$-modulus pairs.
    We extend this result to $\mathbb{Q}$-modulus pairs by a dilation argument. 
    More precisely, for a positive integer $N$, we have the following identifications of sheaves on $X_{\Nis}$
    \begin{align*}
        F_{(X,D_X)} &= (F^{[N]})_{(X,ND_X)} \\
        F_{(X,aZ+D_X)} &= (F^{[N]})_{(X,aNZ+ND_X)} \\
        i_* (F^{(a)}_{-1})_{(Z,D_Z)} &= i_*((F^{[N]})^{(aN)}_{-1})_{(Z,ND_Z)}
    \end{align*}
    Now we can find $N$ such that all divisors appearing on the right hand sides of the above equalities are integral divisors. Then the corresponding sequence 
    \[
        0 \to (F^{[N]})_{(X,ND_X)} \to (F^{[N]})_{(X,aNZ+ND_X)} \to i_*((F^{[N]})^{(aN)}_{-1})_{(Z,ND_Z)} \to 0
    \]
    is known to be exact by {\it loc.cit.} and hence we have proven (1). The proof of (2) can be done in the same manner. 
\end{proof}

\subsection{Strict cube-invariance and strict blow-up invariance}

\begin{theorem}[Strict cube-invariance]\label{CCI}
    Let $F\in \Sh_\Nis(\ulMCor,\Ab)$ be a good LS-cube-invariant sheaf.
    Then the cohomology presheaf $\mathrm{H}^i({-},F)$ on $\mSm$ is cube-invariant.
\end{theorem}

\begin{proof}
    Let $\mathcal{X}=(X,D)\in \mSm$.
    It suffices to show that the canonical morphism $F_{\mathcal{X}}\to \mathrm{R}p_*F_{\mathcal{X}\otimes\bcube}$ is an equivalence, where $p\colon X\times\mathbb{P}^1\to X$ is the canonical projection.
    Considering the Leray spectral sequence, we can replace $X$ by its henselian localization at a point $x\in X$. 
    In this situation we can write $X=\Spec K\{x_1,\cdots,x_c\}$, where $K=k(x)$ and $|D|=\{x_1x_2\cdots x_r=0\}$.
    Write $D=a_1D_1+\dots+a_rD_r$ where $D_i=\{x_i=0\}$.
    Let $D':=a_1D_1+\dots+a_{r-1}D_{r-1}$, $a:=a_r$, and $E=D_r$ so that $D=D'+aE$.
    By Theorem \ref{thm7.1}, we have exact sequences
    \begin{align*}
        &0\to F_{(X,D')}\to F_{(X,D)}\to \iota_*(F^{(a)}_{-1})_{(E,D'|_E)}\to 0,\\
        &0\to F_{(X,D')\otimes\bcube}\to F_{(X,D)\otimes\bcube}\to (\iota\times \id)_*(F^{(a)}_{-1})_{(E,D'|_E)\otimes\bcube}\to 0
    \end{align*}
    where $\iota\colon E\to X$ is the inclusion.
    Therefore it suffices to prove the claim for $(X,D')$ and $(E,D'|_E)$.
    Repeating this argument, we can reduce to the case $r=0$, which is proved in \cite[Theorem 9.3]{Sai20}.
\end{proof}

\begin{lemma}\label{lem:higher-HI}
Let $F\in \Sh_\Nis(\ulMCor,\Ab)$ be a good LS-cube-invariant sheaf and $\sX=(X,D)\in \mSm$.
Let $E_1,\dots,E_n$ be effective $\mathbb{Q}$-Cartier divisors on $\mathbb{A}^1$ and write $\pi\colon \mathbb{A}^n_X\to X$ for the projection.
Then we have
\[
    \mathrm{R}^q\pi_* (F_{(\mathbb{A}^1,E_1)\otimes\ldots \otimes(\mathbb{A}^1, E_n)\otimes \sX})=0\quad (q>0).
\]
\end{lemma}

\begin{proof}
    We proceed as in the proof of Theorem \ref{thm7.1} by using the dilation argument. 
    For any positive integer $N$, we note 
    \[
        F_{(\AA^1,E_1) \otimes \cdots (\AA^1,E_n) \otimes (X,D)} = (F^{[N]})_{(\AA^1,NE_1) \otimes \cdots (\AA^1,NE_n) \otimes (X,ND)},
    \]
    and we can make $NE_1,\dots,NE_n, ND$ integral divisors by choosing an appropriate $N$. 
    Then the vanishing of the higher direct image of the right hand side is known by \cite[Lemma 2.10]{BRS}.
\end{proof}

\begin{theorem}[Strict blow-up invariance] \label{BI}
    Let $F\in \Sh_\Nis(\ulMCor,\Ab)$ be an excellent LS-cube-invariant sheaf.
    Then the cohomology presheaf $\mathrm{H}^i({-},F)$ on $\mSm$ is blow-up invariant.
\end{theorem}

\begin{proof}
    This follows from Theorem \ref{CCI}, Lemma \ref{lem:higher-HI}, and \cite[Theorem 1.4]{Koizumi-blowup}.
\end{proof}

\section{Relation with reciprocity sheaves} \label{sec:Fmod}

\subsection{Reciprocity presheaves} \label{subsec:reciprocity-presheaves}

First we recall the definition of reciprocity presheaves from \cite{KSY2}.

\begin{definition}\label{def;RSC}
    Let $F$ be a presheaf of abelian groups on $\Cor$.
    Let $X\in \Sm$ and $a\in F(X)$.
    A \emph{modulus} for $a$ is a $\mathbb{Q}$-modulus pair $\mathcal{Y}=(Y,E)$ with the following properties:
    \begin{enumerate}
        \item $X=Y-|E|$ and $Y$ is proper.
        \item For any $T\in \Sm$ and 
        $\alpha,\beta\in \Cor(T,X)=\ulMCor((T,\varnothing), \mathcal{Y})$ which are cube-homotopic (cf. Definition \ref{def;cubehomotopic}), we have $\alpha^*a=\beta^*a\in F(T)$.
    \end{enumerate}
    We say that $F$ is a \emph{reciprocity presheaf} if every section of $F$ admits a modulus.
    We write $\RSC$ for the category of reciprocity presheaves.
\end{definition}

\begin{remark}\label{rmk;ulMCorQ}
Recall that a presheaf $F$ of abelian groups on $\Cor$ is $\A^1$-invariant,
i.e. $F(X)\simeq F(X\times\A^1)$ for every $X\in \Sm_k$, if and only if for every $a\in F(X)$ with $X\in \Sm$ and $\alpha,\beta\in \Cor(T,X)$ with $T\in \Sm$ which are $\A^1$-homotopic, we have $\alpha^*a=\beta^*a\in F(T)$.
Here, we say that $\alpha,\beta$ are $\A^1$-homotopic if there is 
$\gamma\in\Cor(T\times\A^1,X)$ such that $\gamma\circ i_0=\alpha$ and $\gamma\circ i_1=\beta$, 
where $i_0, i_1\colon T\to T\otimes\A^1$ are induced by
$\{\nu\}\hookrightarrow \mathbb{A}^1$ for $\nu=0,1$.
Thus, Definition \ref{def;RSC} is a modulus refinement of the $\A^1$-invariance.
If $F$ is $\A^1$-invariant, then for $a\in F(X)$ as above, any $(Y,E)$ with $X=Y-|E|$ and $Y$ proper is a modulus of $a$ so that $F$ belongs to $\RSC$.
We refer to \cite[\S11.1]{BRS} for examples of reciprocity sheaves which are non-$\A^1$-invariant.
\end{remark}

\begin{remark}
    If $(Y,E)$ is a modulus for $a\in F(X)$, then $(Y,nE)\;(n>0)$ is also a modulus for $a$.
    This shows that our definition of reciprocity presheaves is equivalent to the original one \cite{KSY1} \cite{KSY2} which uses $\mathbb{Z}$-modulus pairs instead of $\mathbb{Q}$-modulus pairs.
\end{remark}

\begin{definition} \label{def:omegaCI}
    Let $F$ be a presheaf of abelian groups on $\Cor$.
    For a $\mathbb{Q}$-modulus pair $\mathcal{X}$, we define a subgroup $\omega^\CI F(\mathcal{X})\subset \omega^*F(\mathcal{X})=F(\mathcal{X}^\circ)$ by
    (see Definition \ref{def;compactification} for a compactification)
    $$
        \omega^\CI F(\mathcal{X})=\left\{a\in F(\mathcal{X}^\circ)\;\middle|\; \begin{array}{l}
            \text{There exists a compactification $(\overline{X},\overline{D},\Sigma)$ of $\mathcal{X}$}\\
            \text{such that $(\overline{X},\overline{D}+\Sigma)$ is a modulus for $a$}
        \end{array}\right\}.
    $$
\end{definition}

This defines a presheaf $\omega^\CI F$ on $\ulMCor$ \cite[Proposition 2.3.7]{KSY2}.
By definition, $\omega_!\omega^\CI F=F$ holds if and only if $F$ is a reciprocity presheaf.
The next lemma shows that $\omega^\CI F$ is the largest good LS-cube-invariant subpresheaf of $\omega^*F$:

\begin{lemma}\label{omega_CI_max}
    Let $F$ be a presheaf of abelian groups on $\Cor$.
    Then $\omega^\CI F$ is good and LS-cube-invariant.
    Moreover, if $G\subset \omega^*F$ is a good LS-cube-invariant subpresheaf, then we have $G\subset \omega^\CI F$.
\end{lemma}

\begin{proof}
    It is clear from the definition that $\omega^\CI F$ has semi-purity.
    We show that $\omega^\CI F$ has M-reciprocity.
    Suppose that $\mathcal{X}$ is a $\mathbb{Q}$-modulus pair and $a\in \omega^\CI F(\mathcal{X})$.
    By definition, there exists a compactification $(\overline{X},\overline{D},\Sigma)$ of $\mathcal{X}$ such that $(\overline{X},\overline{D}+\Sigma)$ is a modulus for $a$.
    Since $(\overline{X},\overline{D}+\Sigma,\varnothing)$ is a compactification of $(\overline{X},\overline{D}+\Sigma)$, it follows that $a\in \omega^\CI F(\overline{X},\overline{D}+\Sigma)$.
    Therefore $\omega^\CI F$ has M-reciprocity.

    Next we show that $\omega^\CI F$ is LS-cube-invariant.
    By Lemma \ref{LS_CI_characterization}, it suffices to show that for $\mathbb{Q}$-modulus pairs $\mathcal{X},\mathcal{Y}$ with $\mathcal{X}$ log-smooth and finite correspondences $\alpha,\beta\colon \mathcal{X}\to \mathcal{Y}$ which are homotopic, we have $\alpha^*=\beta^*\colon \omega^\CI F(\mathcal{Y})\to \omega^\CI F(\mathcal{X})$.
    Since $\omega^\CI F$ has semi-purity, we may assume that $\mathcal{X}=(T,\varnothing)$ where $T\in \Sm$.
    Let $a\in \omega^\CI F(\mathcal{Y})$.
    By definition, there exists a compactification $(\overline{Y},\overline{E},\Sigma)$ of $\mathcal{Y}$ such that $(\overline{Y},\overline{E}+\Sigma)$ is a modulus for $a$.
    Since $\alpha,\beta\colon (T,\varnothing)\to \mathcal{Y}\to (\overline{Y},\overline{E}+\Sigma)$ are cube-homotopic, we get $\alpha^*a=\beta^*a$ by definition of a modulus.
    This shows that $\omega^\CI F$ is LS-cube-invariant.

    Finally, we prove the second statement.
    Let $\mathcal{X}=(X,D)$ is a $\mathbb{Q}$-modulus pair and $a\in G(\mathcal{X})$.
    By the M-reciprocity of $G$, we can find a compactification $(\overline{X},\overline{D},\Sigma)$ of $\mathcal{X}$ such that $a\in G(\overline{X},\overline{D}+\Sigma)$.
    Let $T\in \Sm$ and let $\alpha,\beta\colon (T,\varnothing)\to (\overline{X},\overline{D}+\Sigma)$ be finite correspondences which are cube-homotopic.
    Then the LS-cube-invariance of $G$ implies that $\alpha^*a=\beta^*a\in G(T,\varnothing)\subset F(T)$.
    If we regard $a$ as an element of $F(\mathcal{X}^\circ)$ via the inclusion $G\subset \omega^*F$, this shows that $(\overline{X},\overline{D}+\Sigma)$ is a modulus for $a$, so we have $a\in \omega^\CI F(\mathcal{X})$.
\end{proof}

\begin{definition} \label{omegaexc}
    Let $F$ be a presheaf of abelian groups on $\Cor$.
    We define $\omega^\exc F\subset \omega^*F$ by $\omega^\exc F = (\omega^\CI F)^\exc$, i.e.,
    $$
        \omega^\exc F(X,D) = \colim_{\varepsilon\to 0}\omega^\CI F(X,(1-\varepsilon)D).
    $$
\end{definition}

By definition, $\omega_!\omega^\exc F=F$ holds if and only if $F$ is a reciprocity presheaf.
The next lemma shows that $\omega^\exc F$ is the largest excellent LS-cube-invariant subpresheaf of $\omega^*F$:

\begin{lemma}\label{omega_exc_max}
    Let $F$ be a presheaf of abelian groups on $\Cor$.
    Then $\omega^\exc F$ is excellent and LS-cube-invariant.
    Moreover, if $G\subset \omega^*F$ is an excellent LS-cube-invariant subpresheaf, then we have $G\subset \omega^\exc F$.
\end{lemma}

\begin{proof}
    The first statement follows from Lemma \ref{omega_CI_max} and Lemma \ref{exc_preserves_everything}(3).
    If $G\subset \omega^*F$ is an excellent LS-cube-invariant subpresheaf, then we have $G\subset \omega^\CI F$ by Lemma \ref{omega_CI_max}.
    Taking $({-})^{\exc}$ of both sides, we get $G\subset \omega^{\exc}F$.
\end{proof}

\begin{theorem}\label{reciprocity_presheaf_characterization}
    Let $F$ be a presheaf of abelian groups on $\Cor$.
    The following conditions are equivalent:
    \begin{enumerate}
        \item $F$ is a reciprocity presheaf.
        \item There exists an excellent LS-cube-invariant presheaf $G\in \PSh(\ulMCor,\Ab)$ such that $\omega_!G\cong F$.
        \item There exists a good LS-cube-invariant presheaf $G\in \PSh(\ulMCor,\Ab)$ such that $\omega_!G\cong F$.
    \end{enumerate}
\end{theorem}

\begin{proof}
    If $F$ is a reciprocity presheaf, then $\omega^\exc F$ gives an excellent LS-cube-invariant presheaf with $\omega_!\omega^\exc F\cong F$.
    This proves (1) $\Rightarrow$ (2).
    The implication (2) $\Rightarrow$ (3) is trivial.
    Suppose that (3) holds.
    By the semi-purity, we can regard $G$ as a subpresheaf of $\omega^*F$.
    By Lemma \ref{omega_CI_max}, we have $G\subset \omega^\CI F$, which implies $F(X)=G(X,\varnothing)\subset \omega^\CI F(X,\varnothing)$ for $X\in \Sm$.
    Therefore any element of $F(X)$ admits a modulus, so $F$ is a reciprocity presheaf.
\end{proof}

\subsection{Reciprocity sheaves} \label{subsec:reciprocity-sheaves}

\begin{definition} \label{sec;rec}
    A \emph{reciprocity sheaf} is a reciprocity presheaf which is a Nisnevich sheaf.
    We write $\RSC_\Nis$ for the category of reciprocity sheaves.
\end{definition}

\begin{lemma}\label{omega_CI_sheaf}
    If $F$ is a reciprocity sheaf, then $\omega^\CI F$ is a Nisnevich sheaf on $\ulMCor$.
\end{lemma}

\begin{proof}
    This is proved in \cite[Corollary 4.16]{RS21} for $\mathbb{Z}$-modulus pairs.
    Our task is to extend this result to $\mathbb{Q}$-modulus pairs.
    First we note that we have $\omega^\CI F\subset (\omega^\CI F)_\Nis\subset \omega^*F$ by the exactness of the Nisnevich sheafification.
    Moreover, we have $\omega^\CI F(\mathcal{X})=(\omega^\CI F)_\Nis(\mathcal{X})$ for $\mathbb{Z}$-modulus pairs.
    We will prove that $\omega^\CI F=(\omega^\CI F)_\Nis$.
    
    By Lemma \ref{omega_CI_max}, it suffices to show that $(\omega^\CI F)_\Nis$ is good and LS-cube-invariant.
    Lemma \ref{good_Nis} shows that $(\omega^\CI F)_\Nis$ is good.
    To show that $(\omega^\CI F)_\Nis$ is LS-cube-invariant, it suffices to show that the split surjection $i_0^*\colon (\omega^\CI F)_\Nis(\mathcal{X}\otimes \bcube)\to (\omega^\CI F)_\Nis(\mathcal{X})$ is injective for any log-smooth $\mathbb{Q}$-modulus pair $\mathcal{X}$.
    We take a positive integer $N$ such that $(X,ND)$ is a $\mathbb{Z}$-modulus pair.
    Consider the following commutative diagram:
    $$
    \xymatrix{
        (\omega^\CI F)_\Nis(\mathcal{X}\otimes \bcube)\ar[r]^-{i_0^*}\ar@{^(->}[d]
        &(\omega^\CI F)_\Nis(\mathcal{X})\ar@{^(->}[d]\\
        (\omega^\CI F)_\Nis((X,ND)\otimes \bcube)\ar[r]^-{i_0^*}_-{\sim}
        &(\omega^\CI F)_\Nis(X,ND).
    }
    $$
    The vertical maps are injective by the semi-purity, and the lower horizontal map $i_0^*$ is an isomorphism by the result for $\mathbb{Z}$-modulus pairs.
    Therefore the upper horizontal map is injective.
\end{proof}

\begin{lemma}\label{omega_exc_sheaf}
    If $F$ is a reciprocity sheaf, then $\omega^\exc F$ is a Nisnevich sheaf.
\end{lemma}

\begin{proof}
    This follows from Lemma \ref{omega_CI_sheaf} and Lemma \ref{exc_preserves_everything}.
\end{proof}

\begin{theorem}
    Let $F$ be a Nisnevich sheaf of abelian groups on $\Cor$.
    The following conditions are equivalent:
    \begin{enumerate}
        \item $F$ is a reciprocity sheaf.
        \item There exists an excellent LS-cube-invariant sheaf $G\in \Sh_\Nis(\ulMCor,\Ab)$ such that $\omega_!G\cong F$.
        \item There exists a good LS-cube-invariant sheaf $G\in \Sh_\Nis(\ulMCor,\Ab)$ such that $\omega_!G\cong F$.
    \end{enumerate}
\end{theorem}

\begin{proof}
    This follows from Theorem \ref{reciprocity_presheaf_characterization}, Lemma \ref{omega_CI_sheaf}, and \ref{omega_exc_sheaf}.
\end{proof}

\subsection{Functor from $\RSC_\Nis$ to $\mDA^\eff$}

\begin{definition} \label{func_modulus}
    Let $F$ be a reciprocity sheaf.
    Then $\omega^\exc F$ is an excellent LS-cube-invariant sheaf on $\ulMCor$.
    By Theorem \ref{CCI} and theorem \ref{BI}, the cohomology presheaf 
    on $\mSm$ is cube-invariant and blow-up invariant.
  Therefore, we get an object of $\mDA^\eff(k)$:
    $$
        F^\modulus:=\mathrm{R}\Gamma_{\Nis}(-,(\omega^\exc F)|_{\mSm}) \in \Sh_{\Nis}(\mSm_k,\Mod_\Z).
    $$ 
    This defines a functor $({-})^\modulus\colon \RSC_\Nis\to \mDA^\eff(k)$.
\end{definition}

The following result, stating that the cohomology of $\omega^{\exc}F$ is representable in the category of motives with modulus, is clear from the construction:

\begin{theorem} \label{thm:mapRG}
    Let $F$ be a reciprocity sheaf.
    For any $\mathcal{X}\in \mSm$, we have a canonical equivalence
    $$
        \map_{\mDA^\eff(k)}(\motive(\mathcal{X}),F^\modulus) \simeq \mathrm{R}\Gamma(\mathcal{X},\omega^\exc F).
    $$
\end{theorem}

\def\Log{\mathcal{L}og}

In \cite{Sai_log}, the third author constructed the logarithmic version of the above functor.
For any reciprocity sheaf $F \in \RSC_{\Nis}$, and $(X,D) \in \SmlSm$, we set
$$
    \Log F(X,D) := \omega^{\CI} F(X,D).
$$
The main result of \cite[\S 6]{Sai_log} is that this $\Log F$ has a functoriality for the logarithmic correspondences
and that its cohomology is $\mathrm{CI}^{\log}\cup\mathrm{BI}^{\log}$-local.
In particular, we obtain a functor 
$$
    (-)^{\log} : \RSC_{\Nis} \to \logDA^{\eff} ; \quad F \mapsto 
    \mathrm{R}\Gamma(-,\Log F)\in \Sh_{\sNis}(\SmlSm_k,\Mod_\Z).
$$
The existence of this functor, connecting the theory of reciprocity and logarithmic motives, has a fundamental importance.
The following result shows that our functor $(-)^{\modulus}$ is a ``lift'' of the functor $(-)^{\log}$.

\begin{theorem} \label{mod-log-diag}
    There exists a natural equivalence of functors $t_* \circ (-)^{\modulus} \simeq (-)^{\log}$. 
\end{theorem}

As a preparation of the proof of Theorem \ref{mod-log-diag}, we prove the following result that has its own independent interest.

\begin{theorem} \label{Rec-THA}
    Any reciprocity sheaf satisfies the tame Hasse-Arf property. More precisely, for any reciprocity sheaf $F \in \RSC_{\Nis}$, $\sX = (X,D) \in \mSm$, smooth divisor $Z$ on $X$ intersecting transversally with $|D|$, and $\varepsilon \in (0,1] \cap \mathbb{Q}$, the natural map 
    $$
        \omega^{\CI}F (X,D+\varepsilon Z) \to \omega^{\CI}F (X,D+Z)
    $$
    is an isomorphism.
\end{theorem}

\begin{proof}
    We proceed as in the proof of Theorem \ref{THA} with a small modification. Let $F \in \RSC_{\Nis}$, $(X,D) \in \mSm_X$, and $Z \subset X$ be a smooth divisor which intersects transversely with $|D|$. 
    Our goal is to prove that the natural morphism $\omega^{\CI}F(X,D+\varepsilon Z) \to \omega^{\CI}F(X,D+Z)$ is an isomorphism. Since both $F_{(X,D+\varepsilon Z)}$ and $F_{(X,D+Z)}$ are sheaves on the Nisnevich (hence Zariski) small site on $X$, we may replace $X$ by a Zariski neighborhood of a point in $X$. 
    Thus, by Lemma \ref{transversal_structure}, we may assume that there exists a diagram of the form \eqref{diag:excision} with $D':=p^*D=q^*\mathrm{pr}_1^*(D|_Z)$. 
    Set $U:=X \setminus Z$, $U':=X' \setminus Z$.
    Moreover, since the problem is still Nisnevich local on $X$, we may assume that $X$ is Henselian local. 
    For simplicity of notation, set $\widetilde{F}:=\omega^{\CI}F$.
    Note that the two squares in the induced diagram 
    \[\xymatrix{
        \widetilde{F}(X,D+\varepsilon Z) \ar[d] \ar[r] & \widetilde{F}(X',D'+\varepsilon Z) \ar[d]  & \ar[l] \widetilde{F}(Z \times \PP^1, \varepsilon Z \times \{0\}+D_{|Z}\times \PP^1) \ar[d] \\
        \widetilde{F}(X,D+ Z) \ar[r] & \widetilde{F}(X',D'+ Z) & \ar[l] \widetilde{F}(Z \times \PP^1,  Z \times \{0\}+D_{|Z}\times \PP^1)
    }\]
    are both Cartesian and coCartesian. Indeed, since $\{X' \to X, U \to X\}$ forms a Nisnevich distinguished square and since $X$ is Nisnevich local, for each $? \in \{1,\varepsilon\}$, we have a short exact sequence
    \[
        0 \to \widetilde{F}(X,D+?Z) \to \widetilde{F}(X',D'+?Z) \oplus \widetilde{F}(U,D|_U) \to \widetilde{F}(U',D|_{U'}) \to 0.
    \]
    Then the snake lemma shows that the left and the middle vertial arrows in the above diagram have the same kernels and cokernels. 
    Next, set $A:=\PP^1-\{0\}$ to lighten the notation, and for $?\in \{1,\varepsilon\}$, consider the following exact sequence associated with the elementary Nisnevich cover $\{X' \to Z \times \PP^1, Z \times A \to Z \times \PP^1\}$:
   \begin{align*}
        0 &\to \widetilde{F}(Z \times \PP^1,D|_Z \times \PP^1 + ?Z \times \{0\}) \to \widetilde{F}(X',D'+?Z) \oplus \widetilde{F}(Z \times A,D|_Z \times A) \to \widetilde{F}(U',D|_{U'}) \\
        &\to \mathrm{H}^1_{\Nis} (Z \times \PP^1, \widetilde{F}_{(Z \times \PP^1,D|_{Z} \times \PP^1 + ?Z \times \{0\})}) .
    \end{align*}
    Since $(\PP^1,\{0\}) \cong (\PP^1,\{\infty\})$, the last term is isomorphic to $\mathrm{H}^1_{\Nis} (Z,\widetilde{F}_{(Z ,D|_Z)})$ by Theorem \ref{CCI} and Lemma \ref{omega_CI_max}, which vanishes since $Z$ is henselian local.
    Then the snake lemma argument as above shows the case of the right square.

    Thus, we have proven that the squares in the above diagram are cartesian and cocartesian. 
    In particular, the left vertical arrow is an isomorphism if and only if so is the right vertical one, but the latter is indeed an isomorphism by Lemma \ref{cube_epsilon_invariance_sheaf}. 
\end{proof}

\begin{corollary} \label{cor:ciexc}
    For any reciprocity sheaf $F \in \RSC_{\Nis}$ and log-smooth $\mathbb{Q}$-modulus pair $\sX =(X,D) \in \mSm$ such that $D$ has multiplicity $\leq 1$, we have  $\omega^{\exc}F(\sX) = \omega^{\CI}F(\sX) = \omega^{\CI}F(X,|D|)$.
\end{corollary}

\begin{proof}
    This follows immediately from Definition \ref{omegaexc} and Theorem \ref{Rec-THA}.
\end{proof}

We are now ready to prove the desired comparison theorem.

\begin{proof}[Proof of Theorem \ref{mod-log-diag}]
    We first construct a natural transformation $(-)^{\log} \to t_* \circ (-)^{\modulus}$.
    By adjunction, it suffices to construct a natural transformation $t^* \circ (-)^{\log} \to (-)^{\modulus}$.
    We will do this on the level of sheaves.
    Let $F \in \RSC_{\Nis}$ and $\mathcal{X}=(X,D)\in \mSm$. We compute
    $$
        t^*F^{\log} (X,D) = F^{\log} (X,|D|) = \omega^{\CI}F(X,|D|)
        = \omega^{\CI}F(X,\varepsilon |D|) =  \omega^{\exc}F(X,\varepsilon |D|),
    $$
    where the third equality follows from the tame Hasse-Arf property \ref{Rec-THA} and the fourth equality follows from Corollary \ref{cor:ciexc}.
    Taking $\varepsilon$ small enough, the identity map on $X$ induces a morphism of $\mathbb{Q}$-modulus pairs $(X,D) \to (X,\varepsilon |D|)$ and hence the composite
    $$
        t^*F^{\log} (X,D) = \omega^{\exc}F(X,\varepsilon |D|) \to \omega^{\exc}F(X,D) = F^{\modulus} (X,D).
    $$
    It is obvious that this map does not depend on the choice of $\varepsilon$.

    Next, we prove that the natural transformation $(-)^{\log} \to t_* \circ (-)^{\modulus}$ that we have constructed above is an equivalence. Let $F \in \RSC_{\Nis}$ and consider the morphism $F^{\log} \to t_* F^{\modulus}$ in $\logDA^{\eff}$. Since the $\infty$-category $\logDA^{\eff}$ is compactly generated by representable objects $\motive (X,D), (X,D) \in \SmlSm$ \cite[Proposition 2.6.8]{LogHom}, it suffices to show that the induced map of spectra
    \begin{equation} \label{eq:mapmap}
        \map_{\logDA^{\eff}} (\motive (X,D),F^{\log}) \to \map_{\logDA^{\eff}} (\motive (X,D),t_*F^{\modulus})
    \end{equation}
    is an equivalence for any $(X,D) \in \SmlSm$.
    We have by adjunction
    $$
        \map_{\logDA^{\eff}} (\motive (X,D),t_*F^{\modulus}) \simeq \map_{\mDA^{\eff}} (t^*\motive (X,D),F^{\modulus}).
    $$
    Moreover, $t^*\motive (X,D)$ is computed as 
    $$
        t^*\motive (X,D) = \colim_{\varepsilon \to 0} \motive (X,\varepsilon D) = \motive (X,D),
    $$
    where the first equality follows from \ref{mH_tH_adjunction} and the second is the tame Hasse-Arf theorem \ref{THA}.
    Thus, the right hand side of \eqref{eq:mapmap} is equivalent to the spectrum $\mathrm{R}\Gamma(X,(\omega^{\exc}F)_{(X,D)})$ by Theorem \ref{thm:mapRG}.
    Since $(\omega^{\exc}F)_{(X,D)} = (\omega^{\CI}F)_{(X,D)}$ by Corollary \ref{cor:ciexc} (noting that \'etale maps do not increase the multiplicity of divisors), this is identified with 
    $\mathrm{R}\Gamma(X,(\omega^{\CI}F)_{(X,D)})$.
    But this also computes the left hand side of \eqref{eq:mapmap}. This completes the proof.
\end{proof}

\subsection{$(-)^{\modulus}$ for $\mathbb{A}^1$-invariant sheaves}

\begin{proposition}\label{prop:mod_for_A1_invariants}
    Let $F\in \Sh_\Nis(\Cor,\Ab)$ be an $\mathbb{A}^1$-invariant sheaf.
    Then we have $\omega^\exc F = \omega^* F$.
    In particular, $F^{\modulus}$ coincides with the sheaf of spectra
    $$
        \mathcal{X}\mapsto \mathrm{R}\Gamma(\mathcal{X}^\circ, F_{\mathcal{X}^\circ}).
    $$
\end{proposition}

\begin{proof}
    Let $X\in \Sm$ and $a\in F(X)$.
    Since $F$ is $\mathbb{A}^1$-invariant, any $\mathbb{Q}$-modulus pair $(Y,E)$ with $X=Y-|E|$ and $Y$ proper is a modulus for $a$ (see Remark \ref{rmk;ulMCorQ}).
    This shows that $\omega^{\CI}F(\mathcal{X})=F(\mathcal{X}^\circ)=\omega^*F(\mathcal{X})$.
    Since $\omega^*F$ has left continuity, we get $\omega^\exc F = \omega^* F$.
\end{proof}

\subsection{$(-)^{\modulus}$ for Hodge cohomology}

\def\tF{\widetilde{F}}
\newcommand{\xr}[1] {\xrightarrow{#1}}
\newcommand{\xl}[1] {\xleftarrow{#1}}
\newcommand{\lra}{\longrightarrow}
\newcommand{\eq}[2]{\begin{equation}\label{#1}#2 \end{equation}}
\newcommand{\ml}[2]{\begin{multline}\label{#1}#2 \end{multline}}
\newcommand{\mlnl}[1]{\begin{multline*}#1 \end{multline*}}
\newcommand{\ga}[2]{\begin{gather}\label{#1}#2 \end{gather}}
\def\inj{\hookrightarrow}

Fix a non-negative integer $q \geq 0$.
Recall the cube-invariant sheaf $\ulM \Omega^q$ on $\mSm$ from \S \ref{sec;MOmega}.
In this section, we prove that $(\Omega^q)^{\modulus}\simeq \mathrm{m}\Omega^q$.
First, in order to extend $\ulM \Omega^q$ to a sheaf on $\ulMCor$, we generalize Definition \ref{def;FFil} as follows.

\begin{definition}\label{def:ram_fil}
	Let $F\in \Sh_\Nis(\Cor,\Ab)$.
	A \emph{ramification filtration} $\Fil$ on $F$ is a collection of increasing filtrations $\{\Fil_r F(L)\}_{r\in \mathbb{Q}_{\geq 0}}$ on $F(L)$ indexed by $L\in \Phi$ which satisfies the following:
	\begin{enumerate}
		\item	For any $L\in \Phi$, we have $\Im\bigl(F(\mathcal{O}_L)\to F(L)\bigr)\subset \Fil_0F(L)$.
		\item	If $L\in \Phi$ and $L'/L$ is a finite extension with ramification index $e$, then we have
				$$
					\Tr_{L'/L}\bigl(\Fil_r F(L')\bigr)\subset \Fil_{r/e}F(L)\quad (r\in \mathbb{Q}_{\geq 0}).
				$$
	\end{enumerate}
    Here, the trace map $\Tr_{L'/L}\colon F(L')\to F(L)$ is induced by the finite correspondence $\Spec L\to \Spec L'$ which is the transpose of the canonical map.
\end{definition}

\begin{definition} \label{def;ramfil}
	Let $F\in \Sh_\Nis(\Cor,\Ab)$ and let $\Fil$ be a ramification filtration on $F$.
	Let $\mathcal{X}\in \MCor$ and $a\in F(\mathcal{X}^\circ)$.
	We say that $a$ is \emph{bounded by $D_X$} if for any $L\in \Phi$ and any commutative diagram of the following form, we have $\rho^*a\in \Fil_{v_L(\widetilde{\rho}^*D_X)}F(L)$:
        \begin{align}\label{DVR_diagram}
            \xymatrix{
                \Spec L\ar[r]^-{\rho}\ar@{^(->}[d]             &\mathcal{X}^\circ\ar@{^(->}[d]\\
                \Spec \mathcal{O}_L\ar[r]^-{\widetilde{\rho}}       &X.
            }
        \end{align}
	We write $F_{\Fil}(\mathcal{X})$ for the subgroup of $F(\mathcal{X}^\circ)$ consisting of elements bounded by $D_X$.
\end{definition}

\begin{lemma}\label{lem:ram_fil}
	Let $F\in \Sh_\Nis(\Cor,\Ab)$ and let $\Fil$ be a ramification filtration on $F$.
	\begin{enumerate}
		\item	For any $\mathcal{X},\mathcal{Y}\in \MCor^{\mathbb{Q}}$ and $\alpha\in \MCor^{\mathbb{Q}}(\mathcal{X},\mathcal{Y})$, we have $\alpha^*F_{\Fil} (\mathcal{Y})\subset F_{\Fil} (\mathcal{X})$.
		\item	$F_{\Fil}$ is a Nisnevich sheaf on $\MCor^\mathbb{Q}$.
	\end{enumerate}
\end{lemma}

\begin{proof}
    See \cite[Lemma 2.5]{Koizumi-blowup}.
\end{proof}

It is proved in \cite[Lemma 4.4]{Koizumi-blowup} that the filtration (\ref{eq:MOmega_filtration}) defines a ramification filtration on $\Omega^q$.
We write $\ulM\Omega^q$ for the sheaf on $\ulMCor$ associated to this ramification filtration. Recall that this sheaf has an explicit global formula; see \eqref{eq:MOmega}.

\begin{lemma}\label{lem:MOmega_rec}
    For any $\mathcal{X}=(X,D)\in \mSm$, we have
    $$
        \ulM\Omega^q(\mathcal{X})\subset \omega^{\exc}\Omega^q(\mathcal{X}).
    $$
\end{lemma}

\begin{proof}
By \eqref{eq:MOmega}, it suffices to show that $\ulM\Omega^q(\mathcal{X})\subset \omega^{\CI}\Omega^q(\mathcal{X})$.
Since the problem is local on $X$, we may assume that $D=\sum_{i=1}^mr_i\div(x_i)$ for some coordinate $x_1,\dots,x_n$ on $X$.
Let $a\in \ulM\Omega^q(\mathcal{X})$.
Take a normal compactification $\overline{X}$ of $X$ such that
\begin{itemize}
    \item $\overline{X}-X$ is the support of an effective Cartier divisor $\Sigma$ on $\overline{X}$, and
    \item $\div(x_i)$ extends to an effective Cartier divisor $\overline{D}_i$ for $i=1,2,\dots,m$.
\end{itemize}
Set $\overline{D}=\sum_{i=1}^m \overline{D}_i$.
Then $(\overline{X},\overline{D},\Sigma)$ is a compactification of $\mathcal{X}$.
We claim that for sufficiently large $n$, we have $a\in \ulM\Omega^q(\overline{X},\overline{D}+n\Sigma)$.
As $\ulM\Omega^q$ is LS-cube-invariant, this shows that $(\overline{X},\overline{D}+n\Sigma)$ is a modulus for $a$ and thus finishes the proof of the lemma.

Since $\overline{X}$ is quasi-compact, the claim is local on $\overline{X}$.
Let $U=\Spec A$ be an affine open subset of $\overline{X}$ such that we can write
$\overline{D}_i|_U=\div(y_i)$ and $\Sigma|_U = \div(f)$.
By \eqref{eq:MOmega}, we can write
$$
    a = \sum_{i=1}^m \dfrac{1}{x_i^{\ceil{r_i}-1}}\alpha_i + \sum_{i=1}^m\beta_i\dfrac{dx_i}{x_i^{\ceil{r_i}}},
$$
where $\alpha_i\in \Omega^q_{A[1/f]}$, $\beta_i\in \Omega^{q-1}_{A[1/f]}$.
Recall that $\overline{D}_i|_U$ extends the divisor $\div(x_i)$ on $\Spec A[1/f]$.
Therefore we have $x_i=e_iy_i$ for some $e_i\in A[1/f]^\times$.
Let $E_i$ be the Cartier divisor on $U$ defined by $e_i$.
Then we have $|E_i|\subset |\Sigma|$, so there is some $n_1>0$ such that $r_iE_i\leq n_1\Sigma$ holds for $i=1,2,\dots,m$.
Moreover, there is some $n_2>0$ such that $f^{n_2}\alpha_i\in \Omega^q_A$ and $f^{n_2}\beta_i\in \Omega^{q-1}_A$ holds for $i=1,2,\dots,m$.
Take $n\geq n_1+n_2$.
Then we have $f^na\in \ulM\Omega^q(U,\overline{D}|_U)$ and hence $a\in \ulM\Omega^q(U,\overline{D}|_U+n\Sigma|_U)$.
This finishes the proof of the claim.
\end{proof}

\begin{thm}\label{thm;MOmega_mod}
For any $\mathcal{X}=(X,D)\in \mSm$, we have
\[ \ulM \Omega^q(\mathcal{X}) = \omega^\exc \Omega^q(\mathcal{X})\]
as subgroups of $\Omega^q(\mathcal{X}^\circ)$.
In other words, we have $(\Omega^q)^{\modulus}\simeq \mathrm{m}\Omega^q$.
\end{thm}

\begin{proof}
By Lemma \ref{lem:MOmega_rec}, we have $\ulM \Omega^q(\mathcal{X}) \subset \omega^\exc \Omega^q(\mathcal{X})$.
Let us prove the opposite inclusion.
It suffices to prove that for each $L\in \Phi$ and $r\in \Q_{\geq 0}$, 
\eq{eq0;thm;MOmega_mod}{
\omega^\exc \Omega^q(\sO_L,\fm_L^{\lceil r \rceil})\subset  \Fil_r\Omega^q(L),}
where 
$$
    \mathrm{Fil}_r \Omega^q (L)
    =\begin{cases}
        \Omega^q(\mathcal{O}_L) & (r=0), \\
        t^{-\lceil r \rceil + 1} \cdot \Omega^q (\mathcal{O}_L)(\log) & (r>0).
    \end{cases}
$$
Since the both sides of \eqref{eq0;thm;MOmega_mod} depend only on $\lceil r \rceil$, we may assume $r\in \Z$.
We use existing results computing $\omega^{\CI}\Omega^q(\mathcal{O}_L,\mathfrak{m}_L^r)$.
If $\ch(k)=0$, the equality $\omega^{\CI}\Omega^q(\mathcal{O}_L,\mathfrak{m}_L^r)=\Fil_r\Omega^q(L)$ is proved in \cite[Theorem 6.4]{RS21}.
This proves \eqref{eq0;thm;MOmega_mod} in case $\ch(k)=0$.
If $\ch(k)=p>0$, we have
$$
    \omega^{\CI}\Omega^q(\mathcal{O}_L,\mathfrak{m}_L^r)
    =\begin{cases}
        \Omega^q(\mathcal{O}_L) & (r=0), \\
        t^{-r + 1} \cdot \Omega^q (\mathcal{O}_L)(\log) & (r\in \mathbb{Z}_{>0}-p\mathbb{Z}_{>0}),\\
        t^{-r} \cdot \Omega^q (\mathcal{O}_L) & (r\in p\mathbb{Z}_{>0})
    \end{cases}
$$
by \cite[Corollary 6.8]{RSram3}.
This proves \eqref{eq0;thm;MOmega_mod} for $r\in \mathbb{Z}_{>0}-p\mathbb{Z}_{>0}$.
When $r\in p\mathbb{Z}_{>0}$, we have to take care of the difference between $\omega^{\CI}$ and $\omega^{\exc}$.
Let $a\in \omega^{\exc}\Omega^q(\mathcal{O}_L,\mathfrak{m}_L^{r})$, where $r\in p\mathbb{Z}_{>0}$.
Then, we have $a\in \omega^{\CI}\Omega^q(\mathcal{O}_L,\mathfrak{m}_L^{r-(1/N)})$ for some positive integer $N$ with $p\nmid N$.
Take a totally ramified extension $L'/L$ with ramification index $N$.
Let $\pi\colon \Spec L'\to \Spec L$ denote the canonical map.
Then, we have
$$
    \pi^*a \in \omega^{\CI}\Omega^q(\mathcal{O}_{L'},\mathfrak{m}_{L'}^{rN-1})=\Fil_{rN-1}\Omega^q(L') \subset \Fil_{rN}\Omega^q(L').
$$
Since $\Fil$ is a ramification filtration on $\Omega^q$, we get
$$
    Na=\Tr_{L'/L}(\pi^*a)\in \Fil_r\Omega^q(L).
$$
Since $N$ is invertible in $k$, we get $a\in \Fil_r\Omega^q(L)$.
This proves \eqref{eq0;thm;MOmega_mod} for $r\in p\mathbb{Z}_{>0}$.
\end{proof}

\begin{remark}
    In the proof of Theorem \ref{thm;MOmega_mod}, we employed the result from \cite{RSram3}. We can also produce an alternative proof following the argument of \cite[Theorem 6.4]{RS21} by upgrading the theory of local symbols for $\mathbb{Q}$-modulus pairs. However, we decided not to include the detailed exposition here for the cleanliness of the present paper.
\end{remark}

\subsection{$(-)^{\modulus}$ for unramified cohomology}\label{sec:Hone_exc}

Assume that $\ch(k)=p>0$.
In this subsection, we interpret Brylinski-Kato's ramification filtration on the \'etale cohomology in our framework.

\begin{definition} \label{def;unramcoh}
    Let $q\geq 0$ and $m\geq 1$.
    We define a Nisnevich sheaf $\mathrm{H}^{q+1}_{\ur, m}$ on $\Sm$ to be the Nisnevich sheafification of the presheaf
    $$
        X\mapsto \mathrm{H}^{q+1}_{\et}(X,\mathbb{Z}/m(q)),
    $$
    where $\mathbb{Z}/m(q)$ is the mod-$m$ \'etale motivic complex of weight $q$.
    Similarly, we define a Nisnevich sheaf $\mathrm{H}^{q+1}_{\ur}$ on $\Sm$ to be the Nisnevich sheafification of the presheaf
    $$
        X\mapsto \mathrm{H}^{q+1}_{\et}(X,\mathbb{Q}/\mathbb{Z}(q)).
    $$
\end{definition}

\begin{remark}
    Since the presheaf $\mathrm{H}^1_\et({-},\mathbb{Z}/m)$ is already a Nisnevich sheaf, we write $\mathrm{H}^1_{\et,m}$ for $\mathrm{H}^1_{\ur,m}$.
    Similarly, we write $\mathrm{H}^1_{\et}$ for $\mathrm{H}^1_{\ur}$.
\end{remark}

By definition, we have $\mathrm{H}^{q+1}_{\ur}\simeq \varinjlim_{m}\mathrm{H}^{q+1}_{\ur,m}$.
If we write $m=m'p^n$ with $p\nmid m'$, then we have $\mathrm{H}^{q+1}_{\ur,m}\simeq \mathrm{H}^{q+1}_{\ur,m'} \oplus \mathrm{H}^{q+1}_{\ur,p^n}$.
Note that $\mathrm{H}^{q+1}_{\ur,p^0}=0$.
Moreover, there is an equivalence
$$
    \mathbb{Z}/p^n(q)[q+1]\simeq \cofib (\Witt_n\Omega^q\xrightarrow{1-F}\Witt_n\Omega^q/dV^{n-1}\Omega^{q-1})
$$
in the derived category of \'etale sheaves on $\Sm$.
Therefore, for $n \geq 1$, the sheaf $\mathrm{H}^{q+1}_{\ur,p^n}$ is isomorphic to the cokernel of $(1-F)\colon \Witt_n\Omega^q\to\Witt_n\Omega^q/dV^{n-1}\Omega^{q-1}$ in $\Sh_\Nis(\Sm)$ via the connecting map
$$
    \delta\colon \Witt_n\Omega^q/dV^{n-1}\Omega^{q-1}\to \mathrm{H}^{q+1}_{\ur,p^n}.
$$
In particular, the sheaf $\mathrm{H}^{q+1}_{\ur,p^n}$ has a natural structure of a sheaf on $\Cor$.
Moreover, since $\RSC_\Nis\subset \Sh_\Nis(\Cor,\Ab)$ is closed under taking quotients, it follows that $\mathrm{H}^{q+1}_{\ur,p^n}$ is a reciprocity sheaf.

\begin{definition}
    For $L\in \Phi$, we define a filtration $\{\Fil_r\Witt_n(L)\}_{r\in \mathbb{Q}_{\geq 0}}$ on $\Witt_n(L)$ by
    \begin{equation} \label{eq:MW_filtration}
    \mathrm{Fil}_r \Witt_n (L)
    =\begin{cases}
        \Witt_n(\mathcal{O}_L) & (r=0), \\
        \{a\in \Witt_n(L)\mid [t^{\lceil r \rceil - 1}] \cdot \mathrm{F}^{n-1}(a)\in \Witt_n(\mathcal{O}_L)\} & (r>0).
    \end{cases}
    \end{equation}
    It is proved in \cite[Lemma 4.11]{Koizumi-blowup} that the filtration \eqref{eq:MW_filtration} defines a ramification filtration on $\Witt_n$ (see Definition \ref{def;ramfil} for the terminology).
    We write $\ulM\Witt_n$ for the sheaf on $\ulMCor$ associated to this ramification filtration.
\end{definition}

\begin{lemma}
    The filtration \eqref{eq:MW_filtration} coincides with the filtration \eqref{eq:MWOmega_filtration} for $q=0$.
    Consequently, the sheaf $\ulM\Witt_n$ on $\mSm$ is canonically isomorphic to the sheaf $\ulM\Witt_n\Omega^0$.
\end{lemma}

\begin{proof}
    See \cite[Proposition 2.11 (2)]{Shiho}.
\end{proof}

\begin{lemma}
    For any $\mathcal{X}=(X,D)\in \mSm$, we have
    \begin{align}\label{eq:MW}
    \ulM\Witt_n(\mathcal{X})=\{(a_0,\dots,a_{n-1})\in \Witt_n(\mathcal{X}^\circ)\mid a_j^{p^{n-j-1}}\in \Gamma(X,\mathcal{O}_X(\ceil{D}-|D|))\}.
    \end{align}
\end{lemma}

\begin{proof}
    See \cite[Lemma 4.13]{Koizumi-blowup}.
\end{proof}

\begin{definition}\label{def:BK-filtration}
    The \emph{Brylinski-Kato filtration} $\{\Fil_r\mathrm{H}^{1}_{\et,p^n}(L)\}_{r\in \mathbb{Q}_{\geq 0}}$ on $\mathrm{H}^{1}_{\et,p^n}$ is defined by
    \begin{equation} \label{eq:Hone_filtration}
    \Fil_r\mathrm{H}^{1}_{\et,p^n}(L)
    =
    \Im (\Fil_r\Witt_n\xrightarrow{\delta}\mathrm{H}^1_{\et,p^n}(L)).
    \end{equation}
    For $q\geq 0$, the \emph{Brylinski-Kato filtration} $\{\Fil_r\mathrm{H}^{q+1}_{\ur,p^n}(L)\}_{r\in \mathbb{Q}_{\geq 0}}$ on $\mathrm{H}^{q+1}_{\ur,p^n}$ is defined by
    \begin{equation} \label{eq:Hur_filtration}
    \Fil_r\mathrm{H}^{q+1}_{\ur,p^n}(L)
    =
    \begin{cases}
        \mathrm{H}^{q+1}_{\ur,p^n}(\mathcal{O}_L)&(r=0),\\
        \Im (\Fil_r\mathrm{H}^1_{\et,p^n}(L)\otimes \mathrm{K}^{\mathrm{M}}_q(L)\to \mathrm{H}^{q+1}_{\ur,p^n}(L))  &(r>0).
    \end{cases}
    \end{equation}
    For $m=m'p^n$ with $p\nmid m'$, we set
    $$
        \Fil_r\mathrm{H}^{q+1}_{\ur,m}(L)=
        \begin{cases}
            \mathrm{H}^{q+1}_{\ur,m}(\mathcal{O}_L)&(r=0),\\
            \mathrm{H}^{q+1}_{\ur,m'}(L)\oplus \Fil_r\mathrm{H}^{q+1}_{\ur,p^n}(L)&(r>0).
        \end{cases}
    $$
    We set $\Fil_r\mathrm{H}^{q+1}_{\ur}(L)= \varinjlim_{m} \Fil_r\mathrm{H}^{q+1}_{\ur,m}(L)$.
    The sheaf associated to the Brylinski-Kato filtration is denoted by $\ulM\mathrm{H}^{q+1}_{\ur,m}, \ulM\mathrm{H}^{q+1}_{\ur}\in \Sh_\Nis(\mSm,\Ab)$.
\end{definition}

\begin{lemma}\label{lem:MH1_is_Im_of_delta}
    For any $\mathcal{X}=(X,D)\in \mSm$, the sheaf $(\ulM\mathrm{H}^1_{\et,p^n})_\mathcal{X}$ coincides with the image of the morphism
    $$
    \delta\colon (\ulM\Witt_n)_\mathcal{X}\to
    (\omega^*\mathrm{H}^1_{\et,p^n})_\mathcal{X}.
    $$
\end{lemma}

\begin{proof}
    By the definition of the Brylinski-Kato filtration on $\mathrm{H}^1_{\et,p^n}$, the morphism $\delta$ factors as
    $$
        \delta\colon (\ulM\Witt_n)_\mathcal{X}\to
        (\ulM\mathrm{H}^1_{\et,p^n})_\mathcal{X}\to 
        (\omega^*\mathrm{H}^1_{\et,p^n})_\mathcal{X}.
    $$
    This shows that the image of $\delta\colon \ulM\Witt_n(\mathcal{X})\to \mathrm{H}^1_{\et,p^n}(\mathcal{X}^\circ)$ is contained in $\ulM\mathrm{H}^1_{\et,p^n}(\mathcal{X})$.

    Conversely, suppose that $a\in \ulM\mathrm{H}^1_{\et,p^n}(\mathcal{X})$.
    Write $D=\sum_{i=1}^m r_iD_i$, where $D_i$ are smooth and irreducible.
    Let $\xi_i$ be the generic point of $D_i$ and $L_i = \Frac\mathcal{O}_{X,\xi_i}^h$.
    Our assumption implies that $a\in \Fil_{r_i}\mathrm{H}^1_{\et,p^n}(L_i)$ for $i=1,2,\dots,m$.
    By \cite[Proposition 1.31]{Yatagawa17}, this implies that $a$ lies in the image of the morphism of sheaves
    $$
        (\ulM\Witt_n)_{\mathcal{X}}\xrightarrow{\delta}(\omega^*\mathrm{H}^1_{\et,p^n})_{\mathcal{X}}.
    $$
    Therefore, the sheaf $(\ulM\mathrm{H}^1_{\et,p^n})_\mathcal{X}$ is contained in the image of $(\ulM\Witt_n)_\mathcal{X}$ under $\delta$.
\end{proof}

Next, we show that $\omega^\exc \mathrm{H}^1_\et$ and $\ulM \mathrm{H}^1_\et$ coincides on $\mSm$:

\begin{theorem}\label{thm:MHone_exc}
    For any $\mathcal{X}=(X,D)\in \mSm$, we have
    \[ \ulM \mathrm{H}^1_\et(\mathcal{X}) = \omega^\exc \mathrm{H}^1_\et(\mathcal{X})\]
    as subgroups of $\mathrm{H}^1_{\et}(\mathcal{X}^\circ)$.
    In particular, the sheaf $\ulM\mathrm{H}^1_\et|_{\mSm}$ on $\mSm$ is $(\CI\cup\BI)$-local and hence defines an object $\mathrm{mH}^1_\et$ of $\mDA^\eff(k)$.
\end{theorem}

\begin{lemma}\label{lem:MW_rec}
    For any $\mathcal{X}=(X,D)\in \mSm$, we have
    $$
        \ulM\Witt_n(\mathcal{X})\subset \omega^{\exc}\Witt_n(\mathcal{X}).
    $$
\end{lemma}

\begin{proof}
    By \eqref{eq:MW}, it suffices to show that $\ulM\Witt_n(\mathcal{X})\subset \omega^{\CI}\Witt_n(\mathcal{X})$.
    Since the problem is local on $X$, we may assume that $D=\sum_{i=1}^mr_i\div(x_i)$ for some coordinate $x_1,\dots,x_n$ on $X$.
Let $a\in \ulM\Witt_n(\mathcal{X})$.
Take a normal compactification $\overline{X}$ of $X$ such that
\begin{itemize}
    \item $\overline{X}-X$ is the support of an effective Cartier divisor $\Sigma$ on $\overline{X}$, and
    \item $\div(x_i)$ extends to an effective Cartier divisor $\overline{D}_i$ for $i=1,2,\dots,m$.
\end{itemize}
Set $\overline{D}=\sum_{i=1}^m \overline{D}_i$.
Then $(\overline{X},\overline{D},\Sigma)$ is a compactification of $\mathcal{X}$.
We claim that for sufficiently large $n$, we have $a\in \ulM\Witt_n(\overline{X},\overline{D}+n\Sigma)$.
As $\ulM\Witt_n$ is LS-cube-invariant, this shows that $(\overline{X},\overline{D}+n\Sigma)$ is a modulus for $a$ and thus finishes the proof of the lemma.

Since $\overline{X}$ is quasi-compact, the claim is local on $\overline{X}$.
Let $U=\Spec A$ be an affine open subset of $\overline{X}$ such that we can write
$\overline{D}_i|_U=\div(y_i)$ and $\Sigma|_U = \div(f)$.
By \eqref{eq:MW}, we can write
$$
    a = (a_0,a_1,\dots,a_{n-1}),\quad x_1^{\ceil{r_1}-1}\cdots x_m^{\ceil{r_m}-1}a_j^{p^{n-j-1}}\in A[1/f].
$$
Recall that $\overline{D}_i|_U$ extends the divisor $\div(x_i)$ on $\Spec A[1/f]$.
Therefore we have $x_i=e_iy_i$ for some $e_i\in A[1/f]^\times$.
Let $E_i$ be the Cartier divisor on $U$ defined by $e_i$.
Then we have $|E_i|\subset |\Sigma|$, so there is some $n_1>0$ such that $r_iE_i\leq n_1\Sigma$ holds for $i=1,2,\dots,m$.
Moreover, there is some $n_2>0$ such that
$$
f^{n_2}x_1^{\ceil{r_1}-1}\cdots x_m^{\ceil{r_m}-1}a_j^{p^{n-j-1}}\in A.
$$
Take $n\geq n_1+n_2$.
Then we have
$f^ny_1^{\ceil{r_1}-1}\cdots y_m^{\ceil{r_m}-1}a_j^{p^{n-j-1}}\in A$
and hence $a\in \ulM\Witt_n(U,\overline{D}|_U+n\Sigma|_U)$.
This finishes the proof of the claim.
\end{proof}

\begin{lemma}\label{lem:MHone_rec}
    For any $\mathcal{X}=(X,D)\in \mSm$, we have
    $$
        \ulM\mathrm{H}^1_{\et, p^n}(\mathcal{X})\subset \omega^{\exc}\mathrm{H}^1_{\et, p^n}(\mathcal{X}).
    $$
\end{lemma}

\begin{proof}
Recall from Lemma \ref{lem:MH1_is_Im_of_delta} that $(\ulM\mathrm{H}^1_{p^n})_\mathcal{X}$ coincides with the image of the morphism
$$
    \delta\colon (\ulM\Witt_n)_\mathcal{X}\to
    (\omega^*\mathrm{H}^1_{\et,p^n})_\mathcal{X}.
$$
Consider the following commutative diagram of Nisnevich sheaves on $X$:
$$
    \xymatrix{
    (\ulM\Witt_n)_\mathcal{X}\ar@{^(->}[r]\ar@{->>}[d]     &(\omega^*\Witt_n)_\mathcal{X}\ar[d]&(\omega^{\exc}\Witt_n)_\mathcal{X}\ar@{_(->}[l]\ar[d]\\
    (\ulM\mathrm{H}^1_{\et, p^n})_\mathcal{X}\ar@{^(->}[r] &(\omega^*\mathrm{H}^1_{\et, p^n})_\mathcal{X}&(\omega^{\exc}\mathrm{H}^1_{\et, p^n})_\mathcal{X}.\ar@{_(->}[l]
    }
$$
By Lemma \ref{lem:MW_rec}, we have $(\ulM\Witt_n)_\mathcal{X}\subset (\omega^{\exc}\Witt_n)_\mathcal{X}$.
It follows from the above commutative diagram that $(\ulM\mathrm{H}^1_{\et, p^n})_\mathcal{X}$ is contained in $(\omega^{\exc}\mathrm{H}^1_{\et, p^n})_\mathcal{X}$.
\end{proof}

\begin{proposition}\label{prop;MHone_mod}
For any $\mathcal{X}=(X,D)\in \mSm$, we have
\[ \ulM \mathrm{H}^1_{\et, p^n}(\mathcal{X}) = \omega^\exc \mathrm{H}^1_{\et, p^n}(\mathcal{X})\]
as subgroups of $\mathrm{H}^1_{\et, p^n}(\mathcal{X}^\circ)$.
\end{proposition}

\begin{proof}
By Lemma \ref{lem:MHone_rec}, we have $\ulM \mathrm{H}^1_{\et, p^n}(\mathcal{X}) \subset \omega^\exc\mathrm{H}^1_{\et, p^n}(\mathcal{X})$.
Let us prove the opposite inclusion.
It suffices to prove that for each $L\in \Phi$ and $r\in \Q_{\geq 0}$, 
\eq{eq0;prop;MHone_mod}{
\omega^\exc \mathrm{H}^1_{\et, p^n}(\sO_L,\fm_L^{\lceil r \rceil})\subset  \Fil_r\mathrm{H}^1_{\et, p^n}(L).}
Since the both sides of \eqref{eq0;prop;MHone_mod} depend only on $\lceil r \rceil$, we may assume $r\in \Z$.
We use existing results computing $\omega^{\CI}\mathrm{H}^1_{\et, p^n}(\mathcal{O}_L,\mathfrak{m}_L^r)$.
Namely, we have
$$
    \omega^{\CI}\mathrm{H}^1_{\et, p^n}(\mathcal{O}_L,\mathfrak{m}_L^r)
    =\Fil_r^{\mathrm{nonlog}}\mathrm{H}^1_{\et, p^n} (L)
$$
by \cite[Theorem 8.8]{RS21}, where $\Fil_r^{\mathrm{nonlog}}$ is the \emph{non-logarithmic filtration} on $\mathrm{H}^1_{\et, p^n}(L)$.
Since $\Fil_r^{\mathrm{nonlog}}$ coincides with $\Fil_r$ when $r\not\in p\mathbb{Z}_{>0}$, this proves \eqref{eq0;prop;MHone_mod} for $r\not\in p\mathbb{Z}_{>0}$.
When $r\in p\mathbb{Z}_{>0}$, we have to take care of the difference between $\omega^{\CI}$ and $\omega^{\exc}$.
Let $a\in \omega^{\exc}\mathrm{H}^1_{\et, p^n}(\mathcal{O}_L,\mathfrak{m}_L^{r})$, where $r\in p\mathbb{Z}_{>0}$.
Then, we have $a\in \omega^{\CI}\mathrm{H}^1_{\et, p^n}(\mathcal{O}_L,\mathfrak{m}_L^{r-(1/N)})$ for some positive integer $N$ with $p\nmid N$.
Take a totally ramified extension $L'/L$ with ramification index $N$.
Let $\pi\colon \Spec L'\to \Spec L$ denote the canonical map.
Then, we have
$$
    \pi^*a \in \omega^{\CI}\mathrm{H}^1_{\et, p^n}(\mathcal{O}_{L'},\mathfrak{m}_{L'}^{rN-1})=\Fil_{rN-1}\mathrm{H}^1_{\et, p^n}(L') \subset \Fil_{rN}\mathrm{H}^1_{\et, p^n}(L').
$$
Since $\Fil$ is a ramification filtration on $\mathrm{H}^1_{\et, p^n}$, we get
$$
    Na=\Tr_{L'/L}(\pi^*a)\in \Fil_r\mathrm{H}^1_{\et, p^n}(L).
$$
Since $p\nmid N$, we get $a\in \Fil_r\mathrm{H}^1_{\et, p^n}(L)$.
This proves \eqref{eq0;prop;MHone_mod} for $r\in p\mathbb{Z}_{>0}$.
\end{proof}

\begin{proof}[Proof of Theorem \ref{thm:MHone_exc}]
    By Proposition \ref{prop;MHone_mod}, we have the equality for the $p$-primary part.
    It suffices to show that for any positive integer $m$ with $p\nmid m$, we have
    $$
        \ulM\mathrm{H}^1_{\et,m}(\mathcal{X})=\omega^\exc\mathrm{H}^1_{\et,m}(\mathcal{X}).
    $$
    By definition, we have $\ulM\mathrm{H}^1_{\et,m}(\mathcal{X})=\mathrm{H}^1_{\et,m}(\mathcal{X}^\circ)=\omega^*\mathrm{H}^1_{\et,m}(\mathcal{X})$.
    Therefore, the above equality follows from the $\mathbb{A}^1$-invariance of $\mathrm{H}^1_{\et,m}$ and Proposition \ref{prop:mod_for_A1_invariants}.
\end{proof}

We expect that Theorem \ref{thm:MHone_exc} extends to higher unramified cohomologies:

\begin{conj}\label{conj:unramified}
    For any $q\geq 0$ and $\mathcal{X}=(X,D)\in \mSm$, we have
    \[ \ulM \mathrm{H}^{q+1}_\et(\mathcal{X}) = \omega^\exc \mathrm{H}^{q+1}_\ur(\mathcal{X})\]
    as subgroups of $\mathrm{H}^{q+1}_{\ur}(\mathcal{X}^\circ)$.
    In particular, the sheaf $\ulM\mathrm{H}^{q+1}_\ur$ on $\mSm$ is $(\CI\cup\BI)$-local and hence defines an object $\mathrm{mH}^{q+1}_\ur$ of $\mDA^\eff(k)$.
    Moreover, there is an equivalence
    $$
        \Omega_{S^1_t}(\mathrm{mH}^{q+1}_{\ur}) \simeq \mathrm{mH}^q_{\ur}.
    $$
\end{conj}

\subsection{$(-)^{\modulus}$ for rank $1$ connections}\label{sec:Conn_exc}

The analogy between the wild ramifications of $\ell$-adic sheaves on varieties in positive characteristic and the irregular singularities of integrable connections on varieties in characteristic $0$ has been pointed out by many authors, e.g., \cite{Deligne70}, \cite{Kato94}.
The goal of this subsection is to show that the filtrations capturing the irregular singularities of rank $1$ connections are representable in our motivic homotopy category.

Let $k$ be a field of characteristic $0$.
For any $X \in \Sm_k$, consider the morphism
\[
    \dlog_X : \OO_X^\times \to \Omega^1_X; \quad u \mapsto \dlog (u):=du/u.
\]
Recall from \cite[Lemma 6.9]{RSram1} that this induces a morphism $\dlog : \OO^\times_X \to Z\Omega^1_X \subset \Omega^1_X$ in $\RSC_{\Nis}$, where $Z\Omega^1_X \subset \Omega^1_X$ is the kernel of the differential $d : \Omega^1_X \to \Omega^2_X$.

It is well-known that, for each $X \in \Sm_k$, the cokernel of $\dlog_X\colon \OO^\times_X \to \Omega^1_X$ in the Zariski topology is canonically identified with the Zariski sheaf of isomorphism classes of rank $1$ connections, which we denote by $\Conn^1_X$:
\[
    \Conn^1_X \cong \Coker_\Zar(\dlog_X : \OO_X^\times \to \Omega^1_X).
\]
Similarly, denoting by $\Conn^1_{\int,X} \subset \Conn^1_X$ the subsheaf of integrable rank $1$ connections on $X$, we have a canonical identification
\[
    \Conn^1_{\int,X} \cong \Coker_\Zar (\dlog_X : \OO_X^\times \to Z\Omega^1_X).
\]
Since $\RSC \subset \PSh(\Cor,\Ab)$ is closed under taking quotients (see \cite[Remark 2.2.5]{KSY2}), the cokernel presheaf of $\dlog\colon \mathcal{O}^\times\to \Omega^1$ (resp. $\dlog\colon \mathcal{O}^\times\to Z\Omega^1$) is a reciprocity presheaf.
By \cite[Corollary 3.2.2]{KSY2}, we have
\begin{align*}
    \Conn^1 &:= \Coker_\Zar(\dlog : \OO^\times \to \Omega^1)\cong \Coker_\Nis(\dlog : \OO^\times \to \Omega^1),\\
    \Conn^1_\int &:= \Coker_\Zar(\dlog : \OO^\times \to Z\Omega^1)\cong \Coker_\Nis(\dlog : \OO^\times \to Z\Omega^1).
\end{align*}
Moreover, \cite[Theorem 2.4.1 (1)]{KSY2} shows that $\Conn^1$ and $\Conn^1_\int$ are reciprocity sheaves.

\begin{definition}
    For $L\in \Phi$, we define the \emph{irregularity filtration} $\{\Fil_r\Conn^1(L)\}_{r\in \mathbb{Q}_{\geq 0}}$ on $\Conn^1(L)$ by
    $$
        \Fil_r\Conn^1(L)=
            \Im(\Fil_r\Omega^1(L)\to \Conn^1(L)).
    $$
    We also define the filtration $\{\Fil_r\Conn^1_\int(L)\}_{r\in \mathbb{Q}_{\geq 0}}$ on $\Conn^1_\int(L)$ to be the restriction of the above filtration to $\Conn^1_\int$.
    The fact that $\{\Fil_r \Omega^1(L)\}_{r \in \mathbb{Q}_{\geq 0}}$ defines a ramification filtration (see Definition \ref{def;ramfil} for the terminology) implies that so does $\{\Fil_r \Conn^1(L)\}_{r \in \mathbb{Q}_{\geq 0}}$, and hence we obtain a sheaf $\ulM \Conn^1 := (\Conn^1)_{\Fil}$ on $\ulMCor$.
    Similarly, we obtain a sheaf $\ulM\Conn^1_\int := (\Conn^1_\int)_{\Fil}$ on $\ulMCor$.
\end{definition}

\begin{lemma}\label{lem:conn_epi}
    Let $\mathcal{X}=(X,D) \in \mSm$ and write $j\colon U:=X-D\to X$ for the inclusion.
    Then, the canonical morphism of Zariski sheaves
    $$
        j_*\Omega^1_U\to j_*\Conn_U
    $$
    is an epimorphism.
\end{lemma}

\begin{proof}
    We may assume that $X$ is connected.
    Let $\overline{k}_X$ denote the separable closure of $k$ in the function field of $X$, regarded as a constant Zariski sheaf on $X$.
    By definition, we have an exact sequence of Zariski sheaves
    $$
        0\to \mathcal{O}_U^\times/\overline{k}_X\to \Omega^1_U\to \Conn^1_U\to 0
    $$
    on $U$.
    Therefore, it suffices to show that $\mathrm{R}^1j_*(\mathcal{O}^\times_U/\overline{k}_X)=0$.
    Consider the exact sequence
    $$
        \mathrm{R}^1j_*\mathcal{O}^\times_U\to \mathrm{R}^1j_*(\mathcal{O}^\times_U/\overline{k}_X)\to \mathrm{R}^2j_*\overline{k}_X.
    $$
    The last term is $0$ because constant sheaves are flasque in the Zariski topology.
    On the other hand, for $x\in X$ we have
    $$  
    (\mathrm{R}^1j_*\mathcal{O}^\times_U)_x\cong\Pic(\Spec \mathcal{O}_{X,x}\times_X U)=0,
    $$
    because $\Pic(\Spec \mathcal{O}_{X,x})\to \Pic(\Spec \mathcal{O}_{X,x}\times_X U)$ is surjective \cite[Lemma 0BD9]{stacks-project}.
    This shows that $\mathrm{R}^1j_*(\mathcal{O}^\times_U/\overline{k}_X)=0$.
\end{proof}

\begin{lemma}\label{lem:conn_purity}
    Let $\mathcal{X}=(X,D) \in \mSm$.
    Write $D=\sum_{i=1}^m r_iD_i$, where $D_i$ are smooth and irreducible.
    Let $\xi_i$ be the generic point of $D_i$ and $L_i = \Frac\mathcal{O}_{X,\xi_i}^h$.
    Write $\iota_i\colon \Spec L_i\to X$ for the canonical morphism.
    Then, the following sequence of Zariski sheaves on $X$ is exact:
    $$
        \ulM\Omega^1_\mathcal{X}\to (\omega^*\Conn^1)_\mathcal{X}\to \bigoplus_{i=1}^m\iota_{i,*}\left(\dfrac{\Conn^1(L_i)}{\Fil_{r_i}\Conn^1(L_i)}\right).
    $$
\end{lemma}

\begin{proof}
    Consider the following commutative diagram of Zariski sheaves on $X$:
    \[
    \xymatrix{
    \ulM\Omega^1_\mathcal{X}\ar[r]\ar@{=}[d]&
    (\omega^*\Omega^1)_\mathcal{X}\ar[r]\ar[d]^-{\alpha}&
    \displaystyle\bigoplus_{i=1}^m\iota_{i,*}\left(\dfrac{\Omega^1(L_i)}{\Fil_{r_i}\Omega^1(L_i)}\right)\ar[d]^-{\beta}\\
    \ulM\Omega^1_\mathcal{X}\ar[r]&
    (\omega^*\Conn^1)_\mathcal{X}\ar[r]&
    \displaystyle\bigoplus_{i=1}^m\iota_{i,*}\left(\dfrac{\Conn^1(L_i)}{\Fil_{r_i}\Conn^1(L_i)}\right)
    }
    \]
    The top row is exact by \eqref{eq:MOmega}.
    The morphism $\alpha$ is an epimorphism by lemma \ref{lem:conn_epi}.
    Moreover, the morphism $\beta$ is an isomorphism because the image of
    $\dlog\colon L_i^\times\to \Omega^1(L_i)$
    is contained in $\Fil_{r_i}\Omega^1(L_i)$.
    The exactness now follows by a diagram chasing.
\end{proof}

\begin{lemma}\label{lem:MConn1_is_Im_of_MOmega}
    For any $\mathcal{X}=(X,D) \in \mSm$, the sheaf $\ulM\Conn^1_\mathcal{X}$ coincides with the image of the morphism
    $$
    \ulM\Omega^1_\mathcal{X}\to(\omega^*\Conn^1)_\mathcal{X},
    $$
    which is induced by adjunction from the quotient morphism $\omega_! \ulM \Omega^1 \cong \Omega^1 \twoheadrightarrow \Conn^1$.
\end{lemma}

\begin{proof}
    By the definition of the irregularity filtration on $\Conn^1$, the morphism in question factors as
    $$
        \ulM\Omega^1_\mathcal{X}\to
        \ulM\Conn^1_\mathcal{X}\to 
        (\omega^*\Conn^1)_\mathcal{X}.
    $$
    This shows that the image of $\ulM\Omega^1_\mathcal{X}\to(\omega^*\Conn^1)_\mathcal{X}$ is contained in $\ulM\Conn^1_\mathcal{X}$.

    Conversely, suppose that $a\in \ulM\Conn^1(\mathcal{X})$.
    Write $D=\sum_{i=1}^m r_iD_i$, where $D_i$ are smooth and irreducible.
    Let $\xi_i$ be the generic point of $D_i$ and $L_i = \Frac\mathcal{O}_{X,\xi_i}^h$.
    Our assumption implies that $a\in \Fil_{r_i}\Conn^1(L_i)$ for $i=1,2,\dots,m$.
    By Lemma \ref{lem:conn_purity}, this implies that, Zariski locally on $X$, the element $a$ lies in the image of $\ulM\Omega^1(\mathcal{X})$.
    This proves that $\ulM\Conn^1_\mathcal{X}$ is contained in the image of $\ulM\Omega^1_\mathcal{X}\to(\omega^*\Conn^1)_\mathcal{X}$.
\end{proof}

By our general construction from Definition \ref{omegaexc} and Lemma \ref{omega_exc_sheaf}, we obtain an excellent LS-cube invariant Nisnevich sheaf $\omega^{\exc}\Conn^1$ on $\ulMCor$, whose cohomology is readily representable in $\mDA^{\eff}(k)$ by Theorem \ref{thm:mapRG}.
We show that $\omega^{\exc}\Conn^1$ and $\ulM\Conn^1$ coincides on $\mSm$:

\begin{theorem} \label{thm:ConnComp}
    For any $\mathcal{X}=(X,D)\in \mSm$, we have
    \[
    \ulM\Conn^1(\mathcal{X}) = \omega^{\exc}\Conn^1(\mathcal{X})
    \]
    as subgroups of $\Conn^1(\mathcal{X}^\circ)$.
    In particular, the sheaf $\ulM\Conn^1|_{\mSm}$ on $\mSm$ is $(\CI\cup\BI)$-local and hence defines an object $\mathrm{m}\Conn^1$ of $\mDA^\eff(k)$.
\end{theorem}

\begin{remark}
    By the same argument, we can show $\ulM\Conn^1_{\int}(\mathcal{X}) = \omega^{\exc}\Conn^1_{\int}(\mathcal{X})$ for $\mathcal{X}\in \mSm$.
\end{remark}

\begin{proof}
Recall from Lemma \ref{lem:MConn1_is_Im_of_MOmega} that $\ulM\Conn^1_\mathcal{X}$ coincides with the image of the morphism
$$
    \ulM\Omega^1_\mathcal{X}\to
    (\omega^*\Conn^1)_\mathcal{X}.
$$
Consider the following commutative diagram of Nisnevich sheaves on $X$:
$$
    \xymatrix{
    \ulM\Omega^1_\mathcal{X}\ar@{^(->}[r]\ar@{->>}[d]&
    (\omega^*\Omega^1)_\mathcal{X}\ar[d]&
    (\omega^{\exc}\Omega^1)_\mathcal{X}\ar@{_(->}[l]\ar[d]\\
    \ulM\Conn^1_\mathcal{X}\ar@{^(->}[r]&
    (\omega^*\Conn^1)_\mathcal{X}&
    (\omega^{\exc}\Conn^1)_\mathcal{X}.\ar@{_(->}[l]
    }
$$
By Lemma \ref{lem:MOmega_rec}, we have $(\ulM\Omega^1)_\mathcal{X}\subset (\omega^{\exc}\Omega^1)_\mathcal{X}$.
It follows from the above commutative diagram that $\ulM \Conn^1_\mathcal{X} \subset (\omega^\exc\Conn^1)_\mathcal{X}$.

Let us prove the opposite inclusion.
It suffices to prove that for each $L\in \Phi$ and $r\in \Q_{\geq 0}$, 
\begin{align}\label{eq0;thm;MConn1_mod}
    \omega^\exc \Conn^1(\sO_L,\fm_L^{\lceil r \rceil})\subset  \Fil_r\Conn^1(L).
\end{align}
Since the both sides of \eqref{eq0;prop;MHone_mod} depend only on $\lceil r \rceil$, we may assume $r\in \Z$.
We use an existing result computing $\omega^{\CI}\Conn^1(\mathcal{O}_L,\mathfrak{m}_L^r)$.
Namely, we have
$$
    \omega^{\CI}\Conn^1(\mathcal{O}_L,\mathfrak{m}_L^r)
    =\Fil_r\Conn^1(L)
$$
by \cite[Theorem 6.11]{RS21}.
Since $\omega^\exc\Conn^1\subset \omega^{\CI}\Conn^1$, this proves \eqref{eq0;thm;MConn1_mod}.
\end{proof}

\section*{Competing interests}

The second author is employed by NTT, Inc.
The first and third authors declare none.

\printbibliography

\end{document}